\documentclass[12pt,a4paper,oneside]{article}

\usepackage[T1]{fontenc}
\usepackage[utf8]{inputenc}
\usepackage{lmodern}

\usepackage{amsmath,amssymb,theorem}

\usepackage[left=2.9cm,right=2.9cm,top=2cm,bottom=2cm,includefoot]{geometry}

\usepackage[yyyymmdd,hhmmss]{datetime}

\usepackage{enumerate}
\usepackage[nottoc]{tocbibind}

\usepackage{hyperref}

\usepackage{cite} 

\usepackage{epsfig}
\usepackage{epic}
\usepackage{eepic}
\usepackage{graphics}
\usepackage{graphicx}
\usepackage{color}
\usepackage{longtable}
\usepackage{subfigure} 

\usepackage{caption}
\newtheorem{theorem}{Theorem}
\newtheorem{corollary}[theorem]{Corollary}

{\theorembodyfont{\rmfamily}\newtheorem{example}[theorem]{Example}}
\newtheorem{proposition}[theorem]{Proposition}
{\theorembodyfont{\rmfamily}\newtheorem{remark}[theorem]{Remark}}
\newenvironment{proof}{\noindent\textit{Proof.\ }}{\hspace*{\fill}$\Box$\medskip}

\renewcommand{\geq}{\geqslant}
\renewcommand{\leq}{\leqslant}
\newcommand{\tsum}{\textstyle\sum\limits}
\newcommand{\tprod}{\textstyle\prod\limits}
\newcommand{\tint}{\textstyle\int\limits}

\DeclareFontFamily{U}{mathx}{\hyphenchar\font45}
\DeclareFontShape{U}{mathx}{m}{n}{
      <5> <6> <7> <8> <9> <10>
      <10.95> <12> <14.4> <17.28> <20.74> <24.88>
      mathx10
      }{}
\DeclareSymbolFont{mathx}{U}{mathx}{m}{n}
\DeclareMathSymbol{\bigtimes}{1}{mathx}{"91}

\newcommand{\vect}[1]{\underline{\boldsymbol{#1}}}

\usepackage{soul}

\makeatletter
\newcommand{\raisemath}[1]{\mathpalette{\raisem@th{#1}}}
\newcommand{\raisem@th}[3]{\raisebox{#1}{$#2#3$}}
\makeatother

\newcommand{\vectq}{\text{\setul{1.7pt}{.5pt}$\raisemath{0.8pt}{\text{\ul{$\displaystyle\boldsymbol{q}$}}}$}}
\newcommand{\indexvectq}{\text{\setul{1.35pt}{.45pt}$\raisemath{1pt}{\text{\ul{$\boldsymbol{q}$}}}$}}

\newcommand{\dchi}{\displaystyle\chi}
\newcommand{\vectgamma}{\text{\setul{1.3pt}{.4pt}\ul{$\boldsymbol{\gamma}$}}}
\newcommand{\indexvectgamma}{\text{\setul{0.9pt}{.4pt}\ul{$\boldsymbol{\gamma}$}}}

\DeclareMathOperator{\pop}{p}
\newcommand{\p}{\!\rule{0.4pt}{0pt}\pop}
\DeclareMathOperator{\lcm}{lcm}

\DeclareMathOperator{\trig}{trig}
\DeclareMathOperator{\vspan}{span}
\DeclareMathOperator{\tmodop}{mod}
\newcommand{\tmod}{\,\tmodop}

\numberwithin{equation}{section}
\numberwithin{theorem}{section}
\numberwithin{figure}{section}

\DeclareMathOperator{\LDop}{LD}
\newcommand{\LD}{\vect{\LDop}}
\DeclareMathOperator{\LCop}{LC}
\newcommand{\LC}{\vect{\LCop}}

\DeclareMathOperator{\Iop}{I}
\newcommand{\I}{\vect{\Iop}}

\DeclareMathOperator{\Jop}{J}
\newcommand{\J}{\vect{\Jop}}

\allowdisplaybreaks

\begin{document}

\setlength\abovedisplayskip      { 7pt }
\setlength\abovedisplayshortskip { 3pt } 
\setlength\belowdisplayskip      { 7pt }
\setlength\belowdisplayshortskip { 3pt }

\title{Multivariate polynomial interpolation on Lissajous-Chebyshev nodes}

\author{ \begin{tabular}{ll} Peter Dencker$^{\mathrm{1)}}$\setcounter{footnote}{-1}\footnote{$^{\mathrm{1)}}$Institut f\"ur Mathematik, 
Universit\"at zu L\"ubeck.}\quad & Wolfgang Erb$^{\mathrm{2)}}$\setcounter{footnote}{-1}\footnote{$^{\mathrm{2)}}$Department of Mathematics, University of Hawai`i at M\=anoa. \newline }
\setcounter{footnote}{-1}\footnote{The authors gratefully acknowledge the financial support of the German Research Foundation \newline (DFG, grant number ER 777/1-1)}\\ 
\texttt{\normalsize dencker@math.uni-luebeck.de\quad} &  \texttt{\normalsize erb@math.hawaii.edu}
\end{tabular}}

\author{\begin{tabular}{ll} Peter
Dencker$^{\mathrm{1)}}$\setcounter{footnote}{-1}\footnote{$^{\mathrm{1)}}$Institut
f\"ur Mathematik,
Universit\"at zu L\"ubeck.}\quad & Wolfgang
Erb$^{\mathrm{2),a)}}$\setcounter{footnote}{-1}\footnote{$^{\mathrm{2)}}$Department
of Mathematics, University of Hawai`i at M\=anoa. }
\setcounter{footnote}{-1}\footnote{$^{\mathrm{a)}}$The author gratefully
acknowledges the financial support of the German Research Foundation
\newline (DFG, grant number ER 777/1-1)}\\
\texttt{\normalsize dencker@math.uni-luebeck.de\quad} &
\texttt{\normalsize erb@math.hawaii.edu}
\end{tabular}}

\date{\ \\[2mm] \normalsize 2017/08/21}

\maketitle

\thispagestyle{empty}

\enlargethispage{2em}

\tableofcontents 

\bigskip

\begin{abstract}
In this article, we study multivariate polynomial interpolation and quadrature rules on non-tensor product node sets related to Lissajous curves and Chebyshev varieties. 
After classifying multivariate Lissajous curves and the interpolation nodes linked to these curves, we derive a discrete orthogonality structure on these node sets. 
Using this orthogonality structure, we obtain unique polynomial interpolation in appropriately defined spaces of multivariate Chebyshev polynomials. 
Our results generalize corresponding interpolation and quadrature results for the Chebyshev-Gauß-Lobatto points in dimension one and the Padua points in dimension two. 

\vspace{2em}
\end{abstract}

\section{Introduction and foundations}

\subsection{Introduction}

The research of this work has its origins in the investigation of polynomial interpolation schemes for data given on two-dimensional Lissajous curves 
\cite{BosDeMarchiVianelloXu2006,Erb2015,ErbKaethnerAhlborgBuzug2015,ErbKaethnerDenckerAhlborg2015}. 
In this article, we are interested in finding a multivariate generalization of this two-dimensional theory. For this purpose, we explicitly construct $\mathsf{d}$-variate non-tensor product node sets
that are related to the intersection points of Lissajous curves and to the singular points of Chebyshev varieties, and that
allow unique polynomial interpolation in particular spaces of $\mathsf{d}$-variate polynomials.

Studies on Lissajous curves go back to the early 19th century. Here, important work was done by the American mathematician N. Bowditch and the French physicist 
J.A. Lissajous \cite{Lissajous1857}. Another classical reference for plane Lissajous curves is the dissertation of W. Braun \cite{Braun1875} 
accomplished under the supervision of F. Klein. Recent studies of Lissajous curves in dimension three related to knot theory can be found in 
\cite{BogleHearstJonesStoilov1994,JonesPrzytycki1998,KoseleffPecker2011}. Lissajous curves are also closely related to Chebyshev varieties. In dimension two, some of these relations 
are elaborated in \cite{Fischer,KoseleffPecker2011,Merino2003}.

In the framework of bivariate polynomial interpolation, Lissajous curves appear first in the theory of the Padua points. These points are extensively studied
in \cite{BosDeMarchiVianelloXu2006,BosDeMarchiVianelloXu2007,CaliariDeMarchiVianello2005,CaliariDeMarchiVianello2008} and turn out to be node points generated by a 
particular Lissajous figure. Recently, this theory was extended to more general two-dimensional Lissajous curves and its respective node sets, see
\cite{Erb2015,ErbKaethnerAhlborgBuzug2015,ErbKaethnerDenckerAhlborg2015}. 
The investigated node points allow unique polynomial interpolation in properly defined spaces of bivariate polynomials. Moreover, they have a series of properties
that make them very interesting for computational purposes: the interpolating polynomial can be computed in an efficient
way by using fast Fourier methods, cf. \cite{CaliariDeMarchiSommarivaVianello2011,CaliariDeMarchiVianello2008,ErbKaethnerAhlborgBuzug2015,ErbKaethnerDenckerAhlborg2015}, the growth of the Lebesgue constant
is of logarithmic order, cf. \cite{BosDeMarchiVianelloXu2006,Erb2015}, and the points form a Chebyshev lattice of rank one, see \cite{CoolsPoppe2011}. A first approach for quadrature rules and hyperinterpolation in dimension three using Lissajous curves was recently developed in \cite{BosDeMarchiVianello2015}.

The goal of this work is to generalize the above mentioned two-dimensional interpolation theory to arbitrary finite dimensions. 
One challenging task thereby is to find appropriate node sets and the corresponding Lissajous curves that generate these points. Not all Lissajous curves are suitable for this purpose. For instance, in knot theory three dimensional Lissajous curves are investigated that have no 
self-intersection points at all, see \cite{BogleHearstJonesStoilov1994,JonesPrzytycki1998,KoseleffPecker2011}. For an integer valued vector $\vect{n}$ with relatively prime entries, we introduce in this work two types of node sets 
$\LC^{(\vect{n})}$ and $\LC^{(2\vect{n})}_{\vect{\kappa}}$. Due to their intimate link to $\mathsf{d}$-variate Lissajous curves and Chebyshev varieties we denote their elements as Lissajous-Chebyshev nodes. 
For properly defined spaces of multivariate Chebyshev polynomials linked to polygonal index sets 
$\vect{\Gamma}^{(\vect{n})}$ and $\vect{\Gamma}^{(2\vect{n})}_{\vect{\kappa}}$, we are able to show unique polynomial interpolation on these node points. 

The uniqueness of polynomial interpolation was proven for the two-dimensional case in \cite{Erb2015,ErbKaethnerAhlborgBuzug2015} by reducing the bivariate interpolation problem to
an interpolation problem for trigonometric polynomials on the Lissajous curve. In this work, we use a purely discrete approach for the proof of the uniqueness
in which the notion of Lissajous curve is not needed. However, the characterization of the interpolation nodes using 
multivariate Lissajous figures and Chebyshev varieties was the initial motivation and gives an interesting geometric interpretation
of the theory. Therefore, a large part of this work is also concerned with the characterization of these curves and varieties. 
Some of the geometric relations between interpolation nodes and the curves can be seen in Figures \ref{v2:fig:lissajous1}, \ref{v2:fig:lissajous2} and \ref{v2:fig:lissajous3}. 

There exist other non-tensor product node sets that are related to the discussed point sets. In particular bivariate point sets introduced by Morrow, Patterson~\cite{MorrowPatterson1978} and Xu~\cite{Xu1996}
as well as some extensions of these node sets~\cite{DeMarchiVianelloXu2007,Harris2010,Harris2013} are to be mentioned. 
For an overview and further references on general interpolation techniques with multivariate polynomials, we refer to the survey articles \cite{GascaSauer2000,GascaSauer2000b}. 

After introducing the necessary notation, we start in Section \ref{v2:1509092255} with a rough classification of $\mathsf{d}$-variate Lissajous curves. We distinguish 
between degenerate and non-degenerate curves and derive for both cases simplifications for particular values of the parameters as well as some
geometric properties of their self-intersection points. Particular degenerate Lissajous curves are relevant in this work. We will explicitly specify  
the number and the type of their self-intersection points in Theorem \ref{v2:1503151619}. 

In Section \ref{v2:15009022044}, we introduce the set $\LC^{(\vect{n})}$ and show in Theorem \ref{v2:1509011746} that it allows unique polynomial interpolation 
in a space $\Pi_{\mathsf{d}}^{(\vect{n})}$ of $\mathsf{d}$-variate polynomials. 
Essential for the proof is a discrete orthogonality structure proven in Theorem \ref{v2:1507091911}, and the fact that $\mathrm{dim} \, \Pi_{\mathsf{d}}^{(\vect{n})} = 
\# \LC^{(\vect{n})}$, see Proposition \ref{v2:1507151430}. Furthermore, we will give in Theorem \ref{v2:1509092212} and \ref{v2:1508251450} two characterizations of $\LC^{(\vect{n})}$ using
degenerate Lissajous curves as well as Chebyshev varieties.

In Section \ref{v2:201510121546}, we consider a second type $\LC^{(2 \vect{n})}_{\vect{\kappa}}$ of point sets. Similar as for the sets $\LC^{(\vect{n})}$, we show in Theorem \ref{v2:201510121554} that there 
exists a unique polynomial in $\Pi_{\mathsf{d},\vect{\kappa}}^{(2\vect{n})}$ interpolating function values on $\LC^{(2 \vect{n})}_{\vect{\kappa}}$. In comparison to 
$\LC^{(\vect{n})}$, the sets $\LC^{(2 \vect{n})}_{\vect{\kappa}}$ are symmetric with respect to reflections at the coordinate axis. Further, 
the points in $\LC^{(2 \vect{n})}_{\vect{\kappa}}$ can be linked to a Chebyshev variety and to a family of Lissajous curves, see Theorem \ref{v2:201510121616}. 

For both node sets, $\LC^{(\vect{n})}$ 
and $\LC^{(2 \vect{n})}_{\vect{\kappa}}$, we will further derive a simple scheme for the efficient computation of the interpolating polynomial.

\subsection{General notation}

The number of elements of a finite set $\mathsf{X}$ is denoted by $\#\mathsf{X}$. For $\mathsf{x},\mathsf{y}\in\mathsf{X}$, 
we use the Kronecker delta symbol: $\delta_{\mathsf{x},\mathsf{y}}=1$ if $\mathsf{y}=\mathsf{x}$ and $\delta_{\mathsf{x},\mathsf{y}}=0$ otherwise.
Further, to each $\mathsf{x}\in\mathsf{X}$ we associate a Dirac delta function $\delta_{\mathsf{x}}$ on $\mathsf{X}$ 
defined by $\delta_{\mathsf{x}}(\mathsf{y})=\delta_{\mathsf{x},\mathsf{y}}$, $\mathsf{y}\in\mathsf{X}$.
We denote by $\mathcal{L}(\mathsf{X})$ the vector space of all complex-valued functions defined on $\mathsf{X}$. 
Then, the delta functions  $\delta_{\mathsf{x}}$, $\mathsf{x}\in\mathsf{X}$, form a basis of the vector space $\mathcal{L}(\mathsf{X})$. If $\nu$ is a measure defined on the power set $\mathcal{P}(\mathsf{X})$ of a finite non-empty set $\mathsf{X}$ and $\nu(\{\mathsf{x}\})>0$ for all $\mathsf{x}\in\mathsf{X}$, an inner product $\langle \,\cdot,\cdot\,\rangle_{\nu}$ for 
the vector space $\mathcal{L}(\mathsf{X})$ is defined by
\[\langle f_1,f_2\rangle_{\nu}=\tint f_1\overline{f}_2\,\mathrm{d}\nu.\]
 The norm corresponding to  $\langle \,\cdot,\cdot\,\rangle_{\nu}$  is denoted by  $\|\cdot\|_{\nu}$.

\medskip 
For a fixed $\mathsf{d}\in\mathbb{N}$, the elements of the $\mathsf{d}$-dimensional vector space  $\mathbb{R}^{\mathsf{d}}$ over $\mathbb{R}$ are written as $\vect{x}=(x_1,\ldots,x_{\mathsf{d}})$. 
We use the abbreviations $\vect{0}$ and $\vect{1}$ for the $\mathsf{d}$-tuples for which all components are $0$ or $1$, respectively.

\medskip

The  least common multiple in $\mathbb{N}$ of the elements of a finite set $M\subseteq \mathbb{N}$ is denoted by $\lcm M$, the greatest common divisor of the elements of a set $M\subseteq \mathbb{Z}$, $M\supsetneqq \{0\}$, is denoted by $\gcd M$. 

  The following well-known \textit{Chinese remainder theorem} will be used repeatedly.\\
Let $\vect{k}\in \mathbb{N}^{\mathsf{d}}$ and $\vect{a}\in\mathbb{Z}^{\mathsf{d}}$. If
\begin{equation}\label{v2:1503211222}
\forall\,\mathsf{i},\mathsf{j}\in\{1,\ldots,\mathsf{d}\}:\quad a_{\mathsf{i}}\equiv a_{\mathsf{j}}\mod \gcd\{k_{\mathsf{i}},k_{\mathsf{j}}\},
\end{equation}
then there exists a unique $l\in\{\,0,\ldots,\lcm\{k_1,\ldots,k_{\mathsf{d}}\}-1\,\}$ solving the following system of simultaneous congruences
\begin{equation}\label{v2:1505032140}
\forall\,\mathsf{i}\in\{1,\ldots,\mathsf{d}\}:\quad l=a_{\mathsf{i}} \mod k_{\mathsf{i}}. 
\end{equation}
Note that \eqref{v2:1503211222} is also a necessary condition for the existence of $l\in\mathbb{Z}$ 
satisfying  \eqref{v2:1505032140}. 

\medskip

For $\vect{k}\in\mathbb{N}^{\mathsf{d}}$, we set 
\begin{equation} \label{v2:201609011131}
\p[\vect{k}]=\tprod_{\mathsf{j}=1}^{\mathsf{d}} k_{\mathsf{j}},\qquad \p_{\mathsf{i}}[\vect{k}]=
\tprod_{\substack{\mathsf{j}=1\\\mathsf{j}\neq \mathsf{i}}}^{\mathsf{d}} k_{\mathsf{j}}=\p[\vect{k}]
/k_\mathsf{i},\quad \mathsf{i}\in \{1,\ldots,\mathsf{d}\}.
\end{equation}

{\it Throughout this paper we will assume that the integers
\begin{equation}\label{v2:1509151442} 
\underline{n_1,\ldots,n_{\mathsf{d}}\in \mathbb{N}\;  \textit{are pairwise relatively prime}}.
\end{equation}
}
This work considers two slightly different constructions, one naturally linked to \text{$\vect{m}=\vect{n}$} and one naturally linked to $\vect{m}=2\vect{n}$. 
In these cases, we further denote
\[t^{(\vect{n})}_{l}=l\pi/\p[\vect{n}],\quad t^{(2\vect{n})}_{l}=l\pi/(2\p[\vect{n}]),\qquad l\in\mathbb{Z}.\]
For   $m\in\mathbb{N}$, $i\in\mathbb{N}_0$, as well as  $\vect{m}\in\mathbb{N}^{\mathsf{d}}$, $\vect{i}\in\mathbb{N}_0^{\mathsf{d}}$,
we use the notation
\begin{equation}
\vect{z}^{(\vect{m})}_{\vect{i}}=\left(z^{(m_1)}_{i_1},\ldots,z^{(m_{\mathsf{d}})}_{i_{\mathsf{d}}}\right),\quad  z^{(m)}_{i}=\cos\left(i\pi/m\right), \label{v2:201510121326}
\end{equation}
for the so-called \textit{Chebyshev-Gauß-Lobatto} points.


\subsection{Lissajous curves}\label{v2:1509092255} 
For $\vectq\in\mathbb{N}^{\mathsf{d}}$, $\vect{\alpha}\in\mathbb{R}^{\mathsf{d}}$ and $\vect{u}\in\{-1,1\}^{\mathsf{d}}$, we define the \textit{Lissajous curve}  
$\vect{\mathfrak{l}}^{(\indexvectq)}_{\vect{\alpha},\vect{u}}$ by 
\begin{equation} \label{v2:201509161237}
\vect{\mathfrak{l}}^{(\indexvectq)}_{\vect{\alpha},\vect{u}}(t)=\left( u_1 \cos \left( q_1t- \alpha_1\right), \cdots, 
u_{\mathsf{d}} \cos \left( q_{\mathsf{d}}t- \alpha_{\mathsf{d}}\right) \right), \quad t \in \mathbb{R}.
\end{equation} 

The curve  $\vect{\mathfrak{l}}^{(\indexvectq)}_{\vect{\alpha},\vect{u}}$  is called \textit{degenerate} if 
\begin{equation}\label{v2:1509091713}
\text{there exist $t'\in\mathbb{R}$ and $\vect{u}'\in\{-1,1\}^{\mathsf{d}}$ such that $\vect{\mathfrak{l}}^{(\indexvectq)}_{\vect{\alpha},\vect{u}}(\,\cdot\,-t') = \vect{\mathfrak{l}}^{(\indexvectq)}_{\vect{0},\vect{u}'}$},
\end{equation}
and \textit{non-degenerate} otherwise.
This terminology  is motivated by the following Theorem \ref{v2:1503211414} that gives a rough classification of Lissajous curves and their geometry.

Obviously, the  Lissajous curve $\vect{\mathfrak{l}}^{(\indexvectq)}_{\vect{\alpha},\vect{u}}$ is a periodic function with period $2\pi$.
By scaling the parameter $t$, the general definition \eqref{v2:201509161237} can easily be reduced to the case $\gcd\{q_1,\ldots,q_{\mathsf{d}}\}=1$ and we restrict ourselves to this case. In particular, the following theorem implies that in this case $2\pi$ is the fundamental period of $\vect{\mathfrak{l}}^{(\indexvectq)}_{\vect{\alpha},\vect{u}}$.
\begin{theorem}\label{v2:1503211414}  Let $\vectq=(q_1,\ldots,q_{\mathsf{d}})\in\mathbb{N}^{\mathsf{d}}$  satisfy $\gcd\{q_1,\ldots,q_{\mathsf{d}}\}=1$.\\
Let $\mathfrak{U}_{\vect{\alpha},\vect{u}}^{(\indexvectq)}(\vect{x})$ be the set of all $t\in [0,2\pi)$ satisfying $\vect{\mathfrak{l}}^{(\indexvectq)}_{\vect{\alpha},\vect{u}}(t)=\vect{x}$.

a) If $\vect{\mathfrak{l}}^{(\indexvectq)}_{\vect{\alpha},\vect{u}}$ is non-degenerate,
then $\#\mathfrak{U}_{\vect{\alpha},\vect{u}}^{(\indexvectq)}(\vect{x})>1$  for only finitely many  $\vect{x}$. \\ Furthermore, for each $\vect{x}$ we have $\#\mathfrak{U}_{\vect{\alpha},\vect{u}}^{(\indexvectq)}(\vect{x})\leq 2^{\mathsf{d}-1}$.

b) If $\vect{\mathfrak{l}}^{(\indexvectq)}_{\vect{\alpha},\vect{u}}$ is degenerate, then, up to precisely two exceptions, $\#\mathfrak{U}_{\vect{\alpha},\vect{u}}^{(\indexvectq)}(\vect{x})\geq 2$ for  $\vect{x}\in\vect{\mathfrak{l}}^{(\indexvectq)}_{\vect{\alpha},\vect{u}}([0,2\pi))$. There are only finitely many 
$\vect{x}$ with $\#\mathfrak{U}_{\vect{\alpha},\vect{u}}^{(\indexvectq)}(\vect{x})>2$ and in this case $\#\mathfrak{U}_{\vect{\alpha},\vect{u}}^{(\indexvectq)}(\vect{x})$ is even.
 Furthermore, for every $\vect{x}$ we have
$\#\mathfrak{U}_{\vect{\alpha},\vect{u}}^{(\indexvectq)}(\vect{x})\leq 2^{\mathsf{d}}$.
\end{theorem} 

 If the Lissajous curve is non-degenerate, we can find, up to finitely many exceptions, only one point in time $t\in [0,2\pi)$ corresponding to a point on the curve. On the other hand, degenerate curves are doubly
traversed: in the standard form $\vect{\mathfrak{l}}^{(\indexvectq)}_{\vect{0},\vect{u}}$ we have $\vect{\mathfrak{l}}^{(\indexvectq)}_{\vect{0},\vect{u}}=\vect{\mathfrak{l}}^{(\indexvectq)}_{\vect{0},\vect{u}}(2\pi\,-\,\cdot\,)$, and thus the value $t = \pi$ can be interpreted as a reverse point in time   after which the curve $\vect{\mathfrak{l}}^{(\indexvectq)}_{\vect{0},\vect{u}}$ is traversed in the opposite direction. 

\medskip

\begin{proof} If $t-s\in 2\pi\mathbb{Z}$, we write $s\eqsim t$. We use the abbreviations  $q=\lcm\{q_1,\ldots,q_{\mathsf{d}}\}$, $a_{\mathsf{i}}=q/q_{\mathsf{i}}$, and $\vect{\mathfrak{l}}=\vect{\mathfrak{l}}^{(\indexvectq)}_{\vect{\alpha},\vect{u}}$.  Defining $\mathcal{A}(t)=\{s\in [0,2\pi)\,|\,\vect{\mathfrak{l}}(s)=\vect{\mathfrak{l}}(t)\,\}$, we have $s\in \mathcal{A}(t)$ if and only if there is a  $\vect{v}\in \{-1,1\}^{\mathsf{d}}$ such that
 \begin{equation}\label{v2:1503171442}
\forall\,\mathsf{i}\in\{1,\ldots,\mathsf{d}\}:\quad q_{\mathsf{i}}(s-v_{\mathsf{i}}t)-(1-v_{\mathsf{i}})\alpha_{\mathsf{i}}\eqsim 0.
\end{equation} 
By B\'ezout's lemma, we can find $A_{1,2},A_{2,1}\in\mathbb{Z}$ such that $A_{2,1}q_{1}+A_{1,2}q_{2}=\gcd\{q_{1},q_{2}\}$.
 If $s\in \mathcal{A}(t)$, we obtain
\[
\gcd\{q_{1},q_{2}\}(s- A_{2,1}q_{1}v_{1}t- A_{1,2}q_{2}v_{2}t)
\eqsim A_{2,1}(1-v_{1})\alpha_{1}+A_{1,2}(1-v_{2})\alpha_{2}.
\] 
By a simple induction argument and the assumption $\gcd\{q_1,\ldots,q_{\mathsf{d}}\}=1$, we can therefore find $B_{\mathsf{i}},C_{\mathsf{i}}\in\mathbb{Z}$ (only depending on $\vectq$) such that 
\begin{equation}\label{v2:1508222356}
\forall\,s,t\in [0,2\pi):\quad\eqref{v2:1503171442}\ \Longrightarrow\  s \eqsim 
(\,{\textstyle\sum_{\mathsf{i}} B_{\mathsf{i}} v_{\mathsf{i}}})t+{\textstyle\sum_{\mathsf{i}}C_{\mathsf{i}}(1-v_{\mathsf{i}})\alpha_{i}}.
\end{equation}
\vspace{-1em}\\
We fix $t\in [0,2\pi)$ and consider the set
\[\vect{\mathfrak{A}}(t) = \{\vect{v}\in \{-1,1\}^{\mathsf{d}}\,|\, \exists\,s\in [0,2\pi)\ \text{with}\  
\eqref{v2:1503171442}\}.\]
Let $\vect{v}\in \vect{\mathfrak{A}}(t)$. Then, by the definition of $\vect{\mathfrak{A}}(t)$ there exists $s\in [0,2\pi)$ with \eqref{v2:1503171442}. If \eqref{v2:1503171442} is satisfied for $\vect{v}$ and $s'\in [0,2\pi)$ instead of $s$, then \eqref{v2:1508222356} implies $s'=s$. Therefore, a  function $\mathfrak{a}(t;\,\cdot\,):\,\mathfrak{A}(t)\to [0,2\pi)$ is well-defined  on $\mathfrak{A}(t)$ by the condition that  $\mathfrak{a}(t;\vect{v})=s\in [0,2\pi)$ satisfies \eqref{v2:1503171442} for $\vect{v}\in\vect{\mathfrak{A}}(t)$. By the definition of $\mathcal{A}(t)$, the possible values of the defined function  are precisely the elements of $\mathcal{A}(t)$. Thus, $\mathfrak{a}(t;\,\cdot\,)$ is a function from $\mathfrak{A}(t)$ onto $\mathcal{A}(t)$. We always have  $\vect{1}\in \vect{\mathfrak{A}}(t)$ and $\mathfrak{a}(t;\vect{1})=t$.

\medskip

If  $\vect{v},-\vect{v}\in \vect{\mathfrak{A}}(t)$, then there exist $s,r\in [0,2\pi)$  satisfying \eqref{v2:1503171442} and
 \begin{equation}\label{v2:1503242042}
\forall\,\mathsf{i}\in\{1,\ldots,\mathsf{d}\}:\quad q_{\mathsf{i}}(r-(-v_{\mathsf{i}})t)-(1-(-v_{\mathsf{i}}))\alpha_{\mathsf{i}}\eqsim 0.
\end{equation} 
Adding \eqref{v2:1503171442} and \eqref{v2:1503242042} yields $2\alpha_{\mathsf{i}}\eqsim q_{\mathsf{i}}(s+r)$ for all $\mathsf{i}$. Let $t'=(s+r)/2$. Now, $2\alpha_{\mathsf{i}}\eqsim 2q_{\mathsf{i}}t'$, $\mathsf{i}\in\{1,\ldots,\mathsf{d}\}$, signifies the existence of $k_{\mathsf{i}}\in\mathbb{Z}$ with $\alpha_{\mathsf{i}}=q_{\mathsf{i}}t'+k_{\mathsf{i}}\pi$ and, therefore, implies
\eqref{v2:1509091713} with $u'_{\mathsf{i}}=(-1)^{k_{\mathsf{i}}}u_{\mathsf{i}}$, $\mathsf{i}\in\{1,\ldots,\mathsf{d}\}$.  We have shown:
 \begin{equation}\label{v2:1503252146}
\text{$\exists\,t\in[0,2\pi),\,\exists\,\vect{v}\in \vect{\mathfrak{A}}(t)$ with $-\vect{v}\in \vect{\mathfrak{A}}(t)$ $\Longrightarrow$ $\vect{\mathfrak{l}}$ is degenerate.}
\end{equation}
Let $\vect{x}=\vect{\mathfrak{l}}(t)$. Since $\#\vect{\mathfrak{A}}(t)\leq \#\{-1,1\}^{\mathsf{d}}$, we have $\#\mathfrak{U}_{\vect{\alpha},\vect{u}}^{(\indexvectq)}(\vect{x})=\#\mathcal{A}(t)\leq 2^{\mathsf{d}}$.
If $\vect{\mathfrak{l}}$ is  non-degenerate, then  \eqref{v2:1503252146} implies $\#\vect{\mathfrak{A}}(t)\leq 2^{\mathsf{d}-1}$,
and thus $\#\mathfrak{U}_{\vect{\alpha},\vect{u}}^{(\indexvectq)}(\vect{x})=\#\mathcal{A}(t)\leq 2^{\mathsf{d}-1}$.

\medskip

\noindent Next, we consider the following assumption:
\begin{equation}\label{v2:1508230019} 
\text{$\vect{\mathfrak{l}}$ is non-degenerate and there are infinitely many $\vect{x}$ with $\#\mathfrak{U}_{\vect{\alpha},\vect{u}}^{(\indexvectq)}(\vect{x})>1$.}
\end{equation}
Then, we can find sequences  $(r_m)$, $(s_m)$ with values in $[0,2\pi)$, such that  $\vect{\mathfrak{l}}(r_m)=\vect{\mathfrak{l}}(s_m)$ as well as 
 $r_m\neq s_m$ and $s_m\neq s_l$ if $m\neq l$ for all $m,l \in \mathbb{N}$. Further, there exist $\vect{v}_m \in \{-1,1\}^{\mathsf{d}}$ such that  $r_m=\mathfrak{a}(s_m;\vect{v}_m)$. Since $[0,2\pi]$ is compact and $\{-1,1\}^{\mathsf{d}}$ is finite, we can assume without loss of generality that the sequence $(s_m)$ is convergent and that the sequence 
 $(\vect{v}_m)$ is constant, i.e. $\vect{v}_m=\vect{v}$ for all $m \in \mathbb{N}$. Since $r_m\neq s_m$, we have $\vect{v}\neq \vect{1}$. Furthermore, since $\vect{1}\in\vect{\mathfrak{A}}(s_m)$, statement \eqref{v2:1503252146} and assumption \eqref{v2:1508230019} imply $\vect{v}\neq -\vect{1}$. Therefore, we can find  $\mathsf{i},\mathsf{j}$ such that $v_{\mathsf{i}}=-1$, $v_{\mathsf{j}}=1$, and \eqref{v2:1503171442} implies
that for all $m$ there is a $k_m\in\mathbb{Z}$ such that
\[2qs_m-2a_{\mathsf{i}}\alpha_{\mathsf{i}}=a_{\mathsf{i}}(q_{\mathsf{i}}(r_m+s_m)-2\alpha_{\mathsf{i}})-a_{\mathsf{j}}q_{\mathsf{j}}(r_m-s_m) =2\pi k_m.\]
Since the left hand side of this equation converges, the sequence $(k_m)$ converges as well.
 The integer-valued sequence $(k_m)$ is constant for sufficiently large $m$ and so is $(s_m)$.  
Therefore, assumption \eqref{v2:1508230019} implies a contradiction. 

\medskip

The statement a) is now completely proven. By definition \eqref{v2:1509091713}, it remains to show the statement  b) for $\vect{\alpha}=\vect{0}$. In this case, we have $\vect{\mathfrak{l}}(t) = \vect{\mathfrak{l}}(2\pi-t)$.
Therefore, $\# \mathcal{A}(t) \geq 2$ for all $t \in (0,2 \pi) \setminus \{\pi\}$ and  $\# \mathcal{A}(t)$ is even for all $t \in (0,2 \pi) \setminus \{\pi\}$. 
If $s\in  \mathcal{A}(0)$, then \eqref{v2:1508222356} implies $s=0$, therefore  $\# \mathcal{A}(0)=1$. 
If  $s\in \mathcal{A}(\pi)$, then \eqref{v2:1508222356} implies $s=0$ or $ s=\pi$. Thus, since $\# \mathcal{A}(0)=1$, we also have  $\#\mathcal{A}(\pi)=1$.  

Finally, we  consider  the case  $\#\mathcal{A}(t)\geq 3$. Then, $\#\vect{\mathfrak{A}}(t)\geq 3$ and there exists $\vect{v}\notin\{-\vect{1},\vect{1}\}$
such that $s=\mathfrak{a}(t;\vect{v})\in \mathcal{A}(t)$. There are  $\mathsf{i},\mathsf{j}$ such that $v_{\mathsf{i}}=-1$, $v_{\mathsf{j}}=1$. Now, applying  \eqref{v2:1503171442} with  $\vect{\alpha}=\vect{0}$ and using the notation $q, q_{\mathsf{i}}, a_{\mathsf{i}}$  introduced at the beginning of the proof, we obtain
\[2qt=a_{\mathsf{i}}q_{\mathsf{i}}(s+t)-a_{\mathsf{j}} q_{\mathsf{j}}(s-t)\eqsim 0.\]
Therefore, we can find $l\in \{1,\ldots,2q-1\}\setminus\{q\}$ such that $t=l\pi/q$.
In particular, there are only finitely many $t\in [0,2\pi)$ such that $\#\mathcal{A}(t)\geq 3$.
\end{proof}

\medskip

For our purposes, the following  Lissajous trajectories turn out to be appropriate.  
For $\vect{n}\in\mathbb{N}^{\mathsf{d}}$  satisfying the general assumption \eqref{v2:1509151442} as well as 
for $\epsilon\in\{1,2\}$, $\vect{\kappa}\in\mathbb{Z}^{\mathsf{d}}$ and $\vect{u}\in\{-1,1\}^{\mathsf{d}}$, we consider the Lissajous curves  $\vect{\ell}^{(\epsilon\vect{n})}_{\vect{\kappa},\vect{u}}\,:\,\mathbb{R}\to [-1,1]^{\mathsf{d}}$ defined by 
\[\vect{\ell}^{(\epsilon\vect{n})}_{\vect{\kappa},\vect{u}}(t)=\left( u_1 \cos \left(  \;\!\p_1[\vect{n}]  \cdot t- \dfrac{\kappa_1\pi}{\epsilon n_1}\right), \cdots, 
u_{\mathsf{d}} \cos \left( \;\!\p_{\mathsf{d}}[\vect{n}] \cdot t- \dfrac{\kappa_{\mathsf{d}}\pi}{\epsilon n_{\mathsf{d}}}\right) \right), \quad t \in \mathbb{R},\]
where $\p_{\mathsf{i}}[\vect{n}]$, $\mathsf{i} \in \{1, \ldots, \mathsf{d}\}$, denote the 
products given in \eqref{v2:201609011131}. The parameter $\vect{u}$ can be expressed also in terms of different values of the parameter $\vect{\kappa}$. However, this redundancy in the definition will turn out to be useful.  For example, in this way it is possible to restrict the considerations to particular values of $\vect{\kappa} \in \mathbb{Z}^{\mathsf{d}}$. 
Statement b) of the following Proposition gives an illustrative description of the considered class of Lissajous curves.  Furthermore, note that always $\vect{\ell}^{(\vect{n})}_{\vect{\kappa},\vect{u}}=\vect{\ell}^{(2\vect{n})}_{2\vect{\kappa},\vect{u}}$. 

\begin{proposition} \label{v2:201510131412}  Let $\vect{n}\in\mathbb{N}^{\mathsf{d}}$  satisfy \eqref{v2:1509151442}. Further, let $\vect{\kappa}\in\mathbb{Z}^{\mathsf{d}}$ and $\vect{u}\in\{-1,1\}^{\mathsf{d}}$.

\medskip

a)  There exist $\vect{\kappa}'=(\kappa'_1,\ldots,\kappa'_{\mathsf{d}})$ with  $\kappa'_1=0$ and $\kappa'_{\mathsf{i}} \in \{0,1\}$, $\mathsf{i} \in \{2, \ldots,\mathsf{d}\}$, as well as  $r'\in\mathbb{Z}$ and $\vect{u}'\in\{-1,1\}^{\mathsf{d}}$, such that
\begin{equation}\label{v2:1509151217}
\vect{\ell}^{(2\vect{n})}_{\vect{\kappa},\vect{u}}(\,\cdot\,-t_{r'}^{(2 \vect{n})})=\vect{\ell}^{(2\vect{n})}_{\vect{\kappa}',\vect{u}'}.
\end{equation}

b) There exist
$\vect{\delta} \in\{0,1\}^{\mathsf{d}}$, $r' \in \mathbb{Z}$  and $\vect{u}'\in\{-1,1\}^{\mathsf{d}}$ such that \eqref{v2:1509151217} is satisfied with $\kappa'_{\mathsf{i}}= \delta_{\mathsf{i}} n_{\mathsf{i}}$,
$\mathsf{i}\in\{1,\ldots,\mathsf{d}\}$. In this way, the curve can be written as 
\[\vect{\ell}^{(2\vect{n})}_{\vect{\kappa},\vect{u}}(t -t_{r'}^{(2 \vect{n})} ) = \left(u_1'\trig_1(\p_1[\vect{n}]\cdot t), \cdots,
u_{\mathsf{d}}'\trig_{\mathsf{d}}(\p_{\mathsf{d}}[\vect{n}]\cdot t)\right),\quad t\in\mathbb{R},\]
with some $u_{\mathsf{i}}'\in\{-1,1\}$ and some 
$\trig_{\,\mathsf{i}}\in\{\sin,\cos\}$ for all $\mathsf{i}\in\{1,\ldots,\mathsf{d}\}$.
\end{proposition}

\begin{proof} For $t'' = -\kappa_1\pi/(2\p[\vect{n}])$, we get $\vect{\ell}^{(2\vect{n})}_{\vect{\kappa},\vect{u}}(\,\cdot\,-t'')=\vect{\ell}^{(2\vect{n})}_{\vect{\kappa}'',\vect{u}}$
with $\kappa''_{\mathsf{i}}=-\kappa_{\mathsf{i}}+\kappa_1$, $\mathsf{i}\in\{1,\ldots,\mathsf{d}\}$, in particular $\kappa''_1=0$.
 Let $b_{\mathsf{i}}\in\mathbb{Z}$ such that $0\leq b_{\mathsf{i}}+\kappa''_{\mathsf{i}}/2<1$.
With $a_{\mathsf{i}} = 2 b_{\mathsf{i}}$ and $k_{\mathsf{i}} = 2 n_{\mathsf{i}}$ condition \eqref{v2:1503211222} is satisfied. Thus, by the Chinese remainder theorem there exists $r \in \mathbb{Z}$ with $r\equiv  2 b_{\mathsf{i}} \tmod 2 n_{\mathsf{i}}$ for all $\mathsf{i}\in\{1,\ldots,\mathsf{d}\}$. 
Therefore, we can find $\rho_{\mathsf{i}} \in \{0,1\}$ such that
\[\forall\,\mathsf{i}\in\{1,\ldots,\mathsf{d}\}:\quad r\equiv  2 b_{\mathsf{i}} + 2 \rho_{\mathsf{i}}n_{\mathsf{i}}\mod 4n_{\mathsf{i}}.\]
We have  \eqref{v2:1509151217} 
with  $t'= t_{r'}^{(2 \vect{n})}$,  $r' = r - \kappa_1$, and 
$\kappa'_{\mathsf{i}}=2b_{\mathsf{i}}+\kappa''_{\mathsf{i}}$, $u'_{\mathsf{i}}=(-1)^{\rho_{\mathsf{i}}}u_{\mathsf{i}}$.

Now, we consider  the statement b). By the Chinese remainder theorem there exists  $r' \in \mathbb{Z}$ with $r'\equiv  -\kappa_{\mathsf{i}} \tmod
n_{\mathsf{i}}$ for all $\mathsf{i}$. Hence, there exists 
$\vect{\delta}\in\{0,1\}^{\mathsf{d}}$ such that $r'\equiv
\delta_{\mathsf{i}}n_{\mathsf{i}} -\kappa_{\mathsf{i}} \tmod
2n_{\mathsf{i}}$ for all $\mathsf{i}$. Therefore, we can find $\rho_{\mathsf{i}} \in \{0,1\}$ such that $r'\equiv
\delta_{\mathsf{i}}n_{\mathsf{i}} -\kappa_{\mathsf{i}}+ 2 \rho_{\mathsf{i}}n_{\mathsf{i}}\tmod
4n_{\mathsf{i}}$ for all $\mathsf{i}$.
 We have  \eqref{v2:1509151217} 
with  
$\kappa'_{\mathsf{i}}=\delta_{\mathsf{i}} n_{\mathsf{i}}$, $u'_{\mathsf{i}}=(-1)^{\rho_{\mathsf{i}}}u_{\mathsf{i}}$.
\end{proof}

\begin{corollary} \label{v2:1511121614}  Let $\vect{n}\in\mathbb{N}^{\mathsf{d}}$  satisfy \eqref{v2:1509151442}. Further, let $\vect{\kappa}\in\mathbb{Z}^{\mathsf{d}}$ and $\vect{u}\in\{-1,1\}^{\mathsf{d}}$.

 a) The Lissajous curve $\vect{\ell}^{(2\vect{n})}_{\vect{\kappa},\vect{u}}$  is degenerate if and only if
\begin{equation}\label{v2:1509092235}
\kappa_{\mathsf{i}} \equiv \kappa_{\mathsf{j}} \mod 2 \qquad \text{for all\quad $\mathsf{i},\mathsf{j}\in\{1,\ldots,\mathsf{d}\}$}.
\end{equation}
In this case, there exist $t' \in \mathbb{R}$ and $\vect{u}' \in \{-1,1\}^{\mathsf{d}}$ such that
\[ \vect{\ell}^{(\epsilon\vect{n})}_{\vect{\kappa},\vect{u}}(t-t') = \vect{\ell}^{(\vect{n})}_{\vect{0},\vect{u}'}(t) = \left( u_1' \cos \left( \p_{1}[\vect{n}] \cdot t \right), \cdots, 
u_{\mathsf{d}}' \cos \left( \p_{\mathsf{d}}[\vect{n}] \cdot t \right) \right),\quad t \in \mathbb{R}. \]
b) The curve 
 $\vect{\ell}^{(\vect{n})}_{\vect{\kappa},\vect{u}}=\vect{\ell}^{(2\vect{n})}_{2\vect{\kappa},\vect{u}}$ is degenerate.
\end{corollary}
\begin{proof} Considering  $\vect{\kappa}'$ given in the proof of Proposition  \ref{v2:201510131412} we see: the condition \eqref{v2:1509092235} holds if and only if \eqref{v2:1509092235} holds with $\vect{\kappa}$ replaced by the parameter $\vect{\kappa}'$. Now, Proposition \ref{v2:201510131412} implies the assertion a). Statement b) follows immediately.
\end{proof}

For the special degenerate Lissajous curves in the standard form $\vect{\ell}^{(\vect{n})}_{\vect{0},\vect{u}}$ we can concretize statement b) of Theorem~\ref{v2:1503211414}. To this end, we consider the sets  
\begin{equation} \label{v2:1509171731}
 \LD^{(\vect{n})}_{\vect{u}} = \left\{\,\vect{\ell}^{(\vect{n})}_{\vect{0},\vect{u}}(t^{(\vect{n})}_{l})\,|\,  l\in \{0,\ldots,\p[\vect{n}]\} \,\right\}.
\end{equation}
Further, for $\mathsf{M}\subseteq \{1,\ldots,\mathsf{d}\}$, we consider the subsets $\vect{F}^{\mathsf{d}}_{\mathsf{M}}$ of $[-1,1]^{\mathsf{d}}$ given by 
\begin{equation} \label{v2:1609011140}
\vect{F}^{\mathsf{d}}_{\mathsf{M}}=\left\{\,\vect{x}\in [-1,1]^{\mathsf{d}}\,\left|\,(\,\forall\,\mathsf{i}\in \mathsf{M}\!: x_{\mathsf{i}}\in (-1,1)\,)\;\text{and}\; (\,\forall\,\mathsf{i}\notin \mathsf{M}\!: x_{\mathsf{i}}\in\{-1,1\}\,)\right.\,\right\}.
\end{equation}

\begin{theorem}\label{v2:1503151619}Let $\vect{n}\in\mathbb{N}^{\mathsf{d}}$  satisfy \eqref{v2:1509151442}. Further, let $\vect{u}\in\{-1,1\}^{\mathsf{d}}$.

For $t\in  [0,2\pi)$, let $\mathcal{A}^{(\vect{n})}(t)$ be the set of all $s\in [0,2\pi)$ with $\vect{\ell}^{(\vect{n})}_{\vect{0},\vect{u}}(s)=\vect{\ell}^{(\vect{n})}_{\vect{0},\vect{u}}(t)$. 

\medskip

\noindent a) We have 
 
 \vspace{-1.4em}
 
 \[  \begin{array}{ll}
      \#\mathcal{A}^{(\vect{n})}(t) = 1 & \text{if}\quad t \in \{ t^{(\vect{n})}_{0}, t^{(\vect{n})}_{\p[\vect{n}]}\}=\{0,\pi\}, \\
      \#\mathcal{A}^{(\vect{n})}(t) = 2 & \text{if}\quad t \in [0,2\pi) \setminus \{\,t^{(\vect{n})}_{l}\,|\, l\in \{0,\ldots,2\p[\vect{n}]-1\}\,\}, \\
     \# \mathcal{A}^{(\vect{n})}(t) \geq 2 & \text{if}\quad t \in \{\,t^{(\vect{n})}_{l}\,|\, l\in \{1,\ldots,2\p[\vect{n}]-1\}\setminus\{\p[\vect{n}]\}\,\}.
    \end{array}
\]   

\noindent b) Let $\mathsf{M}\subseteq \{1,\ldots,\mathsf{d}\}$. For $l\in \{0,\ldots,2\p[\vect{n}]-1\}$, the following are equivalent:\\
\hspace*{2em} i) For all $\mathsf{i}\in \mathsf{M}$: $l\not\equiv 0\tmod n_{\mathsf{i}}$, and for all $\mathsf{i}\in \{1,\ldots,\mathsf{d}\}\setminus \mathsf{M}$: $l\equiv 0\tmod n_{\mathsf{i}}$.
\hspace*{2em}  ii) The value $\vect{\ell}^{(\vect{n})}_{\vect{0},\vect{u}}(t^{(\vect{n})}_{l})$ is an element of the set $\vect{F}^{\mathsf{d}}_{\mathsf{M}}$.

\noindent \hspace*{1em} If i) or ii) is satisfied, then  $\#\mathcal{A}^{(\vect{n})}(t^{(\vect{n})}_{l}) =2^{\#\mathsf{M}}$. Furthermore,
\begin{equation}\label{v2:1505031648}\#\LD^{(\vect{n})}_{\vect{u}}=\dfrac1{2^{\mathsf{d}-1}}\p[\vect{n}+\vect{1}],\quad 
\#(\LD^{(\vect{n})}_{\vect{u}}\cap\vect{F}^{\mathsf{d}}_{\mathsf{M}})=\dfrac1{2^{\#\mathsf{M}-1}}\tprod_{\mathsf{i}\in \mathsf{M}}(n_{\mathsf{i}}-1).
\end{equation}

\end{theorem}

Note that in \eqref{v2:1505031648} the product over $\mathsf{i} \in \emptyset$ is considered as $1$

\medskip

\begin{proof} We use the statements and the notation of the proof of Theorem \ref{v2:1503211414}. 
Then, $\vect{\ell}^{(\vect{n})}_{\vect{0},\vect{u}}(t)=\vect{\mathfrak{l}}^{(\indexvectq)}_{\vect{\alpha},\vect{u}}$ with $q_{\mathsf{i}}=\p_{\mathsf{i}}[\vect{n}]$. Assumption \eqref{v2:1509151442} implies
 $q=\lcm\{q_1,\ldots,q_{\mathsf{d}}\}=\p[\vect{n}]$. We  have shown that 
$\# \mathcal{A}^{(\vect{n})}(t) \geq 2$ for all $t \in (0,2 \pi) \setminus \{\pi\}$. Further, we  have shown that  $\#\mathcal{A}^{(\vect{n})}(t)\geq 3$ yields $t=l\pi/q=l\pi/\p[\vect{n}]=t^{(\vect{n})}_l$ for some  $l\in \{1,\ldots,2q-1\}\setminus\{q\}$. 
We complete the proof by showing statement b), the remaining assertions in a) then follow immediately. 

Let  $l\in \{0,\ldots,2q-1\}$. We use the set $\vect{\mathfrak{A}}(t)$ and the function  $\mathfrak{a}(t;\,\cdot\,):\,\vect{\mathfrak{A}}(t)\to \mathcal{A}^{(\vect{n})}(t)$ introduced in the proof of  Theorem \ref{v2:1503211414} for $t=t^{(\vect{n})}_{l}$.
If $s\in \mathcal{A}^{(\vect{n})}(t^{(\vect{n})}_{l})$, then $s=t^{(\vect{n})}_{l'}$ for some  $l'\in \{0,\ldots,2q-1\}$, and  condition  \eqref{v2:1503171442} for $t=t^{(\vect{n})}_{l}$ is equivalent to 
\[\forall\,\mathsf{i}\in \{1,\ldots,\mathsf{d}\}:\quad   l'\equiv v_{\mathsf{i}} l\mod 2n_{\mathsf{i}}.\]
For arbitrary $\vect{v}\in\{-1,1\}^{\mathsf{d}}$, the numbers $a_{\mathsf{i}}=v_{\mathsf{i}}l$ and $k_{\mathsf{i}}=2n_{\mathsf{i}}$ satisfy condition \eqref{v2:1503211222}. The Chinese remainder theorem implies $\vect{v}\in \vect{\mathfrak{A}}(t^{(\vect{n})}_l)$ and, thus, $\vect{\mathfrak{A}}(t^{(\vect{n})}_l)=\{-1,1\}^{\mathsf{d}}$.

Clearly, the conditions i), ii) in part b) of Theorem \ref{v2:1503151619} are equivalent.  

We have $l\equiv 0\tmod 2n_{\mathsf{j}}$  if and only if 
\[\mathfrak{a}(t^{(\vect{n})}_{l};v_1,\ldots,v_{\mathsf{j}-1},-v_{\mathsf{j}},v_{\mathsf{j}+1},\ldots,v_{\mathsf{d}})=\mathfrak{a}(t^{(\vect{n})}_{l};v_1,\ldots,v_{\mathsf{j}-1},v_{\mathsf{j}},v_{\mathsf{j}+1},\ldots,v_{\mathsf{d}})\]
for all $\vect{v}\in\{-1,1\}^{\mathsf{d}}$. Using this property and $\vect{\mathfrak{A}}(t^{(\vect{n})}_l)=\{-1,1\}^{\mathsf{d}}$, we can conclude that condition i) implies  $\#\mathcal{A}^{(\vect{n})}(t^{(\vect{n})}_{l}) =2^{\#\mathsf{M}}$. We denote by $H_{\mathsf{M}}$ the set of all $l$ in $\{0,\ldots,2q-1\}$ satisfying condition i). Using the identities
\begin{equation}\label{1608141316}
\#H_{\mathsf{M}}=2\tprod_{\mathsf{i}\in \mathsf{M}} (n_{\mathsf{i}}-1),\quad
\tsum_{\mathsf{M}\subseteq \{1,\ldots,\mathsf{d}\}} \dfrac1{2^{\# \mathsf{M}}}\tprod_{\mathsf{i}\in \mathsf{M}}(n_{\mathsf{i}}-1) = \dfrac1{2^{\mathsf{d}}}\tprod_{\mathsf{i}=1}^{\mathsf{d}} (n_{\mathsf{i}}+1),
\end{equation}
finishes the proof of the theorem. Since \eqref{v2:1509151442}, the first identity in \eqref{1608141316} is easily seen using the Chinese remainder theorem. The second identity in \eqref{1608141316} is proven by a straightforward induction argument.
\end{proof}


\section[Polynomial interpolation on \texorpdfstring{${\protect\underbar{\text{LC}}}^{(\protect\underbar{\scriptsize$\boldsymbol{n}$})}$}{LCn}]{Polynomial interpolation on ${\protect\underbar{\text{LC}}}^{\text{\small ( \!\!\protect\underbar{$\boldsymbol{n}$})}}$}\label{v2:15009022044}

\subsection{The node sets} 

We recall the \underline{general assumption \eqref{v2:1509151442} on $\vect{n}\in\mathbb{N}^{\mathsf{d}}$}. In particular, note that
because of this assumption there is at most one $\mathsf{j} \in \{1, \ldots, \mathsf{d} \}$ such that $n_{\mathsf{j}}$ is even. We define
\begin{equation}\label{v2:1509022026}
\begin{split}
\I^{(\vect{n})}&=\I^{(\vect{n})}_0 \cup \I^{(\vect{n})}_1\; \text{with the sets $\I^{(\vect{n})}_{\mathfrak{r}}$, $\mathfrak{r}\in\{0,1\}$, given by}\\
\I^{(\vect{n})}_{\mathfrak{r}}&=\left\{\,\vect{i}\in\mathbb{N}_0^{\mathsf{d}}\,\left|\,\forall\,\mathsf{j}\in\{1,\ldots,\mathsf{d}\}:\ 0\leq i_{\mathsf{j}}\leq n_{\mathsf{j}}\ 
\text{and}\ i_{\mathsf{j}}\equiv \mathfrak{r} \tmod 2\right.\right\}.
\end{split}
\end{equation}
For $\mathsf{M}\subseteq\{1,\ldots,\mathsf{d}\}$, we further introduce the subsets
\begin{equation} \label{v2:1508271623}
\I^{(\vect{n})}_{\mathsf{M}}=\I^{(\vect{n})}_{\mathsf{M},0}\cup \I^{(\vect{n})}_{\mathsf{M},1},\quad 
\I^{(\vect{n})}_{\mathsf{M},\mathfrak{r}}=
\left\{\,\left.\vect{i}\in \I^{(\vect{n})}_{\mathfrak{r}}\,\right|\,0<i_{\mathsf{j}}<n_{\mathsf{j}}\Leftrightarrow\mathsf{j}\in \mathsf{M}\,\right\}.
\end{equation}

From the particular structure of $\I^{(\vect{n})}_{\mathsf{M},\mathfrak{r}}$ as cross products of sets, we immediately obtain the cardinalities
\[\#\I^{(\vect{n})}_{\mathsf{M},\mathfrak{r}}=\dfrac1{2^{\#\mathsf{M}}} \left\{\begin {array}{rl} 
 \tprod_{\mathsf{i}\in \mathsf{M}} \hspace{6pt} (n_{\mathsf{i}}-1) &\text{if all $n_{\mathsf{i}}$ are odd,}\\
  2(1-\mathfrak{r}) \hspace{11pt}  \tprod_{\mathsf{i}\in \mathsf{M}} \hspace{6pt} (n_{\mathsf{i}}-1)  &\text{if $n_{\mathsf{j}}$ is even and $\mathsf{j}\notin\mathsf{M}$,}\\
(n_{\mathsf{j}}-2(1-\mathfrak{r}) ) \hspace{-2pt} \tprod_{\mathsf{i}\in \mathsf{M}\setminus\{\mathsf{j}\}} \hspace{-1pt}(n_{\mathsf{i}}-1)&\text{if $n_{\mathsf{j}}$ is even and $\mathsf{j}\in\mathsf{M}$}. \end{array}\right.
\]
Similarly, for the entire sets we get
\begin{equation}\label{v2:A1505031622}
\#\I^{(\vect{n})}_{\mathfrak{r}}=\dfrac1{2^{\mathsf{d}}}\left\{\begin {array}{rl} \p\,\,\![\vect{n}+\vect{1}] &\text{if all $n_{\mathsf{i}}$ are odd,}\\ 
(n_{\mathsf{j}}+2(1-\mathfrak{r}))\p_{\mathsf{j}}[\vect{n}+\vect{1}] &\text{if $n_{\mathsf{j}}$ is even}. \end{array}\right.
\end{equation}
Therefore, since $\I^{(\vect{n})}_0\cap \I^{(\vect{n})}_1=\emptyset$, we have
\begin{equation}\label{v2:A1505031632}
\#\I^{(\vect{n})} = \dfrac1{2^{\mathsf{d}-1}}\p[\vect{n}+\vect{1}], \quad \#\I^{(\vect{n})}_{\mathsf{M}}=\dfrac1{2^{\#\mathsf{M}-1}}\tprod_{\mathsf{i}\in \mathsf{M}}(n_{\mathsf{i}}-1).
\end{equation}
Using  $\I^{(\vect{n})}_{\mathfrak{r}}$ and the definition \eqref{v2:201510121326}
of the Chebyshev-Gauß-Lobatto points, we define the {\it Lissajous-Chebyshev node sets}
\[\LC^{(\vect{n})} =\LC^{(\vect{n})}_0\cup\LC^{(\vect{n})}_1,\quad \LC^{(\vect{n})}_{\mathfrak{r}}=\left\{\, \vect{z}^{(\vect{n})}_{\vect{i}}\,\left|\,\vect{i}\in \I^{(\vect{n})}_{\mathfrak{r}} \right.\right\}.\]
The following facts are easily seen.
The sets  $\LC^{(\vect{n})}_0$ and  $\LC^{(\vect{n})}_1$ are disjoint. 
The mapping  $\vect{i}\mapsto \vect{z}^{(\vect{n})}_{\vect{i}}$ is a bijection from  $\I^{(\vect{n})}_{\mathfrak{r}}$ onto  $\LC^{(\vect{n})}_{\mathfrak{r}}$,
$\mathfrak{r} \in \{0,1\}$. In particular, it is a bijection from $\I^{(\vect{n})}$ onto $\LC^{(\vect{n})}$.
With the values from \eqref{v2:A1505031622} and \eqref{v2:A1505031632}, we have
\begin{equation}\label{v2:A1505030734}\#\LC^{(\vect{n})} = \#\I^{(\vect{n})},\quad \#\LC^{(\vect{n})}_{\mathfrak{r}}=\#\I^{(\vect{n})}_{\mathfrak{r}}.
\end{equation}
Corresponding to \eqref{v2:1508271623}, we define
\[\LC^{(\vect{n})}_{\mathsf{M}} =\LC^{(\vect{n})}_{\mathsf{M},0}\cup\LC^{(\vect{n})}_{\mathsf{M},1},\quad \LC^{(\vect{n})}_{\mathsf{M},\mathfrak{r}}=\left\{\, \vect{z}^{(\vect{n})}_{\vect{i}}\,\left|\,\vect{i}\in \I^{(\vect{n})}_{\mathsf{M},\mathfrak{r}} \right.\right\}.\]
Obviously, we have
\[\LC^{(\vect{n})}_{\mathsf{M}}=\LC^{(\vect{n})}\cap \vect{F}^{\mathsf{d}}_{\mathsf{M}},\quad 
\LC^{(\vect{n})}_{\mathsf{M},\mathfrak{r}}=\LC^{(\vect{n})}_{\mathfrak{r}}\cap \vect{F}^{\mathsf{d}}_{\mathsf{M}},\]
with the set $\vect{F}^{\mathsf{d}}_{\mathsf{M}}$ given in \eqref{v2:1609011140} and
\begin{equation}\label{v2:1508052354}
\#\LC^{(\vect{n})}_{\mathsf{M}}= \#\I^{(\vect{n})}_{\mathsf{M}},\quad \#\LC^{(\vect{n})}_{\mathsf{M},\mathsf{r}}= \#\I^{(\vect{n})}_{\mathsf{M},\mathsf{r}}.
 \end{equation}
 
 \begin{figure}[htb]
	\centering
	\subfigure[ $\I^{(5,3)}$]{\includegraphics[scale=0.8]{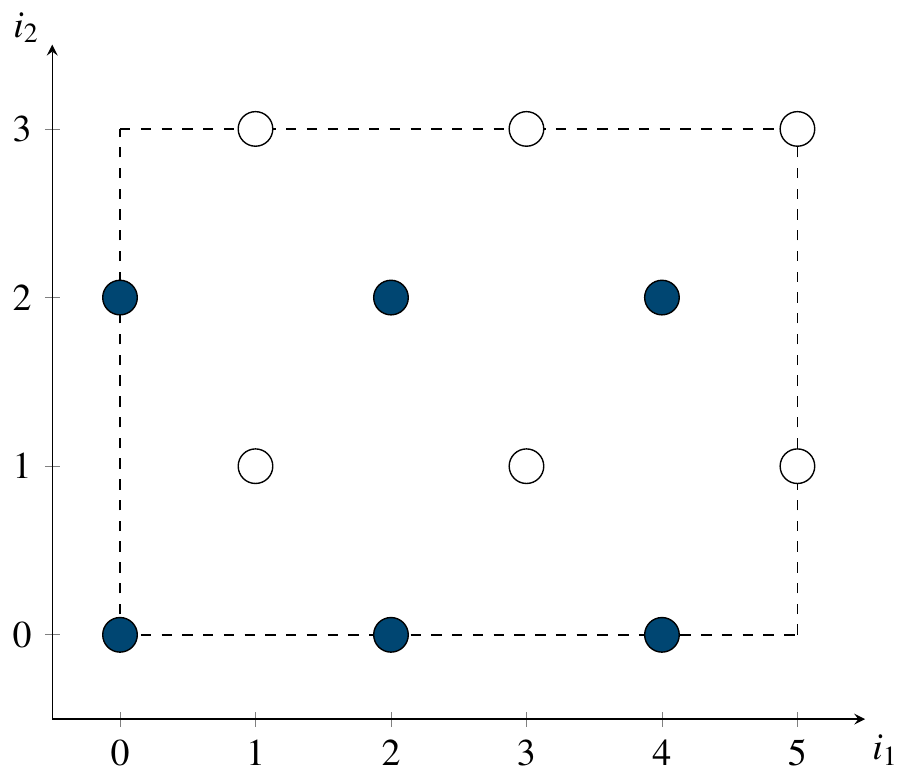}}
	\hfill	
	\subfigure[\hspace*{1em} $\I^{(5,3,2)}$]{\includegraphics[scale=0.8]{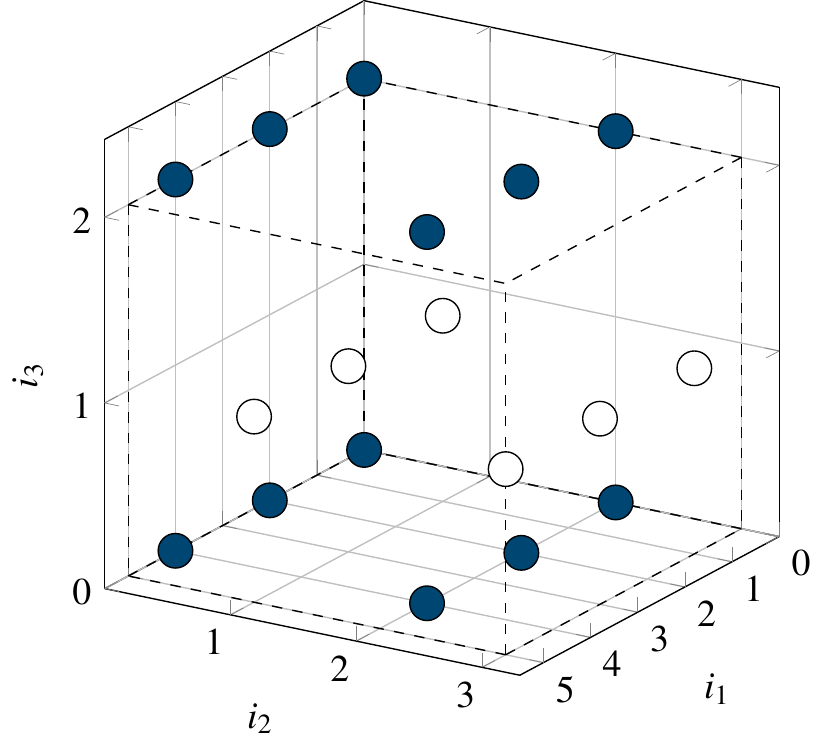}} \\
	\subfigure[\hspace*{1em} $\vect{\ell}^{(5,3)}_{\vect{0},\vect{1}}({[0,\pi]})$ and $\LC^{(5,3)}$]{\includegraphics[scale=0.8]{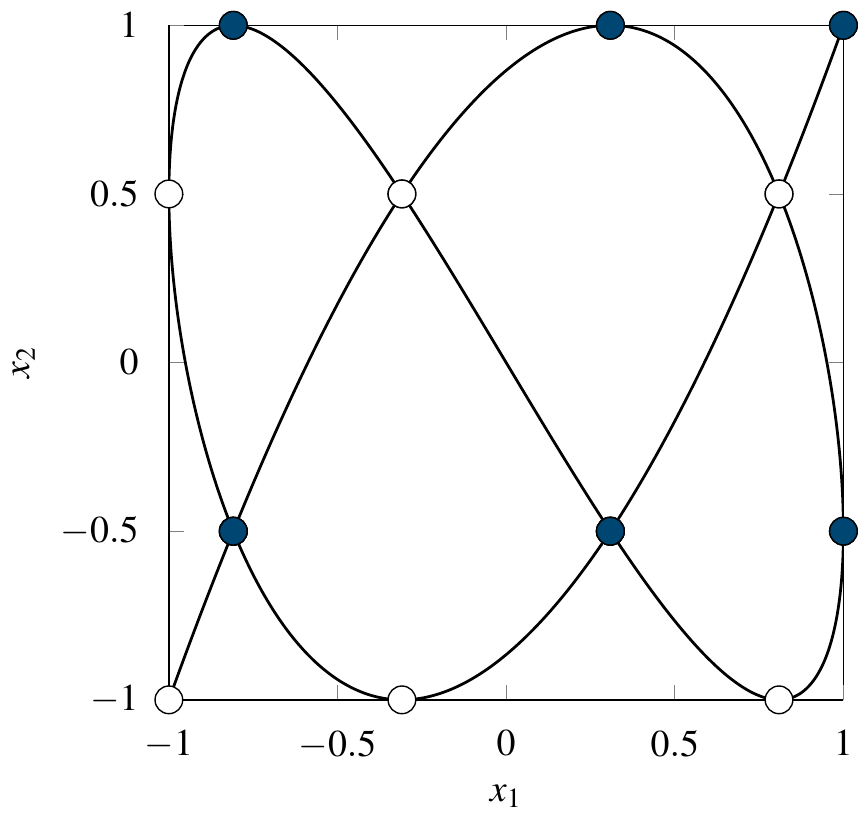}} \hfill
	\subfigure[\hspace*{1em} $\vect{\ell}^{(5,3,2)}_{\vect{0},\vect{1}}({[0,\pi]})$ and $\LC^{(5,3,2)}$]{\includegraphics[scale=0.8]{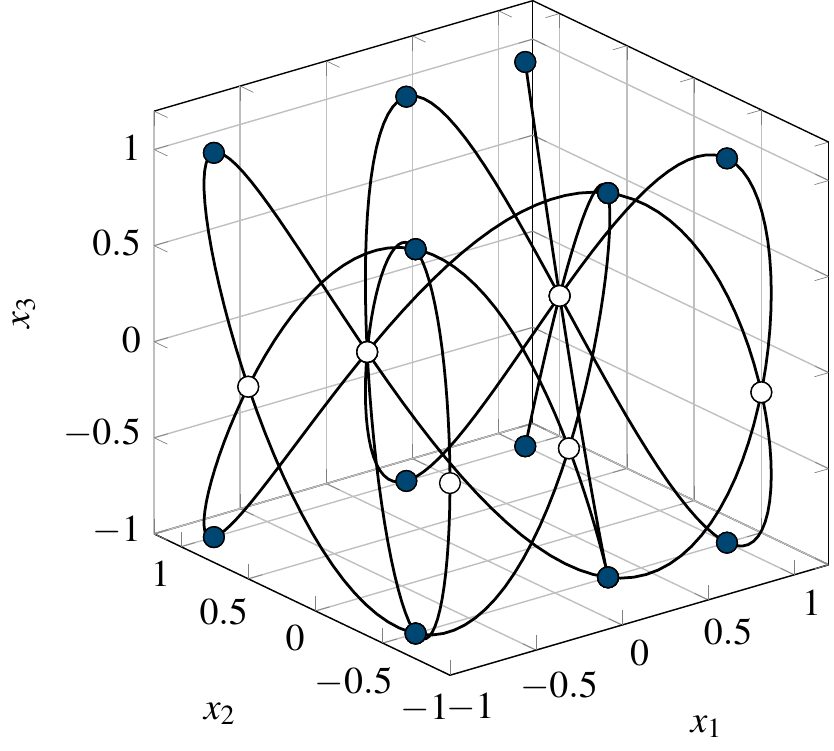}}
  	\caption{Illustration of the sets $\I^{(\vect{n})}$, $\LC^{(\vect{n})}$ and the curves $\vect{\ell}^{(\vect{n})}_{\vect{0},\vect{1}}$ in the dimensions
  	$\mathsf{d} = 2$ and $\mathsf{d} = 3$. The subsets $\I^{(\vect{n})}_0$, $\LC^{(\vect{n})}_0$ and $\I^{(\vect{n})}_1$, $\LC^{(\vect{n})}_1$ are colored in blue and 
  	white, respectively. 
  	}
	\label{v2:fig:lissajous1}
\end{figure}

Examples of the sets $\I^{(\vect{n})}$ and $\LC^{(\vect{n})}$ are illustrated in Figure \ref{v2:fig:lissajous1}. 
In Theorems \ref{v2:1509092212} and \ref{v2:1508251450}, we will see how the set $\LC^{(\vect{n})}$ is linked to the degenerate Lissajous curve $\vect{\ell}^{(\vect{n})}_{\vect{0},\vect{1}}$. In  the proof of these theorems, we use the next Proposition \ref{v2:1509021542} in which
we identify elements of $\I^{(\vect{n})}$ with the elements of a special class decomposition of the set 
\[ H^{(\vect{n})}=\{0,\ldots,2\p[\vect{n}]-1\}. \]
Proposition \ref{v2:1509021542} will play an essential role in the proof of Proposition \ref{v2:1507211320}.

\begin{proposition}\label{v2:1509021542}
a) For all $l \in \mathbb{Z}$, there exist a uniquely determined $\vect{i}\in \I^{(\vect{n})}$ and a (not necessarily  uniquely determined) $\vect{v}\in\{-1,1\}^{\mathsf{d}}$ such that
\begin{equation} \label{v2:1509021546}
\forall\,\mathsf{i}\in \{1,\ldots,\mathsf{d}\}:\quad i_{\mathsf{i}}\equiv  v_{\mathsf{i}}l\mod 2n_{\mathsf{i}}.
\end{equation}
Therefore, a function $\vect{i}^{(\vect{n})}:\,H^{(\vect{n})}\to\I^{(\vect{n})}$ is well defined by
\[\vect{i}^{(\vect{n})}(l)=\vect{i}.\]

b) For $l\in H^{(\vect{n})}$, we have $\vect{i}^{(\vect{n})}(l)\in \I^{(\vect{n})}_{\mathfrak{r}}$ if and only if $l\equiv\mathfrak{r}\tmod 2$.

c) Let $\mathsf{M}\subseteq \{1,\ldots,\mathsf{d}\}$. If  $\vect{i}\in \I^{(\vect{n})}_{\mathsf{M}}$, then $\#\{\,l\in H^{(\vect{n})}\,|\,\vect{i}^{(\vect{n})}(l)=\vect{i}\,\}=2^{\#\mathsf{M}}$.
\end{proposition}
\begin{proof} 
For $l\in\mathbb{Z}$, we can find a tupel $\vect{i}$ with $0\leq i_{\mathsf{i}}\leq n_{\mathsf{i}}$ and $\vect{v} \in \{-1,1\}^{\mathsf{d}}$ satisfying $v_{\mathsf{i}} i_{\mathsf{i}}\equiv l\tmod 2n_{\mathsf{i}}$. 
Furthermore, we have $i_{\mathsf{i}}\equiv l\equiv  i_{\mathsf{j}} \tmod 2$ for all $\mathsf{i}$, $\mathsf{j}$. 
Therefore, $\vect{i}\in \I^{(\vect{n})}$. The restriction $0\leq i_{\mathsf{i}}\leq n_{\mathsf{i}}$, $\mathsf{i}\in \{1,\ldots,\mathsf{d}\}$, implies the uniqueness of
$\vect{i}$. From \eqref{v2:1509021546} and the definitions in \eqref{v2:1509022026} we directly get statement b).

Let $\vect{i}\in \I^{(\vect{n})}$ and $\vect{v}\in\{-1,1\}^{\mathsf{d}}$. 
For $a_{\mathsf{i}}=v_{\mathsf{i}}i_{\mathsf{i}}$ and $k_{\mathsf{i}}=2n_{\mathsf{i}}$, condition \eqref{v2:1503211222} is valid
by the definition of $\I^{(\vect{n})}$ in \eqref{v2:1509022026}. The Chinese remainder theorem implies the existence of a unique  $l\in H^{(\vect{n})}$ satisfying \eqref{v2:1509021546}. Therefore, for each  $\vect{i}\in \I^{(\vect{n})}$ a function $l^{(\vect{n})}_{\vect{i}}:\,\{-1,1\}^{\mathsf{d}}\to H^{(\vect{n})}$ is well defined by $l^{(\vect{n})}_{\vect{i}}(\vect{v})=l$ with   \eqref{v2:1509021546}.   

 We have  $\vect{i}^{(\vect{n})}(l)=\vect{i}$ if and only if there is a $\vect{v}\in\{-1,1\}^{\mathsf{d}}$ such that $l=l^{(\vect{n})}_{\vect{i}}(\vect{v})$. 
  Let $\vect{i}\in \I^{(\vect{n})}_{\mathsf{M}}$. Then  $l^{(\vect{n})}_{\vect{i}}(\vect{v}')=l^{(\vect{n})}_{\vect{i}}(\vect{v})$ if and only if 
\[
\forall\,\mathsf{i}\in \mathsf{M}:\quad (v'_{\mathsf{i}}-v_{\mathsf{i}})i_{\mathsf{i}}\equiv 0\mod 2n_{\mathsf{i}},
\]
i.e. if and only if   $v'_{\mathsf{i}}=v_{\mathsf{i}}$ for all $\mathsf{i}\in\mathsf{M}$.
\end{proof}
 
\vspace{-0.7em}

\begin{theorem}\label{v2:1509092212}
For $\mathfrak{r}\in\{0,1\}$, we have
\begin{equation}\label{v2:1508041447}
\LC^{(\vect{n})}_{\mathfrak{r}}=\left\{\,\vect{\ell}^{(\vect{n})}_{\vect{0},\vect{1}}(t^{(\vect{n})}_{l})\,\left|\,l\in \{0,\ldots,\p[\vect{n}]\}\
\text{with}\ l\equiv\mathfrak{r}\tmod 2\right.\,\right\}.
\end{equation}

\vspace{-0.4em}

\noindent In particular, using the notation \eqref{v2:1509171731} with $\vect{u}=\vect{1}$, we 
can write $\LC^{(\vect{n})}=\LD^{(\vect{n})}_{\vect{1}}$.
\end{theorem}

In the framework of Lissajous curves, the node sets $\LD^{(\vect{n})}_{\vect{1}}$ of degenerate 
Lissajous curves were already considered in Theorem \ref{v2:1503151619}. 
The points in $\LC^{(\vect{n})}_{\mathsf{M}}$, $\#\mathsf{M}\geq 2$, are thus precisely the self-intersection points
of the degenerate curve $\vect{\ell}^{(\vect{n})}_{\vect{0},\vect{1}}$ in $\vect{F}^{\mathsf{d}}_{\mathsf{M}}$ which are traversed $2^{\#\mathsf{M}-1}$ times as $t$  varies  from $0$ to $\pi$.
The statements in \eqref{v2:1505031648} correspond to the respective first statements in \eqref{v2:A1505030734} and in \eqref{v2:1508052354}.

\medskip

\begin{proof} Since $\vect{\ell}^{(\vect{n})}_{\vect{0},\vect{1}}$ is degenerate, the set on the right hand side of \eqref{v2:1508041447} can be reformulated
using  all $l\in H^{(\vect{n})}$ with $l\equiv\mathfrak{r}\tmod 2$. If $\vect{i}^{(\vect{n})}(l)=\vect{i}$, we get $\vect{\ell}^{(\vect{n})}_{\vect{0},\vect{1}}(t^{(\vect{n})}_l)=\vect{z}^{(\vect{n})}_{\vect{i}}$. Thus, Proposition \ref{v2:1509021542} yields the assertion. 
\end{proof}

\noindent For $\gamma\in\mathbb{N}_0$, the univariate Chebyshev polynomials $T_{\gamma}$ of the first kind are given by
\begin{equation} \label{v2:1508251445}
T_{\gamma}(x) = \cos (\gamma \arccos x),\quad x\in [-1,1].\end{equation}
We define the affine real algebraic Chebyshev variety $\mathcal{C}^{(\vect{n})}$ in $[-1,1]^{\mathsf{d}}$ by 
\begin{equation}\label{v2:1508251451} 
\mathcal{C}^{(\vect{n})} = \left\{\,\vect{x} \in [-1,1]^{\mathsf{d}}\,\left|\,T_{n_1}(x_1)  = \ldots = T_{n_{\mathsf{d}}}(x_{\mathsf{d}})\right. \right\}. 
\end{equation} 

This algebraic variety corresponds to the degenerate Lissajous curve $\vect{\ell}^{(\vect{n})}_{\vect{0},\vect{1}}$, 
the singular points of $\mathcal{C}^{(\vect{n})}$ to the self-intersection points of the curve $\vect{\ell}^{(\vect{n})}_{\vect{0},\vect{1}}$.

\begin{theorem} \label{v2:1508251450}
We have $\mathcal{C}^{(\vect{n})}=\vect{\ell}^{(\vect{n})}_{\vect{0},\vect{1}}([0,2\pi))=\vect{\ell}^{(\vect{n})}_{\vect{0},\vect{1}}([0,\pi])$.

A point $\vect{x}\in \mathcal{C}^{(\vect{n})}$ is a  singular point of $\mathcal{C}^{(\vect{n})}$ if and only if it 
is an element of $\LC^{(\vect{n})}_{\mathsf{M}}$ for some $\mathsf{M}\subseteq \{1,\ldots,\mathsf{d}\}$ with $\#\mathsf{M}\geq 2$. Furthermore, for $\mathfrak{r}\in \{0,1\}$ we have
\begin{equation}  \label{v2:1508251536} 
\LC^{(\vect{n})}_{\mathsf{M},\mathfrak{r}} = \left\{\,  \vect{x} \in \vect{F}^{\mathsf{d}}_{\mathsf{M}}\,\left|\,T_{n_1}(x_1) = \ldots = T_{n_{\mathsf{d}}}(x_{\mathsf{d}}) = (-1)^{\mathfrak{r}} \right.\right\}.
\end{equation}
\end{theorem}
The formula $\eqref{v2:1508251536}$ is satisfied for all $\mathsf{M}\subseteq \{1,\ldots,\mathsf{d}\}$. 
In the case $\#\mathsf{M}\in \{0,1\}$, the elements of $\LC^{(\vect{n})}_{\mathsf{M},\mathfrak{r}}$ are regular points of $\mathcal{C}^{(\vect{n})}$ on the corners and edges of $[-1,1]^{\mathsf{d}}$.

\medskip

\begin{proof}
Let $\vect{x}\in \mathcal{C}^{(\vect{n})}$. We choose $\theta_{\mathsf{i}}\in [0,\pi]$ and $t'\in [0,\pi]$ such that 
$x_{\mathsf{i}}=\cos(\theta_{\mathsf{i}})$ and $T_{n_{\mathsf{i}}}(x_{\mathsf{i}})=\cos(\p[\vect{n}]t')$ for all $\mathsf{i}$. 
Then, there exist $\vect{v}\in \{-1,1\}^{\mathsf{d}}$ and $\vect{h}\in\mathbb{Z}^{\mathsf{d}}$ such that $n_{\mathsf{i}}\theta_{\mathsf{i}}= v_{\mathsf{i}}\p[\vect{n}]t'+2\pi h_{\mathsf{i}}$ for 
all $\mathsf{i}$. Thus, $x_{\mathsf{i}}=\cos(\p_{\mathsf{i}}[\vect{n}]t'+2\pi v_{\mathsf{i}}h_{\mathsf{i}}/n_{\mathsf{i}})$ for all $\mathsf{i}$. 
By the Chinese remainder theorem, we can find an $l$ such that $l\equiv v_{\mathsf{i}}h_{\mathsf{i}}\tmod n_{\mathsf{i}}$. For  $t=t'+2l\pi/\p[\vect{n}]$,
we obtain $x_{\mathsf{i}}=\cos(\p_{\mathsf{i}}[\vect{n}]t)$  for all $\mathsf{i}$.
Therefore, $\vect{x}\in \vect{\ell}^{(\vect{n})}_{\vect{0},\vect{1}}(\mathbb{R})=\vect{\ell}^{(\vect{n})}_{\vect{0},\vect{1}}([0,2\pi)) = \vect{\ell}^{(\vect{n})}_{\vect{0},\vect{1}}([0,\pi])$. 
The relation $\vect{\ell}^{(\vect{n})}_{\vect{0},\vect{1}}([0,2\pi))\subseteq \mathcal{C}^{(\vect{n})}$ can be seen immediately by inserting $\vect{\ell}^{(\vect{n})}_{\vect{0},\vect{1}}(t)$ 
in the definition \eqref{v2:1508251451} of the Chebyshev variety.

The following part is similar as in \cite[Section 3.9]{Fischer} and \cite{KoseleffPecker2011} for dimension two. 
For the considered variety, we have $\mathcal{C}^{(\vect{n})} = \left\{\,\vect{x} \in [-1,1]^{\mathsf{d}}\,\left|\,(f_1,\ldots,f_{\mathsf{d}-1})(\vect{x})=\vect{0}\right.\right\}$ with 
$f_{\mathsf{i}}(\vect{x})=T_{n_{\mathsf{i}}}(x_{\mathsf{i}})-T_{n_{\mathsf{i}+1}}(x_{\mathsf{i}+1})$. 
The singular points of this variety are by definition the points $\vect{x}^* \in [-1,1]^{\mathsf{d}}$ for which the Jacobian matrix of $(f_1,\ldots,f_{\mathsf{d}-1})$, given by
\[\begin{pmatrix} T_{n_1}'(x_1^*) & -T_{n_2}'(x_2^*) &  &  &   \\  & T_{n_2}'(x_2^*) & -T_{n_3}'(x_2^*)   & &  \\ && \ddots & \ddots & \\ &&& T_{n_{\mathsf{d}-1}}'(x_{\mathsf{d}-1}^*) & -T_{n_{\mathsf{d}}}'(x_{\mathsf{d}}^*) 
                       \end{pmatrix},
\]
has not full rank $\mathsf{d}-1$. This is exactly the case if
$T'_{n_{\mathsf{i}}}(x_{\mathsf{i}}^*) = 0$ for at least two indices $\mathsf{i} \in \{1, \ldots, \mathsf{d}\}$.  
We have $T'_{n_{\mathsf{i}}}(x_{\mathsf{i}}^*)=0$ if and only if  $x_{\mathsf{i}}^*=z^{(n_{\mathsf{i}})}_{i_{\mathsf{i}}}$
for some $i_{\mathsf{i}}\in \{1,\ldots,n_{\mathsf{i}}-1\}$. Since $\vect{x}^*\in \mathcal{C}^{(\vect{n})}$,  there exists $\mathfrak{r}\in\{0,1\}$ such that
\[ T_{n_1}(x_1^*) = \ldots = T_{n_{\mathsf{d}}}(x_{\mathsf{d}}^*)=(-1)^{\mathfrak{r}}.\]
We get  $\vect{x}^*=\vect{z}^{(\vect{n})}_{\vect{i}}\in \LC^{(\vect{n})}_{\mathfrak{r}}$. The general formula \eqref{v2:1508251536} can be derived straightforward.
Let  $\mathsf{M}\subseteq\{1,\ldots,\mathsf{d}\}$, $\vect{i}\in\I^{(\vect{n})}_{\mathsf{M}}$ and $\vect{x}^*=\vect{z}^{(\vect{n})}_{\vect{i}}$. If $\#\mathsf{M}\in \{0,1\}$, then the rank of the Jacobian matrix is $\mathsf{d}-1$, and if
$\#M\geq 2$, then the rank is $\mathsf{d}-\#M<\mathsf{d}-1$. 
\end{proof}

\begin{example} 

\begin{enumerate}
 \item[(i)] For $\mathsf{d} = 1$ and $n \in \mathbb{N}$, we obtain $\vect{\ell}_{0,1}^{(n)}(t) = \cos t$ as well as $\vect{\ell}_{0,1}^{(n)}([0,\pi]) = [-1,1]$. The points in $\LC^{(n)} = 
 \left\{\, z^{(n)}_{i}\,\left|\,i \in \{0, \ldots , n\} \right. \right\}$ are the univariate Chebyshev-Gauß-Lobatto points. 
 \item[(ii)] For $\mathsf{d} = 2$ and $n \in \mathbb{N}$, the families of the Padua points are given by the sets $\LD^{(n,n+1)}_{\vect{u}}$ and $\LD^{(n+1,n)}_{\vect{u}}$, $\vect{u} \in \{-1,1\}^2$. 
 The corresponding generating curves are the degenerate Lissajous curves $\vect{\ell}_{\vect{0},\vect{u}}^{(n,n+1)}$ or
 $\vect{\ell}_{\vect{0},\vect{u}}^{(n+1,n)}$ already considered in \cite{BosDeMarchiVianelloXu2006}. 
 Up to reflection with respect to the coordinate axis, the Padua points correspond to the point sets $\LC^{(n,n+1)}$ or $\LC^{(n+1,n)}$, respectively.
 \item[(iii)] The generating curve approach of the Padua points was generalized in \cite{Erb2015} to the degenerate Lissajous curves 
 $\vect{\ell}_{\vect{0},\vect{1}}^{(n+p,n)} (t) = ( \cos(n t), \cos( (n+p) t )$
 with relatively prime natural numbers $n$ and $p$. The corresponding interpolation nodes are given by the sets $\LC^{(n+p,n)} = \LD^{(n+p,n)}_{\vect{1}}$. 
 For $n = 3$ and $p = 2$, the set $\LC^{(5,3)}$ and the corresponding generating curve are given in Figure \ref{v2:fig:lissajous1}, (c). 
  \item[(iv)] For $\mathsf{d} = 3$, the degenerate Lissajous curve $\vect{\ell}_{\vect{0},\vect{1}}^{(\vect{n})}$ is given by  
 \[ \vect{\ell}_{\vect{0},\vect{1}}^{(\vect{n})} (t) = \left( \cos(n_2 n_3 t), \cos( n_1 n_3 t ), \cos (n_1 n_2 t ) \right).\] 
The corresponding set $\LD^{(\vect{n})}_{\vect{1}} = \LC^{(\vect{n})}$ contains $(n_1+1)(n_2+1)(n_3+1)/4$ node points, 
$(n_1+1)(n_2+1)(n_3+1)/4 - n_1 - n_2 - n_3 +1$ of which are self-intersection points of the curve $\vect{\ell}_{\vect{0},\vect{1}}^{(\vect{n})}$. For $\vect{n} = (5,3,2)$ and $\vect{n} = (3,1,2)$, such curves
are illustrated in Figure \ref{v2:fig:lissajous1}, (d) and Figure \ref{v2:fig:lissajous3}, (a), respectively. 
\end{enumerate}
\end{example}

\subsection{Discrete orthogonality structure}\label{v2:1507091240}

For $\vectgamma\in \mathbb{N}_0^{\mathsf{d}}$,  we define the functions $\dchi^{(\vect{n})}_{\indexvectgamma}\in \mathcal{L}(\I^{(\vect{n})})$ by
\begin{equation}\label{v2:A1508291531}
\dchi^{(\vect{n})}_{\indexvectgamma}(\vect{i})=\tprod_{\mathsf{j}=1}^{\mathsf{d}}\cos(\gamma_{\mathsf{j}}i_{\mathsf{j}}\pi/n_{\mathsf{j}}).
\end{equation}

\vspace{-0.5em}

\noindent For $\vect{i}\in \I^{(\vect{n})}$ we introduce the weights $\mathfrak{w}^{(\vect{n})}_{\vect{i}}$ by 
\begin{equation}\label{v2:1507091748}
\mathfrak{w}^{(\vect{n})}_{\vect{i}}=2^{\#\mathsf{M}}/(2 \p[\vect{n}])\quad \text{if}\ \ \vect{i}\in\I^{(\vect{n})}_{\mathsf{M}}.
\end{equation}
Furthermore, a   measure $\omega^{(\vect{n})}$ on the power set $\mathcal{P}(\I^{(\vect{n})})$  of $\I^{(\vect{n})}$  is
well-defined by the correspondent values $\omega^{(\vect{n})}(\{\vect{i}\})=\mathfrak{w}^{(\vect{n})}_{\vect{i}}$ for the one-element sets $\{\vect{i}\}\in\mathcal{P}(\I^{(\vect{n})})$.

\begin{proposition}\label{v2:1507211320} 
Let  $\vectgamma\in\mathbb{Z}^{\mathsf{d}}$. If $\tint\dchi^{(\vect{n})}_{\indexvectgamma}\mathrm{d}\rule{1pt}{0pt}\omega^{(\vect{n})}\neq 0$, then
\begin{equation}\label{v2:1507201132}
\text{there exists $\vect{h}\in\mathbb{N}_0^{\mathsf{d}}$ with $\gamma_{\mathsf{i}}=h_{\mathsf{i}}n_{\mathsf{i}}$, $\mathsf{i}=1,\ldots,\mathsf{d}$,  and 
$\tsum_{\mathsf{i}=1}^{\mathsf{d}}h_{\mathsf{i}}\in 2\mathbb{N}_0$}.
\end{equation}
If \eqref{v2:1507201132} is satisfied, then $\tint\dchi^{(\vect{n})}_{\indexvectgamma}\mathrm{d}\rule{1pt}{0pt}\omega^{(\vect{n})}=1$.
\end{proposition}
In the proof, we use the well-known trigonometric relations
\begin{equation}\label{v2:1507081828}
\tprod_{\mathsf{i}=1}^{\mathsf{r}}\cos(\vartheta_{\mathsf{i}})=
\dfrac1{2^{\mathsf{r}}}\tsum_{(v_1, \ldots, v_{\mathsf{r}}) \in\{-1,1\}^{\mathsf{r}}} \cos(v_1\vartheta_1+\cdots+v_{\mathsf{r}}\vartheta_{\mathsf{r}}), \quad \mathsf{r} \in \mathbb{N},
\end{equation}
\begin{equation}\label{v2:1506171253} 
\tsum_{l=0}^N \cos(l\vartheta-\vartheta_0)=\dfrac{\sin\left((N+1)\vartheta/2\right)\, \cos \left(N\vartheta/2-\vartheta_0\right) }{\sin(\vartheta/2)},\quad \vartheta\notin 2\pi\mathbb{Z},\ N\in\mathbb{N}_0.
\end{equation}
\begin{proof}
Using \eqref{v2:1507081828}, we obtain
\begin{align*}
\tint\dchi^{(\vect{n})}_{\indexvectgamma}\mathrm{d}\rule{1pt}{0pt}\omega^{(\vect{n})}
&= \dfrac1{2\p[\vect{n}]}\tsum_{\mathsf{M}\subseteq \{1, \ldots,\mathsf{d}\}}2^{\#\mathsf{M}}\tsum_{\vect{i}\in\I^{(\vect{n})}_{\mathsf{M}}}\tprod_{\mathsf{i}=1}^{\mathsf{d}}\cos(\gamma_{\mathsf{i}}i_{\mathsf{i}}\pi/n_{\mathsf{i}})\\
 &=\dfrac1{2^{\mathsf{d}}2\p[\vect{n}]}\tsum_{\vect{v}'\in \{-1,1\}^{\mathsf{d}}} \tsum_{\mathsf{M}\subseteq \{1, \ldots,\mathsf{d}\}}
2^{\#\mathsf{M}}\tsum_{\vect{i}\in\I^{(\vect{n})}_{\mathsf{M}}}
\cos\left(\pi\tsum_{\mathsf{i}=1}^{\mathsf{d}}v'_{\mathsf{i}}\gamma_{\mathsf{i}}i_{\mathsf{i}}/n_{\mathsf{i}}\right).
\end{align*}
We consider the notation and the statements of Proposition \ref{v2:1509021542}. If $\vect{i}^{(\vect{n})}(l)=\vect{i}$, then \eqref{v2:1509021546} is satisfied for some (not necessarily  uniquely determined) $\vect{v}=\vect{v}(l)\in\{-1,1\}^{\mathsf{d}}$.

Now, Proposition \ref{v2:1509021542}, c), implies
\begin{equation*}
\tint\dchi^{(\vect{n})}_{\indexvectgamma}\mathrm{d}\rule{1pt}{0pt}\omega^{(\vect{n})} = \dfrac1{2^{\mathsf{d}}2\p[\vect{n}]}\tsum_{\vect{v}'\in \{-1,1\}^{\mathsf{d}}}\tsum_{l=0}^{2\p[\vect{n}]-1}
\cos\left(l\pi\tsum_{\mathsf{i}=1}^{\mathsf{d}}v'_{\mathsf{i}} v_{\mathsf{i}}(l) \gamma_{\mathsf{i}}/n_{\mathsf{i}}\right).
\end{equation*}
Therefore, we have
\begin{equation}  \label{v2:1508221802}
\tint\dchi^{(\vect{n})}_{\indexvectgamma}\mathrm{d}\rule{1pt}{0pt}\omega^{(\vect{n})} = \dfrac1{2^{\mathsf{d}}2\p[\vect{n}]}\tsum_{\vect{v}\in \{-1,1\}^{\mathsf{d}}}\tsum_{l=0}^{2\p[\vect{n}]-1}
\cos\left(l\dfrac{\pi}{\p[\vect{n}]}\tsum_{\mathsf{i}=1}^{\mathsf{d}} v_{\mathsf{i}}\gamma_{\mathsf{i}}\p_{\mathsf{i}}[\vect{n}]\right).
\end{equation}
By \eqref{v2:1506171253}, this is zero 
if for all $\vect{v}\in\{-1,1\}^{\mathsf{d}}$ the number
\begin{equation}\label{v2:1507081849}
\dfrac1{\p[\vect{n}]}\tsum_{\mathsf{i}=1}^{\mathsf{d}}v_{\mathsf{i}}\gamma_{\mathsf{i}}\p_{\mathsf{i}}[\vect{n}]
\end{equation}
is not an element of $2\mathbb{Z}$. Let  $\tint\dchi^{(\vect{n})}_{\indexvectgamma}\mathrm{d}\rule{1pt}{0pt}\omega^{(\vect{n})}\neq 0$.  There exists $\vect{v}\in\{-1,1\}^{\mathsf{d}}$ such that \eqref{v2:1507081849} is in $2\mathbb{Z}$. Since the  $n_{\mathsf{i}}$, $\mathsf{i}=1,\ldots,\mathsf{d}$, 
are pairwise relatively prime, there is a $\vect{h}\in\mathbb{N}_0^{\mathsf{d}}$ such that $\gamma_{\mathsf{i}}=h_{\mathsf{i}}n_{\mathsf{i}}$, $\mathsf{i}=1,\ldots,\mathsf{d}$. Then,
\eqref{v2:1507081849} equals  $\tsum_{\mathsf{i}=1}^{\mathsf{d}} v_{\mathsf{i}}h_{\mathsf{i}}$ and this number is an element of $2\mathbb{Z}$ if and only if $h_1+\ldots+h_{\mathsf{d}}\in 2\mathbb{N}_0$. 

On  the other hand, if  \eqref{v2:1507201132} is satisfied, then the number \eqref{v2:1507081849} is always an even integer and therefore  the value of \eqref{v2:1508221802} is  $1$.
\end{proof}

We will show that the set of functions $\dchi^{(\vect{n})}_{\indexvectgamma}$ is a basis for the space $\mathcal{L}(\I^{(\vect{n})})$. To obtain a unique parametrization of these functions, the set 
\begin{equation}\label{v2:1508222042}
\vect{\Gamma}^{(\vect{n})} = \left\{\,\vectgamma\in\mathbb{N}_0^{\mathsf{d}}\ \left|\begin {array}{ll} 
\forall\,\mathsf{i}\in\{1,\ldots,\mathsf{d}\}:& \!\gamma_{\mathsf{i}} < n_{\mathsf{i}}, \\
\forall\,\mathsf{i},\mathsf{j} \ \text{with}\ \mathsf{i}\neq\mathsf{j} \;\hspace{1pt} :& \!\gamma_{\mathsf{i}}/n_{\mathsf{i}}+\gamma_{\mathsf{j}}/n_{\mathsf{j}} <  1
\end{array}\right. 
\right\}\cup\{(0,\ldots,0,n_{\mathsf{d}})\}
\end{equation} 
turns out to be suitable. Examples of the set $\vect{\Gamma}^{(\vect{n})}$ are given in Figure \ref{v2:fig:indexset1}. 

 \begin{figure}[htb]
	\centering
	\subfigure[\hspace*{1em} $\vect{\Gamma}^{(5,3)}$]{\includegraphics[scale=0.8]{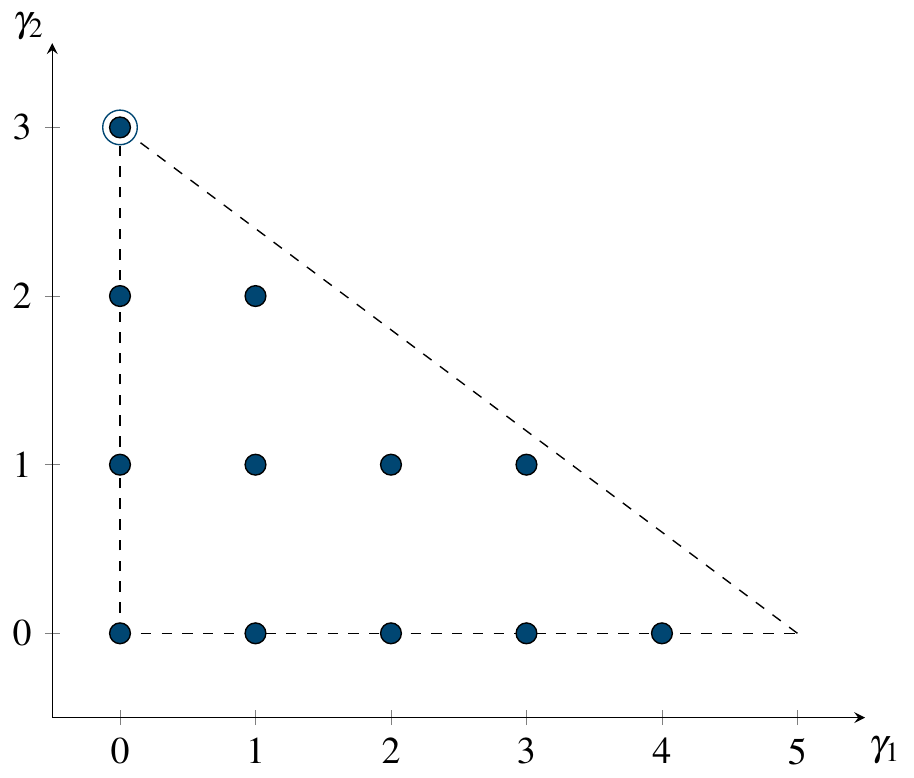}}
	\hfill	
	\subfigure[\hspace*{1em} $\vect{\Gamma}^{(5,3,2)}$]{\includegraphics[scale=0.8]{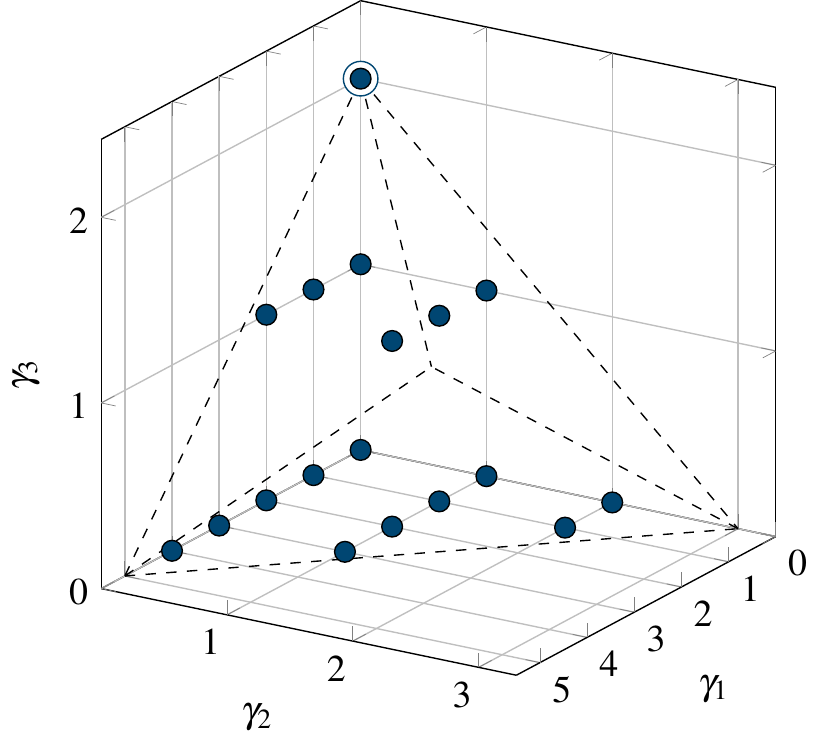}}
  	\caption{Examples of the sets $\vect{\Gamma}^{(\vect{n})}$ for
  	$\mathsf{d} = 2$ and $\mathsf{d} = 3$ related to $\I^{(\vect{n})}$ and $\LC^{(\vect{n})}$ given in Figure \ref{v2:fig:lissajous1}. 
  	The dotted lines mark the boundary of the polytopes enclosing the sets $\vect{\Gamma}^{(\vect{n})}$. The element 
  	$\{(0,\ldots,0,n_{\mathsf{d}})\}$ in \eqref{v2:1508222042} is marked
with an extra ring.
  	}
	\label{v2:fig:indexset1}
\end{figure}

\begin{proposition} \label{v2:1507151430} 
We have \[\#\vect{\Gamma}^{(\vect{n})}= \dfrac1{2^{\mathsf{d}-1}} \p[\vect{n}+\vect{1}].\]
\end{proposition}

\newpage

\begin{proof}
For $\vect{m}\in\mathbb{N}^{\mathsf{d}}$ and $\vectgamma\in\mathbb{N}^{\mathsf{d}}_0$, we set
$\max^{(\vect{m})}[\vectgamma]=\max\{\,\gamma_{\mathsf{i}}/m_{\mathsf{i}}\,|\,\mathsf{i}\in\{1,\ldots,\mathsf{d}\}\}$
and $\mathsf{k}^{(\vect{m})}[\vectgamma]=\max\{\,\mathsf{i}\in\{1,\ldots,\mathsf{d}\}\,|\,\gamma_{\mathsf{i}}/m_{\mathsf{i}}=\max\nolimits^{(\vect{m})}[\vectgamma]\,\}$. Further, we  
 define $\mathfrak{s}^{(\vect{m})}(\vectgamma)$ by
\begin{equation}\label{v2:1509051005}
\mathfrak{s}^{(\vect{m})}(\vectgamma)=(\gamma'_1,\ldots,\gamma'_{\mathsf{d}})\quad\text{with}\quad \gamma'_{\mathsf{i}}=\left\{\begin {array}{rl} 
m_{\mathsf{i}}-\gamma_{\mathsf{i},} &\text{if $\mathsf{i}=\mathsf{k}^{(\vect{m})}[\vectgamma]$},\\
\gamma_{\mathsf{i}},&\text{otherwise,}\end{array}\right.
\end{equation}
and use it for $\vect{m} = \vect{n}$ in this proof. We consider the subsets 
\begin{equation}\label{v2:150829826}
\vect{\Gamma}^{(\vect{n})}_0 = 
\left\{\,\vectgamma\in\ \mathbb{N}_0^{\mathsf{d}}\,\left|\ \forall\,\mathsf{i}:\ 2\gamma_{\mathsf{i}}\leq n_{\mathsf{i}}\right. 
\right\},\quad \vect{\Gamma}^{(\vect{n})}_1 = 
\left\{\,\vectgamma\in  \mathbb{N}_0^{\mathsf{d}}\,\left|\ \forall\,\mathsf{i}:\ 2\gamma_{\mathsf{i}}<n_{\mathsf{i}}\right. \right\}
\end{equation} of $\vect{\Gamma}^{(\vect{n})}$. From the cross product structure of the sets \eqref{v2:150829826} it is easily seen
that the cardinalities of $\vect{\Gamma}^{(\vect{n})}_0$ and $\vect{\Gamma}^{(\vect{n})}_1$ correspond to
the cardinalities of $\I^{(\vect{n})}_0$ and $\I^{(\vect{n})}_1$, respectively.
We will show that $\mathfrak{s}^{(\vect{n})}$ is a bijection from $\vect{\Gamma}^{(\vect{n})}_1$ onto 
$\vect{\Gamma}^{(\vect{n})}\setminus \vect{\Gamma}^{(\vect{n})}_0$.
Having this, the cardinality of 
$\vect{\Gamma}^{(\vect{n})}$ can be derived directly as
$\#\vect{\Gamma}^{(\vect{n})} = \#\vect{\Gamma}^{(\vect{n})}_0 + \#\vect{\Gamma}_1^{(\vect{n})} = \#\I^{(\vect{n})}_0+\#\I^{(\vect{n})}_1$ with the values given in \eqref{v2:A1505031622}.

Obviously, $\mathfrak{s}^{(\vect{n})}(\vect{0})=(0,\ldots,0,n_{\mathsf{d}})$. Now, suppose that $\vectgamma\in \vect{\Gamma}^{(\vect{n})}_1\setminus\{\vect{0}\}$. 
Since the integers $n_{\mathsf{i}}$ 
are pairwise relatively prime, there exists a $\mathsf{k}$ such that
\begin{equation}\label{v2:1505012254}
\forall\,\mathsf{i}\in \{1,\ldots,\mathsf{d}\}\setminus\{\mathsf{k}\}:\quad  \gamma_{\mathsf{i}}/n_{\mathsf{i}}<\gamma_{\mathsf{k}}/n_{\mathsf{k}}
\end{equation}
and by definition  we have $\mathsf{k}^{(\vect{n})}[\vectgamma]=\mathsf{k}$.  Let $\vectgamma'=\mathfrak{s}^{(\vect{n})}(\vectgamma)$. Then, \eqref{v2:1505012254} gives
\begin{equation}\label{v2:1505021153}
\gamma'_{\mathsf{i}}/n_{\mathsf{i}}+\gamma'_{\mathsf{k}}/n_{\mathsf{k}}=1+\gamma_{\mathsf{i}}/n_{\mathsf{i}}-\gamma_{\mathsf{k}}/n_{\mathsf{k}}<1\ \text{for all $\mathsf{i}\neq \mathsf{k}$}.
\end{equation}
This and $\vectgamma\in\vect{\Gamma}^{(\vect{n})}_1$ imply $\vectgamma'\in\vect{\Gamma}^{(\vect{n})}$. Further, since   $\gamma'_{\mathsf{k}} > n_{\mathsf{k}}/2$, 
we have  $\vectgamma'\in \vect{\Gamma}^{(\vect{n})}\setminus \vect{\Gamma}^{(\vect{n})}_0$.

On the other hand, let  $\vectgamma'\in \vect{\Gamma}^{(\vect{n})}\setminus \vect{\Gamma}^{(\vect{n})}_0$. Then, there exists a $\mathsf{k}$ such that 
$\gamma'_{\mathsf{k}}>n_{\mathsf{k}}/2$. By the definition \eqref{v2:1508222042} of $\vect{\Gamma}^{(\vect{n})}$, we have $\gamma'_{\mathsf{i}}< n_{\mathsf{k}}/2$
for all $\mathsf{i}\neq\mathsf{k}$. Therefore, $\mathsf{k}^{(\vect{n})}[\vectgamma']=\mathsf{k}$. Set $\vectgamma=\mathfrak{s}^{(\vect{n})}(\vectgamma')$.  
If $\vectgamma'=(0,\ldots,0,n_{\mathsf{d}})$, then $\vectgamma=\vect{0}$. 
If $\vectgamma'\neq (0,\ldots,0,n_{\mathsf{d}})$, then \eqref{v2:1505021153}
implies \eqref{v2:1505012254} and thus $\mathsf{k}^{(\vect{n})}[\vectgamma]=\mathsf{k}$. In both cases we have
 $\vectgamma'=\mathfrak{s}^{(\vect{n})}(\vectgamma)$. Therefore, the function $\mathfrak{s}^{(\vect{n})}$ is surjective from 
 $\vect{\Gamma}^{(\vect{n})}_1$ onto $\vect{\Gamma}^{(\vect{n})}\setminus \vect{\Gamma}^{(\vect{n})}_0$. Furthermore, we have shown that $\mathfrak{s}^{(\vect{n})}(\mathfrak{s}^{(\vect{n})}(\vectgamma)))=\vectgamma$ for all $\vectgamma\in \vect{\Gamma}^{(\vect{n})}_1$. Therefore, the mapping $\mathfrak{s}^{(\vect{n})}$ is a bijective function from  $\vect{\Gamma}^{(\vect{n})}_1$ onto $\vect{\Gamma}^{(\vect{n})}\setminus \vect{\Gamma}^{(\vect{n})}_0$.
\end{proof}

We use the notation
\begin{equation}\label{v2:1508251214}\mathfrak{e}(\vectgamma)=\#\{\,\mathsf{i}\in\{1,\ldots,\mathsf{d}\}\,|\,\gamma_{\mathsf{i}}>0\,\},
\end{equation} 
and complete this section with the following orthogonality result.

\begin{theorem}\label{v2:1507091911}
The functions $\dchi^{(\vect{n})}_{\indexvectgamma}$, $\vectgamma\in \vect{\Gamma}^{(\vect{n})}$, form an  
orthogonal basis of the inner product space $(\mathcal{L}(\I^{(\vect{n})}),\langle\,\cdot,\cdot\,\rangle_{\omega^{(\vect{n})}})$. Further, the norms of the basis functions satisfy
\begin{equation}\label{v2:1508221825}
 \|\dchi^{(\vect{n})}_{\indexvectgamma}\|_{\omega^{(\vect{n})}}^2 =
 \left\{ \begin{array}{cl}  2^{-\mathfrak{e}(\indexvectgamma)},\; & \text{if}\quad \vectgamma \in\vect{\Gamma}^{(\vect{n})}\setminus\{(0,\ldots,0,n_{\mathsf{d}})\},\\
   1,\; & \text{if}\quad \vectgamma =(0,\ldots,0,n_{\mathsf{d}}).
\end{array} \right. 
\end{equation}

\end{theorem}

\medskip

\begin{proof}
Let $\vectgamma',\vectgamma''\in \vect{\Gamma}^{(\vect{n})}$  be fixed. For $\vect{v}\in\{-1,1\}^{\mathsf{d}}$, we introduce
\begin{equation}\label{v2:1508251455}
\vectgamma(\vect{v})=(|\gamma'_1+v_1\gamma''_1|,\ldots,|\gamma'_{\mathsf{d}}+v_{\mathsf{d}}\gamma''_{\mathsf{d}}|).
\end{equation}
We consider the set 
\[\left\{\left.\,\vect{v}\in\{-1,1\}^{\mathsf{d}}\,\right|\,\vectgamma = \vectgamma(\vect{v})\ \text{satisfies \eqref{v2:1507201132}}\right\}\]
and denote the number of elements in this set with $A$. Using the definition \eqref{v2:A1508291531} and the trigonometric identity \eqref{v2:1507081828} with $\mathsf{r}=2$, we obtain 
\begin{equation}
\label{v2:1507222159}
\dchi^{(\vect{n})}_{\indexvectgamma'}\dchi^{(\vect{n})}_{\indexvectgamma''}=\dfrac1{2^{\mathsf{d}}}\tsum_{\vect{v}\in\{-1,1\}^{\mathsf{d}}}\dchi^{(\vect{n})}_{\indexvectgamma(\vect{v})}.
\end{equation} 
Thus, by Proposition \ref{v2:1507211320}, the value of $\tint\dchi^{(\vect{n})}_{\indexvectgamma'}\dchi^{(\vect{n})}_{\indexvectgamma''}\mathrm{d}\rule{1pt}{0pt}\omega^{(\vect{n})}$ equals $A/2^{\mathsf{d}}$.

Assume that $\vectgamma'\neq\vectgamma''$ and $A > 0$. Since $A > 0$, we have a $\vect{v}\in \{-1,1\}^{\mathsf{d}}$ such that the condition \eqref{v2:1507201132} is satisfied for  $\vectgamma=\vectgamma(\vect{v})$. We consider the $\vect{h} \in \mathbb{N}_0^{\mathsf{d}}$ according to \eqref{v2:1507201132} with $\vectgamma=\vectgamma(\vect{v})$.
Since $\vectgamma'\neq \vectgamma''$, there has to be an $\mathsf{i} \in \{1, \ldots,\mathsf{d}\}$ such that $h_{\mathsf{i}}>0$. 
Further, since $\vectgamma'\neq \vectgamma''$, at least one of the two vectors $\vectgamma'$ and $\vectgamma''$ is different from $(0,\ldots,0,n_{\mathsf{d}})$. 
Thus, by the definition \eqref{v2:1508222042} of $\vect{\Gamma}^{(\vect{n})}$, we always have 
\[0<h_{\mathsf{i}}=|\gamma'_{\mathsf{i}}/n_{\mathsf{i}}+v_{\mathsf{i}}\gamma''_{\mathsf{i}}/n_{\mathsf{i}}|
\leq \gamma'_{\mathsf{i}}/n_{\mathsf{i}}+\gamma''_{\mathsf{i}}/n_{\mathsf{i}}< 1+1 = 2,\]
i.e. we have   $h_{\mathsf{i}}=1$.  As  $h_1+\ldots+h_{\mathsf{d}}\in 2\mathbb{N}_0$, there is a $\mathsf{j}\neq \mathsf{i}$ such that  $h_{\mathsf{j}}> 0$. Therefore both,  
$\vectgamma'$ and $\vectgamma''$,  are different from $(0,\ldots,0,n_{\mathsf{d}})$. If $v_{\mathsf{k}}=-1$ holds 
for some  $\mathsf{k}$, then  $n_{\mathsf{k}}$ divides $\gamma'_{\mathsf{k}}-\gamma_{\mathsf{k}}''$ and thus, since  $0\leq \gamma'_{\mathsf{k}},\gamma_{\mathsf{k}}''<n_{\mathsf{k}}$, 
we have $\gamma'_{\mathsf{k}}=\gamma_{\mathsf{k}}''=0$, i.e. $h_{\mathsf{k}}= 0$. We conclude $v_{\mathsf{i}}=v_{\mathsf{j}}=1$. Now, the definition 
of  $\vect{\Gamma}^{(\vect{n})}$ yields  the contradiction 
\[2\leq h_{\mathsf{i}}+h_{\mathsf{j}}=\gamma_{\mathsf{i}}'/n_{\mathsf{i}}+\gamma_{\mathsf{j}}'/n_{\mathsf{j}}+\gamma''_{\mathsf{i}}/n_{\mathsf{i}}+\gamma''_{\mathsf{j}}/n_{\mathsf{j}}<1+1.\]

On the other hand, let $\vectgamma'=\vectgamma''$. If $\vectgamma'= (0,\ldots,0,n_{\mathsf{d}})$, then $A=2^{\mathsf{d}}$. Now, consider 
$\vectgamma' \in \vect{\Gamma}^{(\vect{n})} \setminus (0,\ldots,0,n_{\mathsf{d}})$. Then, condition \eqref{v2:1507201132} is satisfied for $\vectgamma=\vectgamma(\vect{v})$ if and only if $h_{\mathsf{i}}= (1+v_{\mathsf{i}}) \gamma'_{\mathsf{i}}/n_{\mathsf{i}} \in \{0,1\}$ for all $\mathsf{i}$ and $\tsum_{\mathsf{i}=1}^{\mathsf{d}}h_{\mathsf{i}}\in 2\mathbb{N}_0$. 
Note that $0\leq \gamma'_{\mathsf{i}}<n_{\mathsf{i}}$ for all $\mathsf{i}$, and that 
at most one of the $n_{\mathsf{i}}$ is even.  Therefore, \eqref{v2:1507201132} holds for $\vectgamma=\vectgamma(\vect{v})$ if and only if $h_{\mathsf{i}} = 0$ for all $\mathsf{i} \in \{1,\ldots, \mathsf{d}\}$. The number of all $\vect{v} \in \{-1,1\}^{\mathsf{d}}$ such that $\vectgamma=\vectgamma(\vect{v})$ satisfies \eqref{v2:1507201132} is therefore given by $A=2^{\mathsf{d}-\mathfrak{e}(\indexvectgamma')}$. 

The just shown formula \eqref{v2:1508221825} implies that
$\dchi^{(\vect{n})}_{\indexvectgamma}\neq 0$ for all $\vectgamma$. 
Thus, since the functions $\dchi^{(\vect{n})}_{\indexvectgamma}$, $\vectgamma \in \vect{\Gamma}^{(\vect{n})}$,
are pairwise orthogonal, they are in particular linearly independent. Further, 
since $\#\vect{\Gamma}^{(\vect{n})}=\#\I^{(\vect{n})}=\dim \mathcal{L}(\I^{(\vect{n})})$ by Proposition \ref{v2:1507151430}, and \eqref{v2:A1505031632},
the functions $\dchi^{(\vect{n})}_{\indexvectgamma}$, $\vectgamma \in \vect{\Gamma}^{(\vect{n})}$, form a basis of the vector space $\mathcal{L}(\I^{(\vect{n})})$.
\end{proof}

\subsection{Polynomial interpolation} \label{v2:sec:polyinter1649}
Based on the univariate Chebyshev polynomials in \eqref{v2:1508251445}, we introduce for $\vectgamma\in\mathbb{N}_0^{\mathsf{d}}$ 
the {\it $\mathsf{d}$-variate Chebyshev polynomials} by
\[T_{\indexvectgamma}(\vect{x})=  T_{\gamma_1}(x_1) \cdot \ldots \cdot T_{\gamma_{\mathsf{d}}}(x_{\mathsf{d}}),\quad \vect{x}\in[-1,1]^{\mathsf{d}}.\]
Let $\Pi^{\mathsf{d}}$ be the complex vector space of all $\mathsf{d}$-variate polynomial functions $[-1,1]^{\mathsf{d}}\to\mathbb{C}$. 
From the well-known properties of the Chebyshev polynomials $T_{\gamma}$ of the first kind it follows immediately (cf. \cite{DunklXu}) that 
the $T_{\indexvectgamma}(\vect{x})$, $\vectgamma\in\mathbb{N}_0^{\mathsf{d}}$, form an orthogonal basis of $\Pi^{\mathsf{d}}$
with respect to the inner product defined by
 \begin{equation} \label{v2:1508220014}
  \langle f,g \rangle = \frac{1}{\pi^{\mathsf{d}}} \int_{[-1,1]^{\mathsf{d}}} f(\vect{x}) \overline{g(\vect{x})} w_{\mathsf{d}}(\vect{x})\,\mathrm{d}\vect{x}, \quad w_{\mathsf{d}}(\vect{x})  = \tprod_{\mathsf{i} = 1}^{\mathsf{d}}  \displaystyle\frac{1}{\sqrt{1-x_{\mathsf{i}}^2}}.
 \end{equation}
The norms of these basis elements can easily be computed as
\begin{equation} \label{v2:1508210758}
\|T_{\indexvectgamma}\|^2  = 2^{-\mathfrak{e}(\indexvectgamma)}. \end{equation}

\medskip

Now, we investigate the points $\vect{z}^{(\vect{n})}_{\vect{i}} \in \LC^{(\vect{n})}$, $\vect{i} \in \I^{(\vect{n})}$, with regard to $\mathsf{d}$-variate polynomial
interpolation on  $[-1,1]^{\mathsf{d}}$. We are searching for an interpolation polynomial $P^{(\vect{n})}_h$ that for given data values $h(\vect{i}) \in \mathbb{R}$, $\vect{i} \in \I^{(\vect{n})}$,  satisfies
\begin{equation}\label{v2:1508220011}
 P^{(\vect{n})}_h (\vect{z}^{(\vect{n})}_{\vect{i}}) = h({\vect{i}}) \quad \text{for all}\quad \vect{i} \in \I^{(\vect{n})}.
\end{equation}

We have to specify an appropriate polynomial space as an underlying space for this  polynomial interpolation problem. 
For all $\vectgamma\in\mathbb{N}_0^{\mathsf{d}}$ and $\vect{i}\in \I^{(\vect{n})}$, we have
\begin{equation}\label{v2:1508201411}
T_{\indexvectgamma}(\vect{z}^{(\vect{n})}_{\vect{i}}) = \dchi^{(\vect{n})}_{\indexvectgamma}(\vect{i}).
\end{equation}
Therefore,
we have a direct relation between the basis polynomials $T_{\indexvectgamma}$, the points $\vect{z}^{(\vect{n})}_{\vect{i}} \in \LC^{(\vect{n})}$, 
and the functions $\dchi^{(\vect{n})}_{\indexvectgamma}$ on $\I^{(\vect{n})}$ defined in \eqref{v2:A1508291531}.
The relation \eqref{v2:1508201411} and the results of Section \ref{v2:1507091240} motivate the introduction of the polynomial space
\[\Pi^{(\vect{n})} = \vspan \left\{\, T_{\indexvectgamma}\,\left|\, \vectgamma \in \vect{\Gamma}^{(\vect{n})} \right. \right\}.\]

The set $ \{\, T_{\indexvectgamma}\,|\, \vectgamma \in \vect{\Gamma}^{(\vect{n})} \}$ is an orthogonal basis of the space $\Pi^{(\vect{n})}$ with
respect to the inner product given in \eqref{v2:1508220014}. 
The inner product space $(\Pi^{(\vect{n})}, \langle\cdot\,,\,\cdot\rangle)$ has the \textit{reproducing kernel}
\[K^{(\vect{n})}(\vect{x},\vect{x}') = \tsum_{\indexvectgamma \in \vect{\Gamma}^{(\vect{n})}} \dfrac1{\|T_{\indexvectgamma}\|^2} T_{\indexvectgamma}(\vect{x}) T_{\indexvectgamma}(\vect{x}'),\quad \vect{x},\vect{x}'\in [-1,1]^{\mathsf{d}},
\]
i.e. all polynomials $P\in \Pi^{(\vect{n})}$ can be represented in the form
\[P(\vect{x})=\langle P , K^{(\vect{n})}(\vect{x},\,\cdot\,) \,\rangle,\quad \vect{x}\in [-1,1]^{\mathsf{d}}.\]
We use the weights given in \eqref{v2:1507091748}. For $\vect{i} \in \I^{(\vect{n})}$, we introduce the polynomials 
\begin{equation} \label{v2:1508220009}
 L^{(\vect{n})}_{\vect{i}}(\vect{x}) = \mathfrak{w}^{(\vect{n})}_{\vect{i}} \left( K^{(\vect{n})}(\vect{x}, \vect{z}^{(\vect{n})}_{\vect{i}} ) 
 - T_{n_{\mathsf{d}}}(x_{\mathsf{d}})\;\! T_{n_{\mathsf{d}}}(z^{(n_d)}_{i_{\mathsf{d}}}) \right), \quad \vect{x} \in [-1,1]^{\mathsf{d}},
\end{equation}
and obtain the following result for polynomial interpolation on the point set $\LC^{(\vect{n})}$.

\newpage

\begin{theorem}\label{v2:1509011746}  For $h\in\mathcal{L}(\I^{(\vect{n})})$, the interpolation problem \eqref{v2:1508220011} has the uniquely determined solution  
\begin{equation}\label{v2:1508211556} 
P^{(\vect{n})}_h = \sum_{\vect{i} \in \I^{(\vect{n})}} h({\vect{i}}) L^{(\vect{n})}_{\vect{i}}
\end{equation}
 in the polynomial space $\Pi^{(\vect{n})}$. Further, $ \vspan \{\,P^{(\vect{n})}_h\,|\,h\in \mathcal{L}(\I^{(\vect{n})})\,\}=\Pi^{(\vect{n})}$ 
 and the polynomials $L^{(\vect{n})}_{\vect{i}}$, $\vect{i}\in\I^{(\vect{n})}$, form a basis of  $\Pi^{(\vect{n})}$. 
\end{theorem}

 The polynomial  $L^{(\vect{n})}_{\vect{i}}$ is the unique solution of the interpolation problem \eqref{v2:1508220011}
for $h=\delta_{\vect{i}}$. The polynomials $L^{(\vect{n})}_{\vect{i}}$, $\vect{i}\in\I^{(\vect{n})}$,  are called the 
{\it fundamental solutions}.

\medskip

\begin{proof} For every $h\in \mathcal{L}(\I^{(\vect{n})})$, we define $P^{(\vect{n})}_h$ by \eqref{v2:1508211556}. 
By definition \eqref{v2:1508220009}, we have $L^{(\vect{n})}_{\vect{i}}\in\Pi^{(\vect{n})}$ for $\vect{i}\in \I^{(\vect{n})}$, and therefore $P^{(\vect{n})}_h\in \Pi^{(\vect{n})}$.

For $\vect{i}\in\I^{(\vect{n})}$, let $h_{\vect{i}}\in\mathcal{L}(\I^{(\vect{n})})$ be given by  $h_{\vect{i}}(\vect{i}')=L^{(\vect{n})}_{\vect{i}}(\vect{z}_{\vect{i}'}^{(\vect{n})})$, $\vect{i}'\in\I^{(\vect{n})}$. 
Taking into account the relation \eqref{v2:1508201411} as well as \eqref{v2:1508221825} and \eqref{v2:1508210758}, we can conclude
\begin{equation}\label{v2:1508211739}\begin{split}
h_{\vect{i}}(\vect{i}')&=\mathfrak{w}^{(\vect{n})}_{\vect{i}} \left(\, \tsum_{\indexvectgamma \in \vect{\Gamma}^{(\vect{n})}} \dfrac1{\|T_{\indexvectgamma}\|^2} T_{\indexvectgamma}(\vect{z}_{\vect{i}}^{(\vect{n})}) T_{\indexvectgamma}(\vect{z}_{\vect{i}'}^{(\vect{n})})- T_{n_{\mathsf{d}}}(z_{i_{\mathsf{d}}}^{(n_{\mathsf{d}})})\;\! T_{n_{\mathsf{d}}}(z_{i'_{\mathsf{d}}}^{(n_{\mathsf{d}})}) \right) 
\\
&=\mathfrak{w}^{(\vect{n})}_{\vect{i}} \tsum_{\indexvectgamma \in \vect{\Gamma}^{(\vect{n})}} \dfrac1{\|\dchi^{(\vect{n})}_{\indexvectgamma}\|_{\omega^{(\vect{n})}}^2}\dchi^{(\vect{n})}_{\indexvectgamma}(\vect{i})\dchi^{(\vect{n})}_{\indexvectgamma}(\vect{i}').
\end{split}
\end{equation}
Theorem \ref{v2:1507091911} and \eqref{v2:1508211739} imply that for  every $\vectgamma\in\vect{\Gamma}^{(\vect{n})}$ we have 
\[\langle h_{\vect{i}},\dchi^{(\vect{n})}_{\indexvectgamma}\rangle_{\omega^{(\vect{n})}}=\mathfrak{w}^{(\vect{n})}_{\vect{i}}\dchi^{(\vect{n})}_{\indexvectgamma}(\vect{i}).\]
Further, by definition of $\langle\,\cdot,\cdot\,\rangle_{\omega^{(\vect{n})}}$ we have $\langle \delta_{\vect{i}},\dchi^{(\vect{n})}_{\indexvectgamma}\rangle_{\omega^{(\vect{n})}}=\mathfrak{w}^{(\vect{n})}_{\vect{i}}\dchi^{(\vect{n})}_{\indexvectgamma}(\vect{i})$ for all $\vectgamma\in\vect{\Gamma}^{(\vect{n})}$. 
Since the $\dchi^{(\vect{n})}_{\indexvectgamma}$, $\vectgamma\in\vect{\Gamma}^{(\vect{n})}$, form a basis of $\mathcal{L}(\I^{(\vect{n})})$,
it follows that $h_{\vect{i}}=\delta_{\vect{i}}$ for all $\vect{i}\in \mathcal{L}(\I^{(\vect{n})})$.  This implies that the function  $P^{(\vect{n})}_h$ given in \eqref{v2:1508211556}
fulfills \eqref{v2:1508220011}. 

Clearly, the homomorphism  $h\mapsto P^{(\vect{n})}_h$  from the vector space $\mathcal{L}(\I^{(\vect{n})})$ 
into the vector space $\Pi^{(\vect{n})}$  is injective. Thus,
since $\dim \mathcal{L}(\I^{(\vect{n})}) = \#\I^{(\vect{n})}=\#\vect{\Gamma}^{(\vect{n})}=\dim \Pi^{(\vect{n})}$ by Proposition \ref{v2:1507151430}  and \eqref{v2:A1505031632}, 
this homomorphism is  bijective onto  $\Pi^{(\vect{n})}$. 
\end{proof}

Now, we consider the  uniquely determined coefficients $c_{\indexvectgamma}(h)$ in the expansion
\begin{equation}\label{v2:1509241909B}
P^{(\vect{n})}_h = \sum_{\indexvectgamma \in \vect{\Gamma}^{(\vect{n})}} c_{\indexvectgamma}(h)\;\!T_{\indexvectgamma}
\end{equation}
of the interpolating polynomial $P^{(\vect{n})}_h$ in the orthogonal basis $T_{\indexvectgamma}$, $\vectgamma \in \vect{\Gamma}^{(\vect{n})}$, of the polynomial space  $\Pi^{(\vect{n})}$.
With Theorem \ref{v2:1509011746}, we obtain the following identity. 

\begin{corollary}
For $h\in\mathcal{L}(\I^{(\vect{n})})$, the coefficients $c_{\indexvectgamma}(h)$ in \eqref{v2:1509241909B} are given by
\[c_{\indexvectgamma}(h)=\dfrac1{\|\dchi^{(\vect{n})}_{\indexvectgamma}\|_{\omega^{(\vect{n})}}^2}\langle\;\! h,\dchi^{(\vect{n})}_{\indexvectgamma}\rangle_{\omega^{(\vect{n})}}.\]
\end{corollary}

Using this formula, the coefficients $c_{\indexvectgamma}(h)$ can be computed efficiently using discrete cosine transforms along the $\mathsf{d}$ dimensions of the index set 
\[ \J^{(\vect{n})}= \bigtimes_{\substack{\vspace{-8pt}\\\mathsf{i}=1}}^{\substack{\mathsf{d}\\\vspace{-10pt}}} \{0,\ldots,n_{\mathsf{i}}\}.\]
We introduce
\[ g^{(\vect{n})}(\vect{i}) = \left\{ \begin{array}{cl} \mathfrak{w}^{(\vect{n})}_{\vect{i}} h(\vect{i}), \quad  & \text{if}\; \vect{i}\in\I^{(\vect{n})},\rule[-0.65em]{0pt}{1em}\\ 
       0, \quad 
  & \text{if}\; \vect{i}\in\J^{(\vect{n})} \setminus \I^{(\vect{n})}, \end{array} \right. \]
and, recursively for $\mathsf{j}=1,\ldots,\mathsf{d}$, we set
\[g^{(\vect{n})}_{(\gamma_1,\ldots,\gamma_{\mathsf{j}})}(i_{\mathsf{j}+1},\ldots,i_{\mathsf{d}})= \sum_{i_{\mathsf{j}}=0}^{n_{\mathsf{j}}}
g^{(\vect{n})}_{(\gamma_1,\ldots,\gamma_{\mathsf{j}-1})}(i_{\mathsf{j}},\ldots,i_{\mathsf{d}})\cos(\gamma_{\mathsf{j}}i_{\mathsf{j}}\pi/n_{\mathsf{j}}).\]
Then, since \eqref{v2:1508221825}, we have
\[c_{\indexvectgamma}(h) =  
 \left\{ \begin{array}{rl}  2^{\mathfrak{e}(\indexvectgamma)}g^{(\vect{n})}_{\indexvectgamma},\; & \text{if}\quad 
 \vectgamma \in\vect{\Gamma}^{(\vect{n})}\setminus\{(0,\ldots,0, n_{\mathsf{d}})\},\\[0.3em]
   g^{(\vect{n})}_{\indexvectgamma},\; & \text{if}\quad \vectgamma =(0,\ldots,0, n_{\mathsf{d}}). \end{array}\right.
\]
Using the fast cosine transform $\mathsf{d}$ times, the complexity for the computation of the set of coefficients $c_{\indexvectgamma}(h)$ is of order $\mathcal{O}\!\left(\!\!\;\p[\vect{n}] \ln \!\!\;\p[\vect{n}] \right)$. 
Once the coefficients $c_{\indexvectgamma}(h)$ are computed, the evaluation of the interpolating polynomial $P^{(\vect{n})}_h(\vect{x})$ at $\vect{x} \in [-1,1]^{\mathsf{d}}$ is carried out with help of formula 
\eqref{v2:1509241909B}.

\medskip

We can also formulate a quadrature rule formula for $\mathsf{d}$-variate polynomials.

\begin{theorem} \label{v2:1509011747}
Let $P$ be  a  $\mathsf{d}$-variate polynomial function $[-1,1]^{\mathsf{d}}\to\mathbb{C}$.  If 
 \[ \text{$\langle P, T_{\indexvectgamma} \rangle = 0$ for all $\vectgamma \in\mathbb{N}_0^{\mathsf{d}} \setminus \{ \vect{0} \}$,  satisfying \eqref{v2:1507201132},}\] then 
\begin{equation} \label{v2:1508220012}
\frac{1}{\pi^{\mathsf{d}}} \int_{[-1,1]^{\mathsf{d}}} P(\vect{x}) w(\vect{x}) \,\mathrm{d}\vect{x} = \sum_{\vect{i} \in \I^{(\vect{n})}} \mathfrak{w}^{(\vect{n})}_{\vect{i}} P(\vect{z}^{(\vect{n})}_{\vect{i}}). 
\end{equation}
\end{theorem}

\begin{proof}
It suffices to consider the basis polynomials $P=T_{\indexvectgamma}$. 
For the left hand side of \eqref{v2:1508220012} we obtain $1$ if $\vectgamma = \vect{0}$ and zero otherwise. By \eqref{v2:1508201411} and Proposition \ref{v2:1507211320}, the
right hand side of \eqref{v2:1508220012} is $1$ if condition \eqref{v2:1507201132} is satisfied and zero otherwise. 
\end{proof}

\begin{remark}
For $\mathsf{d} = 1$, the formulas \eqref{v2:1508211556} and \eqref{v2:1508220012} correspond to the formulas for univariate Chebyshev-Gauß-Lobatto interpolation and quadrature on the interval $[-1,1]$, cf. \cite[Section 3.4]{Shen2011}. 
For $\mathsf{d} = 2$, these formulas were first proven for the Padua points \cite{BosDeMarchiVianelloXu2006} and later on extended in \cite{Erb2015} to general two-dimensional degenerate Lissajous curves.
The described approach for the efficient computation of the coefficients $c_{\indexvectgamma}(f)$ using fast Fourier methods
was originally developed in \cite{CaliariDeMarchiSommarivaVianello2011} for the Padua points and later on extended in \cite{ErbKaethnerAhlborgBuzug2015,ErbKaethnerDenckerAhlborg2015} for general two-dimensional Lissajous curves. 
\end{remark}

\begin{remark}
The sets $\LD^{(\vect{n})}_{\vect{u}}$, $\vect{u} \in \{-1,1\}^{\mathsf{d}}$, in \eqref{v2:1509171731} and the corresponding generating curves $\vect{\ell}^{(\vect{n})}_{\vect{0},\vect{u}}$ are reflected versions of the sets $\LC^{(\vect{n})} = \LD^{(\vect{n})}_{\vect{1}}$ and the curve $\vect{\ell}^{(\vect{n})}_{\vect{0},\vect{1}}$, respectively. Therefore, all results of this Section \ref{v2:15009022044} can be formulated and proven analogously also for the node sets $\LD^{(\vect{n})}_{\vect{u}}$.
\end{remark}

\begin{remark}
The sets $\LC^{(\vect{n})}$ can be considered as particular $\mathsf{d}$-dimensional Chebyshev lattices of rank one, see 
\cite{CoolsPoppe2011,PoppeCools2012,PoppeCools2013,PottsVolkmer2015}. 
According to the notation given in \cite{CoolsPoppe2011,PoppeCools2013}, the lattice parameters of the set $\LC^{(\vect{n})}$ are given by 
$d_1 =  \p[\vect{n}]$, $\mathbf{z}_1 = (\p_1[\vect{n}], \ldots, \p_{\mathsf{d}}[\vect{n}])$
and $\mathbf{z}_{\Delta} = \vect{0}$. For general Chebyshev lattices, interpolation is usually only considered in an approximative way, known as hyperinterpolation, see \cite{Sloan1995}. 
For $\LC^{(\vect{n})}$, Theorem \ref{v2:1509011746} shows that more is possible, namely unique interpolation in 
$\Pi^{(\vect{n})}$. 
\end{remark}


\section[Polynomial interpolation on \texorpdfstring{${\protect\underbar{\text{LC}}}^{(2 \protect\underbar{\scriptsize $\boldsymbol{n}$})}_{\protect\underbar{\scriptsize $\boldsymbol{\kappa}$}}$}{LC2nkappa}]
{Polynomial interpolation on ${\protect\underbar{\text{LC}}}^{(\text{\small $2$} \protect\underbar{\small $\boldsymbol{n}$})}_{\protect\underbar{\small $\boldsymbol{\kappa}$}}$} \label{v2:201510121546}
\subsection{The node sets} 

Again, we recall the  \underline{general assumption \eqref{v2:1509151442} on $\vect{n}\in\mathbb{N}^{\mathsf{d}}$}. For $\vect{\kappa}\in\mathbb{Z}^{\mathsf{d}}$, we define
\[\begin{split}
\I^{(2\vect{n})}_{\vect{\kappa}}&=\I^{(2\vect{n})}_{\vect{\kappa},0}\cup \I^{(2\vect{n})}_{\vect{\kappa},1},\ \text{with the sets $\I^{(2\vect{n})}_{\vect{\kappa},\mathfrak{r}}$ , $\mathfrak{r}\in\{0,1\}$, given by}\\
\I^{(2\vect{n})}_{\vect{\kappa},\mathfrak{r}}&=\left\{\,\vect{i}\in\mathbb{N}_0\,\left|\,\forall\,\mathsf{j}\in\{1,\ldots,\mathsf{d}\}:\ 0\leq i_{\mathsf{j}}\leq 2n_{\mathsf{j}}\ \ \text{and}\ \ i_{\mathsf{j}}\equiv \kappa_{\mathsf{j}}+\mathfrak{r} \tmod 2\right.\,\right\}.
\end{split}\]
From the particular structure of $\I^{(2\vect{n})}_{\vect{\kappa},\mathfrak{r}}$ as cross product of sets and the fact that $\I^{(2\vect{n})}_{\vect{\kappa},0}\cap \I^{(2\vect{n})}_{\vect{\kappa},1}=\emptyset$, we immediately obtain the cardinalities
\begin{equation}\label{v2:1505030732}\#\I^{(2\vect{n})}_{\vect{\kappa}}=\#\I^{(2\vect{n})}_{\vect{\kappa},0}+\#\I^{(2\vect{n})}_{\vect{\kappa},1},\quad 
\# \I_{\vect{\kappa},\mathfrak{r}}^{(2\vect{n})}=\tprod_{\substack{\mathsf{i}\in \{1,\ldots,\mathsf{d}\}:\\ \kappa_{\mathsf{i}}\equiv \mathfrak{r}\tmod 2}}(n_{\mathsf{i}}+1)\ \times\!\!\!\!\tprod_{\substack{\mathsf{i}\in \{1,\ldots,\mathsf{d}\}:\\ \kappa_{\mathsf{i}}\not\equiv\mathfrak{r}\tmod 2}}n_{\mathsf{i}}.
\end{equation}
\text{Note that in the special case $\vect{\kappa}=\vect{0}$ we can write $\#\I_{\vect{0},0}^{(2\vect{n})}=\p[\vect{n}+\vect{1}]$, $\#\I_{\vect{0},0}^{(2\vect{n})}=\p[\vect{n}]$.}\\
Using the definition \eqref{v2:201510121326}
of the Chebyshev-Gauß-Lobatto points, we introduce another type of {\it Lissajous-Chebyshev node sets}:
\[\LC^{(2\vect{n})}_{\vect{\kappa}} = \LC^{(2\vect{n})}_{\vect{\kappa},0}\cup \LC^{(2\vect{n})}_{\vect{\kappa},1},\quad
\LC^{(2\vect{n})}_{\vect{\kappa},\mathfrak{r}}=\left\{\, \vect{z}^{(2\vect{n})}_{\vect{i}}\,\left|\,\vect{i}\in \I^{(2\vect{n})}_{\vect{\kappa},\mathfrak{r}} \right.\right\}.\]
We have the following properties. 
The sets  $\LC^{(2\vect{n})}_{\vect{\kappa},0}$ and  $\LC^{(2\vect{n})}_{\vect{\kappa},1}$ are disjoint and the mapping $\vect{i}\mapsto \vect{z}^{(2\vect{n})}_{\vect{i}}$ is a bijection  from $\I^{(2\vect{n})}_{\vect{\kappa},\mathfrak{r}}$ onto $\LC^{(2\vect{n})}_{\vect{\kappa},\mathfrak{r}}$, $\mathfrak{r}\in\{0,1\}$. In particular, it is a bijection from  $\I^{(2\vect{n})}_{\vect{\kappa}}$ onto $\LC^{(2\vect{n})}_{\vect{\kappa}}$. Using the values from \eqref{v2:1505030732}, we have
\[\#\LC^{(2\vect{n})}_{\vect{\kappa}}=\#\I^{(2\vect{n})}_{\vect{\kappa}},\quad \#\LC^{(2\vect{n})}_{\vect{\kappa},\mathfrak{r}}=\#\I^{(2\vect{n})}_{\vect{\kappa},\mathfrak{r}}.\]

Illustrations of the sets $\LC^{(2\vect{n})}_{\vect{\kappa},0}$ for the dimensions $\mathsf{d} = 2$ and $\mathsf{d} = 3$ are given
in Figure \ref{v2:fig:lissajous2} and \ref{v2:fig:lissajous3}, respectively. For $\mathsf{M}\subseteq\{1,\ldots,\mathsf{d}\}$, we further introduce the sets
\[\I^{(2\vect{n})}_{\vect{\kappa},\mathsf{M}}=\I^{(2\vect{n})}_{\vect{\kappa},\mathsf{M},0}\cup \I^{(2\vect{n})}_{\vect{\kappa},\mathsf{M},1},\quad 
\I^{(2\vect{n})}_{\vect{\kappa},\mathsf{M},\mathfrak{r}}=\left\{\,\left.\vect{i}\in \I^{(2\vect{n})}_{\vect{\kappa},\mathfrak{r}}\,\right|\,0<i_{\mathsf{j}}<2n_{\mathsf{j}}\Leftrightarrow\mathsf{j}\in \mathsf{M}\,\right\}\]
and
\[\LC^{(2\vect{n})}_{\vect{\kappa},\mathsf{M}} =\LC^{(2\vect{n})}_{\vect{\kappa},\mathsf{M},0}\cup\LC^{(2\vect{n})}_{\vect{\kappa},\mathsf{M},1},\quad \LC^{(2\vect{n})}_{\vect{\kappa},\mathsf{M},\mathfrak{r}}=\left\{\, \vect{z}^{(2\vect{n})}_{\vect{\kappa},\vect{i}}\,\left|\,\vect{i}\in \I^{(2\vect{n})}_{\vect{\kappa},\mathsf{M},\mathfrak{r}} \right.\right\}.\]
Clearly,
$\LC^{(2\vect{n})}_{\vect{\kappa},\mathsf{M}}=\LC^{(2\vect{n})}_{\vect{\kappa}}\cap \vect{F}^{\mathsf{d}}_{\mathsf{M}}$ and 
$\LC^{(2\vect{n})}_{\vect{\kappa},\mathsf{M},\mathfrak{r}}=\LC^{(2\vect{n})}_{\vect{\kappa},\mathfrak{r}}\cap \vect{F}^{\mathsf{d}}_{\mathsf{M}}$. 

\medskip

Similar as Proposition \ref{v2:1509021542} in the last section, the following Proposition \ref{v2:1509021556} plays an important role in the proofs of the upcoming results. 
In the present case, we obtain an identification of the set $\I^{(2\vect{n})}_{\vect{\kappa}}$ with a particular class decomposition of
\[\vect{H}^{(2\vect{n})}=\{0,\ldots,4\p[\vect{n}]-1\}\times  \{0,1\}^{\mathsf{d}-1}.\]
The deeper importance of this result gets apparent later in Proposition \ref{v2:1508311218}.

\begin{proposition}\label{v2:1509021556}
We assume that
\begin{equation}\label{v2:1508071340}
\text{$\mathsf{g}\in\{1,\ldots,\mathsf{d}\}$ and that $n_{\mathsf{i}}$ is odd for all  $\mathsf{i}\in\{1,\ldots,\mathsf{d}\}\setminus\{\mathsf{g}\}$.}
\end{equation}

a) For all\, $(l,\rho_1,\ldots,\rho_{\mathsf{g}-1},\rho_{\mathsf{g}+1},\ldots,\rho_{\mathsf{d}})\in \vect{H}^{(2\vect{n})}$, there exists a uniquely determined  element  $\vect{i}\in \I^{(2\vect{n})}_{\vect{\kappa}}$ and a  (not necessarily  unique) $\vect{v}\in\{-1,1\}^{\mathsf{d}}$ such that 
\begin{equation}\label{v2:15082921}
\begin{array}{rll}
 i_{\mathsf{g}}&\hspace{-5pt} \equiv v_{\mathsf{g}}\left(l-\kappa_{\mathsf{g}}\right) &\mod 4n_{\mathsf{g}},\\
\forall\,\mathsf{i}\ \text{with}\ \mathsf{i}\neq\mathsf{g}:\quad i_{\mathsf{i}}&\!\!\!\!\;\equiv v_{\mathsf{i}}\left(l+2\rho_{\mathsf{i}}n_{\mathsf{i}}-\kappa_{\mathsf{i}}\right) &\mod 4n_{\mathsf{i}}.
\end{array}
\end{equation}
Therefore, a function $\vect{i}^{(2\vect{n})}_{\vect{\kappa}}:\,\vect{H}^{(2\vect{n})}\to\I^{(2\vect{n})}_{\vect{\kappa}}$ is well defined by
\begin{equation}\label{v2:1509022008}
\vect{i}^{(2\vect{n})}_{\vect{\kappa}}(l,\rho_1,\ldots,\rho_{\mathsf{g}-1},\rho_{\mathsf{g}+1},\ldots,\rho_{\mathsf{d}})=\vect{i}.
\end{equation}

b) We have $\vect{i}^{(2\vect{n})}_{\vect{\kappa}}(l,\rho_1,\ldots,\rho_{\mathsf{g}-1},\rho_{\mathsf{g}+1},\ldots,\rho_{\mathsf{d}})\in \I^{(2\vect{n})}_{\vect{\kappa},\mathfrak{r}}$ if and only if $l\equiv\mathfrak{r}\tmod 2$.

c) Let $\mathsf{M}\subseteq \{1,\ldots,\mathsf{d}\}$. If  $\vect{i}\in \I^{(2\vect{n})}_{\vect{\kappa},\mathsf{M}}$, then $\#\{\,\vect{h}\in \vect{H}^{(2\vect{n})}\,|\,\vect{i}^{(2\vect{n})}_{\vect{\kappa}}(\vect{h})=\vect{i}\,\}=2^{\#\mathsf{M}}$.
\end{proposition}

 \begin{figure}[htb]
	\centering
	\subfigure[\hspace*{0.5em} \text{$\LC^{(10,6)}_{\vect{0}}$ and \, $\vect{\ell}^{(10,6)}_{\vect{0},\vect{1}}({[0,\pi]}) \cup \vect{\ell}^{(10,6)}_{\vect{0},(1,-1)}({[0,\pi]})$}
	]{\includegraphics[scale=0.8]{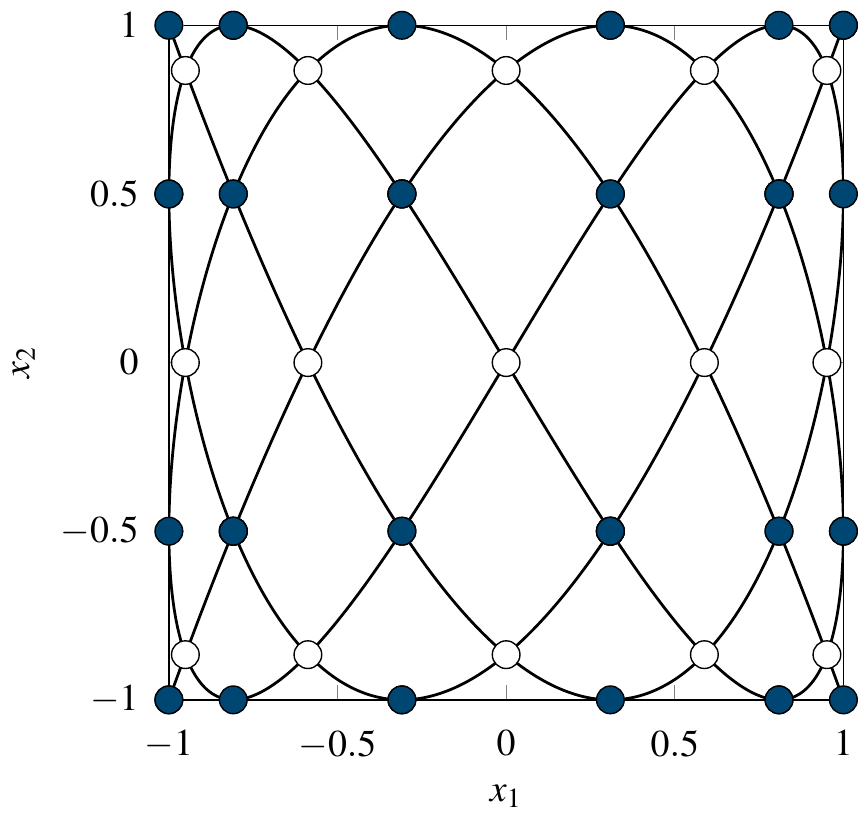}}
	\hfill	
	\subfigure[\hspace*{0.5em} $\LC^{(10,6)}_{(0,1)}$ and \;\!\! $\mathcal{C}^{(10,6)}_{(0,1)} = \vect{\ell}^{(10,6)}_{(0,1),\vect{1}}({[0,2\pi)})$
	]{\includegraphics[scale=0.8]{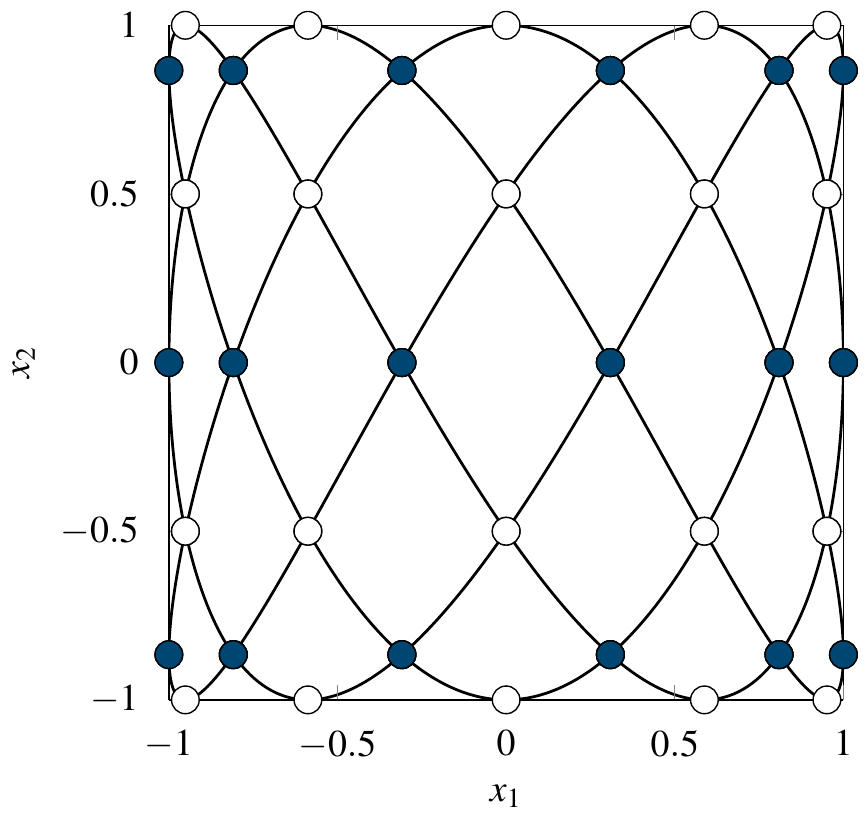}} 
  	\caption{Illustration of the set $\LC^{(2 \vect{n})}_{\vect{\kappa}}$ and the Chebyshev variety $\mathcal{C}^{(2\vect{n})}_{\vect{\kappa}}$ for 
  	$\mathsf{d} = 2$. The subsets $\LC^{(2\vect{n})}_{\vect{\kappa},0}$ and $\LC^{(2\vect{n})}_{\vect{\kappa},1}$ are colored in blue and 
  	white, respectively. 
  	}
	\label{v2:fig:lissajous2}
\end{figure}

\begin{remark} We can always find an index $\mathsf{g}$ satisfying \eqref{v2:1508071340}.
If all entries of $\vect{n}$ are odd, then 
\eqref{v2:1508071340} is satisfied for all $\mathsf{g}\in\{1,\ldots,\mathsf{d}\}$. On the other hand, if there is a (necessarily unique) even entry of 
$\vect{n}$, then $\mathsf{g}$ is the index of the even entry $n_{\mathsf{g}}$.
\end{remark}

\begin{proof} The statements a) and b) are shown in the same way as in the proof of Proposition \ref{v2:1509021542}. 
We turn to statement c) and denote $k_{\mathsf{g}}=4n_{\mathsf{g}}$ and  $k_{\mathsf{i}}=2n_{\mathsf{i}}$ for $\mathsf{i}\neq\mathsf{g}$. If $\vect{i}\in\I^{(2\vect{n})}_{\vect{\kappa}}$ and $\vect{v}\in \{-1,1\}^{\mathsf{d}}$, then for $a_{\mathsf{i}}=v_{\mathsf{i}}i_{\mathsf{i}}+\kappa_{\mathsf{i}}$  condition \eqref{v2:1503211222} is valid 
by the definition of $\I^{(2\vect{n})}_{\vect{\kappa}}$. The Chinese remainder theorem implies the existence of a unique  $l\in \{0,\ldots,4\p[\vect{n}]-1\}$ satisfying
\[\begin{array}{rll}
 v_{\mathsf{g}}i_{\mathsf{g}}+\kappa_{\mathsf{g}}&\!\!\!\equiv l &\mod 4n_{\mathsf{g}},\\
\forall\,\mathsf{i}\ \text{with}\ \mathsf{i}\neq\mathsf{g}:\quad v_{\mathsf{i}}i_{\mathsf{i}}+\rule{1pt}{0pt}\kappa_{\mathsf{i}}\rule{1pt}{0pt}&\!\!\!\equiv l &\mod 2n_{\mathsf{i}}.
\end{array}\]
From this we find  uniquely determined $(\rho_1,\ldots,\rho_{\mathsf{g}-1},\rho_{\mathsf{g}+1},\ldots,\rho_{\mathsf{d}})\in\{-1,1\}^{\mathsf{d}-1}$ such that \eqref{v2:15082921} holds.
Therefore, for each  $\vect{i}\in \I^{(2\vect{n})}_{\vect{\kappa}}$ a function $\vect{h}^{(2\vect{n})}_{\vect{i}}:\,\{-1,1\}^{\mathsf{d}}\to \vect{H}^{(2\vect{n})}$ is well defined by $\vect{h}^{(2\vect{n})}_{\vect{i}}(\vect{v})=(l,\rho_1,\ldots,\rho_{\mathsf{g}-1},\rho_{\mathsf{g}+1},\ldots,\rho_{\mathsf{d}})$ satisfying \eqref{v2:15082921}. We have $\vect{i}^{(2\vect{n})}_{\vect{\kappa}}(\vect{h})=\vect{i}$ if and only if there is a $\vect{v}\in\{-1,1\}^{\mathsf{d}}$ such that $\vect{h}=\vect{h}^{(2\vect{n})}_{\vect{i}}(\vect{v})$.

If $\vect{i}\in \I^{(2\vect{n})}_{\vect{\kappa},\mathsf{M}}$, then $\vect{h}^{(2\vect{n})}_{\vect{i}}(\vect{v}')=\vect{h}^{(2\vect{n})}_{\vect{i}}(\vect{v})$ if and only if
$(v'_{\mathsf{i}}-v_{\mathsf{i}})i_{\mathsf{i}}\equiv 0 \tmod 4n_{\mathsf{i}}$ for all $\mathsf{i}\in \mathsf{M}$, 
i.e. if and only if   $v'_{\mathsf{i}}=v_{\mathsf{i}}$ for all $\mathsf{i}\in\mathsf{M}$.
\end{proof}

\newpage

\begin{theorem}  For $\mathfrak{r}\in \{0,1\}$, we have 
\begin{align} \label{v2:B150915132} \LC^{(2\vect{n})}_{\vect{\kappa},\mathfrak{r}}&
=\left\{\,\vect{\ell}^{(2\vect{n})}_{\vect{\kappa},\vect{u}}(t^{(2\vect{n})}_{l})\,\left|\,\vect{u}\in\{-1,1\}^{\mathsf{d}},\,l\in\{0,\ldots,2\p[\vect{n}]-1\} \! : l\equiv \mathfrak{r}\tmod 2\right.\right\}\\
&\nonumber
=\left\{\,\vect{\ell}^{(2\vect{n})}_{\vect{\kappa},\vect{u}}(t^{(2\vect{n})}_{l})\,\left|\,\vect{u}\in\{-1,1\}^{\mathsf{d}} \! : u_{\mathsf{g}}=1,\,\ l\in\{0,\ldots,4\p[\vect{n}]-1\}\!: l\equiv \mathfrak{r}\tmod 2\right.\right\} \! .
\end{align}
If $\vect{\kappa} = \vect{0}$, we further have the characterizations
\begin{align}
\LC^{(2\vect{n})}_{\vect{0},0} &= \bigcup_{\substack{\vect{u} \in \{-1,1\}^{\mathsf{d}}: \\u_{\mathsf{g}}=1}} \LD^{(\vect{n})}_{\vect{u}},\quad \text{with \eqref{v2:1509171731}}, \label{v2:C150915132A} \\
\LC^{(2\vect{n})}_{\vect{0},1} &= \left\{\,\vect{\ell}^{(2\vect{n})}_{\vect{0},\vect{1}}(t^{(2\vect{n})}_{l})\,\left|\, l\in\{1,\ldots,2 \p[\vect{n}]-1\}\!: l \equiv 1 \tmod 2 \right.\right\}. \label{v2:C150915132B}
\end{align}
The union \eqref{v2:C150915132A} is a union of pairwise disjoint sets. 
\end{theorem}

\renewcommand{\thesubfigure}{}
 \begin{figure}[htb]
	\centering
	\subfigure[\textrm{(a)}\hspace*{1em} \text{$\LC^{(3,1,2)}$ and \, \!\!\! $\vect{\ell}^{(6,2,4)}_{\vect{0},\vect{1}}({[0,\pi]})$}
	]{\includegraphics[scale=0.8]{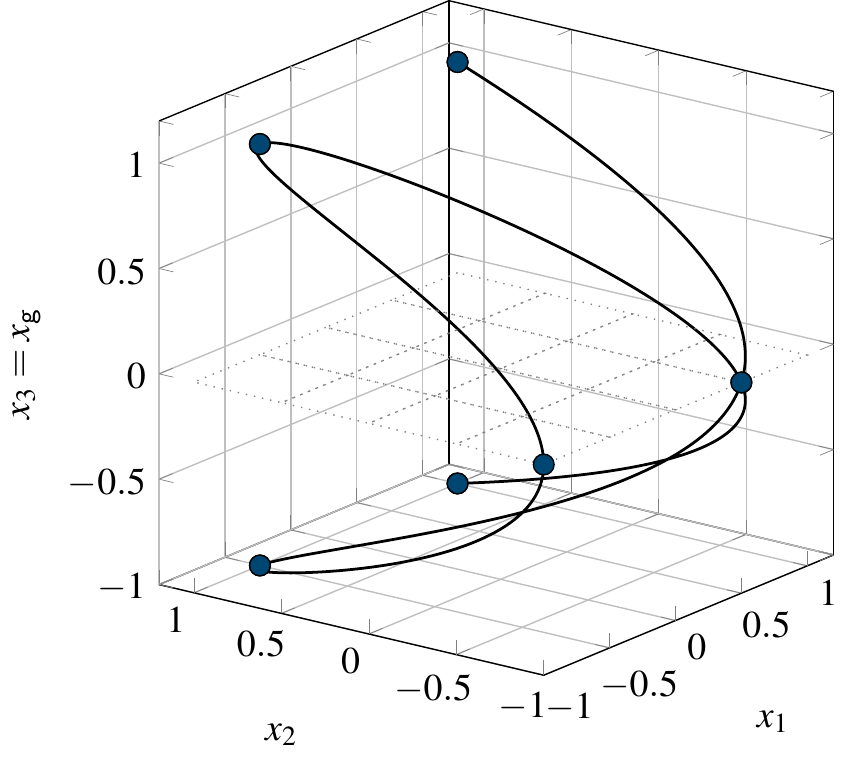}}
	\hfill	
	\subfigure[\hspace*{-4em}\textrm{(b)} \hspace*{1em}\text{$\LC^{(6,2,4)}_{\vect{0}}$ and  $\mathcal{C}^{(6,2,4)}_{\vect{0}} = \hspace{-1.1em} \displaystyle \bigcup_{\vect{u} \in \{-1,1\}^3:\, u_3 = 1} \hspace{-1.2em}
	\vect{\ell}^{(6,2,4)}_{\vect{0},\vect{u}}({[0,\pi]})$}
	]{\includegraphics[scale=0.8]{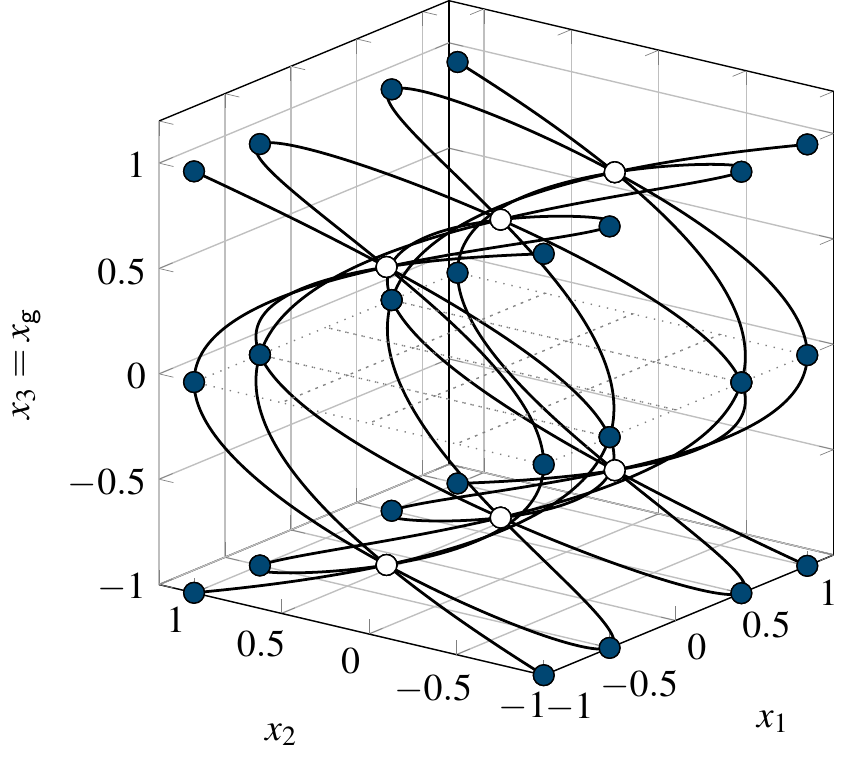}} 
  	\caption{Illustration of the set $\LC^{(2 \vect{n})}_{\vect{0}}$ and the Chebyshev variety $\mathcal{C}^{(2\vect{n})}_{\vect{0}}$ for
  	$\mathsf{d} = 3$. In (b), the subsets $\LC^{(2\vect{n})}_{\vect{0},0}$ and $\LC^{(2\vect{n})}_{\vect{0},1}$ are colored in blue and 
  	white, respectively. 
  	}
	\label{v2:fig:lissajous3}
\end{figure}
\renewcommand{\thesubfigure}{(\alph{subfigure})}

\medskip

\begin{proof} Note that $\p_{\mathsf{g}}[\vect{n}]$ is odd.
If $l'=l+2\p[\vect{n}]$ or $l'=l-2\p[\vect{n}]$, then we have $\vect{\ell}^{(2\vect{n})}_{\vect{\kappa},\vect{u}}(t^{(2\vect{n})}_{l'})=\vect{\ell}^{(2\vect{n})}_{\vect{\kappa},\vect{u}'}(t^{(2\vect{n})}_{l})$ with some $\vect{u}'\in\{-1,1\}^{\mathsf{d}}$ satisfying $u'_{\mathsf{g}}=-u_{\mathsf{g}}$. Therefore,  the  sets on the right hand side of \eqref{v2:B150915132} are equal.
If \eqref{v2:1509022008} holds, then $\vect{\ell}^{(2\vect{n})}_{\vect{\kappa},\vect{u}}(t^{(2\vect{n})}_{l})=\vect{z}^{(2\vect{n})}_{\vect{i}}$ with $u_{\mathsf{i}}=(-1)^{\rho_{\mathsf{i}}}$ for $\mathsf{i}\neq \mathsf{g}$ and  $u_{\mathsf{g}}=1$. Now, Proposition \ref{v2:1509021556} yields the identity \eqref{v2:B150915132}.

We have the identity $\vect{\ell}^{(2\vect{n})}_{\vect{0},\vect{u}}(t^{(2\vect{n})}_{2l}) = \vect{\ell}^{(\vect{n})}_{\vect{0},\vect{u}}(t^{(\vect{n})}_{l})$. Therefore, the characterization \eqref{v2:B150915132} together with the definition \eqref{v2:1509171731} of the sets $\LD^{(\vect{n})}_{\vect{u}}$ gives \eqref{v2:C150915132A}. 

Let $\vect{i} \in \I^{(2\vect{n})}_{\vect{0},1}$. There are
$j_{\mathsf{i}}$ such that
$i_{\mathsf{i}}=2j_{\mathsf{i}}+1$ for all $\mathsf{i}$. Set
$v_{\mathsf{i}}=1$ if $j_{\mathsf{i}}-j_1$ is even, and
$v_{\mathsf{i}}=-1$ otherwise. If $v_{\mathsf{i}}=1$, we have
$v_{\mathsf{i}}i_{\mathsf{i}}-i_1=2(j_{\mathsf{i}}-j_1)$, and if
$v_{\mathsf{i}}=-1$, we have
$v_{\mathsf{i}}i_{\mathsf{i}}-i_1=-2(j_{\mathsf{i}}-j_1+1)-4j_1$.
Therefore, in any case $v_{\mathsf{i}}i_{\mathsf{i}}-i_1$ is a multiple
of $4$.  Thus, for $a_{\mathsf{i}}=v_{\mathsf{i}}i_{\mathsf{i}}$ and
$k_{\mathsf{i}}=4n_{\mathsf{i}}$ condition \eqref{v2:1503211222} is valid.
The Chinese remainder theorem implies the existence of an $l\in
\{0,\ldots,4\p[\vect{n}]-1\}$ satisfying \eqref{v2:15082921} with
$\rho_{\mathsf{i}}=0$ for all $\mathsf{i}\neq \mathsf{g}$ and
$\vect{\kappa}=\vect{0}$. If $l<2\p[\vect{n}]$, we set $l'=l$, and
$l'=4\p[\vect{n}]-l$ otherwise.
Then, we have
$\vect{z}^{(2\vect{n})}_{\vect{i}}=\vect{\ell}^{(2\vect{n})}_{\vect{0},\vect{1}}(t^{(2\vect{n})}_{l'})$,
$l'\in \{0,\ldots,2\p[\vect{n}]\}$, and obviously $l'=1\tmod 2$.  Using \eqref{v2:1505030732} for $\vect{\kappa}=\vect{0}$ and the first statement in \eqref{v2:A1505031632}, 
the last statement of the theorem can be seen by a simple counting argument. 
\end{proof}

\begin{remark} \label{v2:201510151703}
The sets $\LC^{(2\vect{n})}_{\vect{\kappa},\mathfrak{r}}$, and hence also $\LC^{(2\vect{n})}_{\vect{\kappa}}$, are invariant under reflections with respect to the hyperplanes $x_{\mathsf{i}}=0$. For the characterization of the set $\LC^{(2\vect{n})}_{\vect{0},1}$, we can replace $\vect{\ell}^{(2\vect{n})}_{\vect{0},\vect{1}}$ in formula \eqref{v2:C150915132B} 
by any other curve $\vect{\ell}^{(2\vect{n})}_{\vect{0},\vect{u}}$, $\vect{u} \in \{-1,1\}^{\mathsf{d}}$. 
Therefore, \eqref{v2:C150915132A}, \eqref{v2:C150915132B} give the following characterizations. A point in $\LC^{(2\vect{n})}_{\vect{0},\mathsf{M}}$ with $2 \leq \# \mathsf{M} < \mathsf{d}$ is a self-intersection
point of exactly one of the curves $\vect{\ell}^{(2\vect{n})}_{\vect{0},\vect{u}}$, $\vect{u} \in \{-1,1\}^{\mathsf{d}}$, $u_\mathsf{g} = 1$. If $ \mathsf{M} = \{ 1, \ldots, \mathsf{d} \}$, then
a point in $\LC^{(2\vect{n})}_{\vect{0},\mathsf{M}}$ is either
a self-intersection point of exactly one of the curves $\vect{\ell}^{(2\vect{n})}_{\vect{0},\vect{u}}$, $\vect{u} \in \{-1,1\}^{\mathsf{d}}$, $u_\mathsf{g} = 1$, or it is an intersection point of all of these $2^{\mathsf{d}-1}$ curves. This can also be seen in Figure \ref{v2:fig:lissajous2}, (a) and \ref{v2:fig:lissajous3}, (b) in which the sets
$\LC^{(2\vect{n})}_{\vect{0},0}$ and $\LC^{(2\vect{n})}_{\vect{0},1}$ are marked with blue and white dots, respectively. 

Statement c) of Proposition \ref{v2:1509021556} implies a similar (but weaker) assertion also for general $\vect{\kappa} \in \mathbb{Z}^{\mathsf{d}}$. Every element of $\LC^{(2\vect{n})}_{\vect{\kappa}, \mathsf{M}}$ with $\# \mathsf{M} \geq 2$ is a self-intersection point of one of the curves 
$\vect{\ell}^{(2\vect{n})}_{\vect{\kappa},\vect{u}}$, $\vect{u} \in \{-1,1\}^{\mathsf{d}}$, $u_\mathsf{g} = 1$, or it is an intersection point of at least two different curves $\vect{\ell}^{(2\vect{n})}_{\vect{\kappa},\vect{u}}$. 
\end{remark}

\vspace{-0.3em}

Similar as for the degenerate Lissajous curves $\vect{\ell}^{(\vect{n})}_{\vect{0},\vect{1}}$ and the 
node sets $\LC^{(\vect{n})}_{\mathfrak{r}}$, we can use a Chebyshev variety to characterize the sets $\LC^{(2\vect{n})}_{\vect{\kappa},\mathfrak{r}}$. To this end,
we define the affine real algebraic Chebyshev variety $\mathcal{C}^{(2\vect{n})}_{\vect{\kappa}} $ in $[-1,1]^{\mathsf{d}}$ by 
\begin{equation} \label{v2:1509222020}
\mathcal{C}^{(2\vect{n})}_{\vect{\kappa}} 
= \left\{\,\vect{x} \in [-1,1]^{\mathsf{d}}\,\left|\, (-1)^{\kappa_1} T_{2n_1}(x_1) = \ldots = (-1)^{\kappa_{\mathsf{d}}} T_{2n_{\mathsf{d}}}(x_{\mathsf{d}}) \right.\right\}.
\end{equation}

We obtain the following relation between the affine Chebyshev variety $\mathcal{C}^{(2\vect{n})}_{\vect{\kappa}}$, its singularities and the Lissajous curves 
$\vect{\ell}^{(2\vect{n})}_{\vect{\kappa},\vect{u}}$, $\vect{u} \in \{-1,1\}^{\mathsf{d}}$. 

\vspace{-0.5em}

\begin{theorem} \label{v2:201510121616}
The affine Chebyshev variety $\mathcal{C}^{(2\vect{n})}_{\vect{\kappa}}$ can be written as 
\begin{equation} \label{v2:AA1508251536} 
\mathcal{C}^{(2\vect{n})}_{\vect{\kappa}} = \textstyle\bigcup\limits_{\vect{u} \in \{-1,1\}^{\mathsf{d}}} \vect{\ell}^{(2\vect{n})}_{\vect{\kappa},\vect{u}}([0,\pi))
= \textstyle\bigcup\limits_{\substack{\vect{u} \in \{-1,1\}^{\mathsf{d}}:\\ u_{\mathsf{g}}=1}} \vect{\ell}^{(2\vect{n})}_{\vect{\kappa},\vect{u}}([0,2\pi)). 
\end{equation}

\vspace{-0.4em}

A point $\vect{x}\in \mathcal{C}^{(2\vect{n})}_{\vect{\kappa}}$ is a singular point of $\mathcal{C}^{(2\vect{n})}_{\vect{\kappa}}$ if and only if it is an element of $\LC^{(2\vect{n})}_{\vect{\kappa},\mathsf{M}}$ for $\mathsf{M}\subseteq \{1,\ldots,\mathsf{d}\}$ with $\#\mathsf{M}\geq 2$. Furthermore, for $\mathfrak{r}\in \{0,1\}$ 
we have
\begin{equation}  \label{v2:A1508251536} 
\LC^{(2\vect{n})}_{\vect{\kappa},\mathsf{M},\mathfrak{r}} =\left\{ \, \vect{x} \in \vect{F}^{\mathsf{d}}_{\mathsf{M}}\,\left|\, (-1)^{\kappa_1}T_{2n_1}(x_1) = \ldots = (-1)^{\kappa_{\mathsf{d}}}T_{2n_{\mathsf{d}}}(x_{\mathsf{d}}) = (-1)^{\mathfrak{r}} \,\right.\right\}.
\end{equation}
\end{theorem}
The formula \eqref{v2:A1508251536} is satisfied for all $\mathsf{M}\subseteq \{1,\ldots,\mathsf{d}\}$. 
If $\#\mathsf{M}\in \{0,1\}$, the elements of the sets are regular points of the variety on the corners and edges of $[-1,1]^{\mathsf{d}}$.

\medskip

\begin{proof}
Recall that $\p_{\mathsf{g}}[\vect{n}]$ is odd. Thus, if $t'=t+\pi$ or $t'=t-\pi$, we have $\vect{\ell}^{(2\vect{n})}_{\vect{\kappa},\vect{u}}(t')=\vect{\ell}^{(2\vect{n})}_{\vect{\kappa},\vect{u}'}(t)$ with some $\vect{u}'\in\{-1,1\}^{\mathsf{d}}$ satisfying $u'_{\mathsf{g}}=-u_{\mathsf{g}}$. Therefore, the  sets on the right hand side of \eqref{v2:AA1508251536} are both equal to $\{ \vect{\ell}^{(2\vect{n})}_{\vect{\kappa},\vect{u}}(t)\,|\; \vect{u} \in \{-1,1\}^{\mathsf{d}}, \; t \in [0,2\pi)\}$.

Let $\vect{x}\in \mathcal{C}^{(2\vect{n})}_{\vect{\kappa}}$. We choose a $\theta_{\mathsf{i}}\in [0,\pi]$ and $t'\in \mathbb{R}$ such that 
$x_{\mathsf{i}}=\cos(\theta_{\mathsf{i}})$ and $(-1)^{\kappa_{\mathsf{i}}} T_{2n_{\mathsf{i}}}(x_{\mathsf{i}})=\cos(2\p[\vect{n}]t')$ for all $\mathsf{i}$. Then, there exist $\vect{v}\in \{-1,1\}^{\mathsf{d}}$ and $\vect{h}'\in\mathbb{Z}$ such that $2n_{\mathsf{i}}\theta_{\mathsf{i}}+\kappa_{\mathsf{i}}\pi= v_{\mathsf{i}}2\p[\vect{n}]t'+2 h'_{\mathsf{i}}\pi$ for all $\mathsf{i}$. Using $h_{\mathsf{i}}=v_{\mathsf{i}}h'_{\mathsf{i}}+(1-v_{\mathsf{i}})\kappa_{\mathsf{i}}/2\in\mathbb{Z}$, we have
\[x_{\mathsf{i}}=\cos(\p_{\mathsf{i}}[\vect{n}]t'+h_{\mathsf{i}}\pi /n_{\mathsf{i}}-\kappa_{\mathsf{i}}\pi/(2n_{\mathsf{i}}))\quad \text{for all $\mathsf{i}$}.\] By the Chinese remainder theorem, we can find an $l$ such that $l\equiv h_{\mathsf{i}}\tmod n_{\mathsf{i}}$. Then, for  $t=t'+l\pi/\p[\vect{n}]$  
we have $x_{\mathsf{i}}=(-1)^{\rho_{\mathsf{i}}}\cos(\p_{\mathsf{i}}[\vect{n}]t-\kappa_{\mathsf{i}}\pi/(2n_{\mathsf{i}}))$  
with $\rho_{\mathsf{i}}=(l-h_{\mathsf{i}})/n_{\mathsf{i}}$ for all $\mathsf{i}$. Therefore, $\vect{x} \in \vect{\ell}^{(2\vect{n})}_{\vect{\kappa}, \vect{u}}(\mathbb{R})=\vect{\ell}^{(2\vect{n})}_{\vect{\kappa}, \vect{u}}([0,2\pi))$ for an $\vect{u} \in 
\{-1,1\}^{\mathsf{d}}$. The relation $\vect{\ell}^{(2\vect{n})}_{\vect{\kappa},\vect{u}}([0,2\pi))\subseteq \mathcal{C}^{(2\vect{n})}_{\vect{\kappa}}$ is easily verified
by inserting $\vect{\ell}^{(2\vect{n})}_{\vect{\kappa},\vect{u}}(t)$ in the definition \eqref{v2:1509222020}.\\
\indent The remaining part of the proof follows the lines of the proof of Theorem \ref{v2:1508251450}.
\end{proof}

\begin{example}

\begin{enumerate}
\item[(i)]
In \cite{ErbKaethnerAhlborgBuzug2015}, non-degenerate Lissajous curves of the form
\begin{equation}\label{v2:1508101433}
\vect{\ell}^{(2n+2p,2n)}_{(n+p,n),\vect{1}}(t) = \left( \sin(n t), \sin( (n+p) t) \right), \quad n,p \in \mathbb{N},
\end{equation}
were considered, where $n$ and $p$ are assumed to be relatively prime. The curve \eqref{v2:1508101433} is non-degenerate if and only if $p$ is odd. In this case, the curve $\vect{\ell}^{(2n+2p,2n)}_{(n+p,n),\vect{1}}$ is invariant with respect to reflections at both coordinate axis and we have $\vect{\ell}^{(2n+2p,2n)}_{(n+p,n),\vect{1}} = \mathcal{C}^{(2n+2p,2n)}_{(n+p,n)}$. The respective set of node
points is $\LC^{(2n+2p,2n)}_{(n+p,n)}$. Furthermore, we have $\LC^{(2n+2p,2n)}_{(n+p,n)} = \LC^{(2n+2p,2n)}_{(0,1)}$. 

\item[(ii)]
If $n$ and $p$ are relatively prime and, in addition, $n$ is odd and $p$ is even, a generating curve for the set $\LC^{(2n+2p,2n)}_{(n+p,0)} = \LC^{(2n+2p,2n)}_{(0,1)}$ is given by 
\[ \vect{\ell}^{(2n+2p,2n)}_{(n+p,0),\vect{1}}(t) = \left( \sin(n t), \cos( (n+p) t) \right).\]
Again, the curves $\vect{\ell}^{(2n+2p,2n)}_{(n+p,0),\vect{1}}$ and the points $\LC^{(2n+2p,2n)}_{(n+p,0)}$ are invariant with respect to reflections at both coordinate axis. For $\vect{n} = (5,3)$, this case is illustrated in Figure \ref{v2:fig:lissajous2}, (b). This example was already considered in \cite{ErbKaethnerDenckerAhlborg2015}. 

\item[(iii)] An interesting two-dimensional example not considered in \cite{Erb2015,ErbKaethnerAhlborgBuzug2015,ErbKaethnerDenckerAhlborg2015} is given by the point set
$\LC^{(2 \vect{n})}_{\vect{0}}$ with $\vect{n} = (n_1, n_2)$, and $n_1, n_2$ relatively prime. We choose $\mathsf{g}$ according to \eqref{v2:1508071340} and set $\vect{u} = (1,-1)$ if $\mathsf{g} = 1$ and $\vect{u} = (-1,1)$ if $\mathsf{g} = 2$. By Theorem \ref{v2:201510121616}, the set $\mathcal{C}^{(2 \vect{n})}_{\vect{0}}$ corresponds to the union of the Lissajous curves $\vect{\ell}^{(2\vect{n})}_{\vect{0},\vect{1}}$ and
$\vect{\ell}^{(2 \vect{n})}_{\vect{0},\vect{u}}$. Further, by Remark \ref{v2:201510151703}, the subset $\LC^{(2\vect{n})}_{\vect{0},0}$ is the disjoint union of 
$\LD_{\vect{1}}^{(\vect{n})}$ and $\LD_{\vect{u}}^{(\vect{n})}$, and the subset $\LC^{(2\vect{n})}_{\vect{0},1}$ is precisely the set of all intersection points of
the curve $\vect{\ell}^{(2\vect{n})}_{\vect{0},\vect{1}}$ with the curve
$\vect{\ell}^{(2 \vect{n})}_{\vect{0},\vect{u}}$. A corresponding example with $\vect{n} = (5,3)$ and $\mathsf{g} = 1$ is illustrated in Figure \ref{v2:fig:lissajous2}, (a).
\end{enumerate}

\end{example}

\subsection{Discrete orthogonality structure}

Similar as in \eqref{v2:A1508291531}, we define for $\vectgamma\in \mathbb{N}_0^{\mathsf{d}}$
 the functions $\dchi^{(2\vect{n})}_{\indexvectgamma}\in \mathcal{L}(\I^{(2\vect{n})}_{\vect{\kappa}})$ by
\[\dchi^{(2\vect{n})}_{\indexvectgamma}(\vect{i})=\tprod_{\mathsf{j}=1}^{\mathsf{d}}\cos(\gamma_{\mathsf{j}}i_{\mathsf{j}}\pi/(2n_{\mathsf{j}})).\]

\vspace{-0.5em}

We remark that the considered domain $\I^{(2\vect{n})}_{\vect{\kappa}}$ depends on $\vect{\kappa}$ and therefore also the functions $\dchi^{(2\vect{n})}_{\indexvectgamma}$. We omit the explicit indication of this dependency. 

For  $\vect{i}\in \I^{(2\vect{n})}_{\vect{\kappa}}$, we define  the weights $\mathfrak{w}^{(2\vect{n})}_{\vect{\kappa},\vect{i}}$ by 
\[\mathfrak{w}^{(2\vect{n})}_{\vect{\kappa},\vect{i}}=2^{\#\mathsf{M}}/(2^{\mathsf{d+1}} \p[\vect{n}]) \quad \text{if}\ \ \vect{i}\in\I^{(2\vect{n})}_{\vect{\kappa},\mathsf{M}}.\]
Further, a   measure $\omega^{(2\vect{n})}_{\vect{\kappa}}$ on the power set $\mathcal{P}(\I^{(2\vect{n})}_{\vect{\kappa}})$  of $\I^{(2\vect{n})}_{\vect{\kappa}}$  is
well-defined by the corresponding values $\omega^{(2\vect{n})}_{\vect{\kappa}}(\{\vect{i}\})=\mathfrak{w}^{(2\vect{n})}_{\vect{\kappa},\vect{i}}$ for the one-element sets $\{\vect{i}\}\in\mathcal{P}(\I^{(2\vect{n})}_{\vect{\kappa}})$.

\vspace{-0.3em}

\begin{proposition}\label{v2:1508311218}
Let  $\vectgamma\in\mathbb{Z}^{\mathsf{d}}$ and $\displaystyle\dchi^{(2\vect{n})}_{\indexvectgamma}\in \mathcal{L}(\I^{(2\vect{n})}_{\vect{\kappa}})$. If $\tint\dchi^{(2\vect{n})}_{\indexvectgamma}\mathrm{d}\rule{1pt}{0pt}\omega^{(2\vect{n})}_{\vect{\kappa}}\neq 0$, then
\begin{equation}\label{v2:1507201241}
\text{there exists $\vect{h}\in \mathbb{N}_0^{\mathsf{d}}$ with $\gamma_{\mathsf{i}}=2h_{\mathsf{i}}n_{\mathsf{i}}$, $\mathsf{i}=1,\ldots,\mathsf{d}$,  and
$\tsum_{\mathsf{i}=1}^{\mathsf{d}}h_{\mathsf{i}}\in 2\mathbb{N}_0$}.
\end{equation}

\vspace{-0.5em}

If \eqref{v2:1507201241} is satisfied, then
$\tint\dchi^{(2\vect{n})}_{\indexvectgamma}\mathrm{d}\rule{1pt}{0pt}\omega^{(2\vect{n})}_{\vect{\kappa}}=(-1)^{\vartheta_{\vect{\kappa}}(\indexvectgamma)}$ with $\vartheta_{\vect{\kappa}}(\vectgamma)=\tsum_{\mathsf{i}=1}^{\mathsf{d}}h_{\mathsf{i}}\kappa_{\mathsf{i}}$.
\end{proposition}

\begin{proof} 
Note that $\cos(\vartheta+\rho\pi)=(-1)^{\rho}\cos\vartheta$ for $\rho\in\mathbb{Z}$. Using the trigonometric identity \eqref{v2:1507081828} and  
applying Proposition \ref{v2:1509021556} in the same way as Proposition \ref{v2:1509021542} in the proof of Proposition \ref{v2:1507211320}, we obtain for $\tint\dchi^{(2\vect{n})}_{\indexvectgamma}\mathrm{d}\rule{1pt}{0pt}\omega^{(2\vect{n})}_{\vect{\kappa}}$ the value
\begin{equation}\label{v2:A1506171253}
\frac{S(\vectgamma)}{2^{2\mathsf{d}+1}\p[\vect{n}]}\tsum_{\vect{v}\in \{-1,1\}^{\mathsf{d}}} \tsum_{l=0}^{4\p[\vect{n}]-1}
\cos\left(l \dfrac{\pi}{2\p[\vect{n}]}\tsum_{\mathsf{i}=1}^{\mathsf{d}}v_{\mathsf{i}}\gamma_{\mathsf{i}}\p_{\mathsf{i}}[\vect{n}]\ -\ \vartheta_{\vect{\kappa}}(\vectgamma,\vect{v})\pi\right),
\end{equation}
with $\vartheta_{\vect{\kappa}}(\vectgamma,\vect{v})=v_1\gamma_1\kappa_1/(2n_1)+ \ldots+v_{\mathsf{d}}\gamma_{\mathsf{d}}\kappa_{\mathsf{d}}/(2n_{\mathsf{d}})$, and
\[S(\vectgamma)=\tsum_{(\rho_1,\ldots,\rho_{\mathsf{g}-1},\rho_{\mathsf{g}+1},\ldots,\rho_{\mathsf{d}})\in \{0,1\}^{\mathsf{d}-1}}(-1)^{\Sigma'}\ \ \text{with}\ \
\Sigma'=\tsum_{\substack{\mathsf{i}=1\\\mathsf{i}\neq \mathsf{g}}}^{\mathsf{d}}\gamma_{\mathsf{i}}\rho_{\mathsf{i}}.\]

By the identity \eqref{v2:1506171253}, the value \eqref{v2:A1506171253} is zero if for all  $\vect{v}\in\{-1,1\}^{\mathsf{d}}$  the number \eqref{v2:1507081849} is not an element of  $4\mathbb{Z}$.   Further, if there exists an  $\mathsf{i}\neq \mathsf{g}$ such that the integer  $\gamma_{\mathsf{i}}$ is odd, then  $S(\vectgamma)=0$. 
Assume that $\tint\dchi^{(2\vect{n})}_{\indexvectgamma}\mathrm{d}\rule{1pt}{0pt}\omega^{(2\vect{n})}_{\vect{\kappa}}\neq 0$. 
Then, we can conclude that  $\gamma_{\mathsf{i}}\in 2\mathbb{N}_0$ for all $\mathsf{i}\neq \mathsf{g}$ and that there exists  
a  $\vect{v}\in\{-1,1\}^{\mathsf{d}}$ such that  \eqref{v2:1507081849} is an element of   $4\mathbb{Z}$,  and in particular an element of  $\mathbb{Z}$. Since the  $n_{\mathsf{i}}$, $\mathsf{i}\in \{1,\ldots,\mathsf{d}\}$, 
are pairwise relatively prime, there is a $\vect{h}'\in\mathbb{N}_0^{\mathsf{d}}$ such that $\gamma_{\mathsf{i}}=h'_{\mathsf{i}}n_{\mathsf{i}}$, $\mathsf{i}\in \{1,\ldots,\mathsf{d}\}$.   Since the  $n_{\mathsf{i}}$ are odd for $\mathsf{i}\neq \mathsf{g}$ we have $h'_{\mathsf{i}}\in 2\mathbb{N}_0$ for $\mathsf{i}\neq\mathsf{g}$.
The number 
\eqref{v2:1507081849} equals  $\tsum_{\mathsf{i}=1}^{\mathsf{d}}  v_{\mathsf{i}}h'_{\mathsf{i}}$ and this number is an element of $4\mathbb{Z}$. Therefore,  $h'_{\mathsf{g}}$ is also an even
integer. Thus, we can conclude that $h'_1+\ldots+h'_{\mathsf{d}}\in  4\mathbb{N}_0$. Now, for $\vect{h}=\vect{h}'/2$ we obtain \eqref{v2:1507201241}. 

On the other hand, if \eqref{v2:1507201241} is satisfied, then $\vartheta_{\vect{\kappa}}(\vectgamma,\vect{v})=\tsum_{\mathsf{i}=1}^{\mathsf{d}}v_{\mathsf{i}}h_{\mathsf{i}}\kappa_{\mathsf{i}}\equiv -\vartheta_{\vect{\kappa}}(\vectgamma) \tmod 2$ and $S(\vectgamma)=2^{\mathsf{d}-1}$. Thus, since \eqref{v2:1507081849} is in $4\mathbb{Z}$, the value
\eqref{v2:A1506171253} equals $(-1)^{\vartheta_{\vect{\kappa}}(\indexvectgamma)}$. 
\end{proof}

\medskip

We define the sets (see Fig. \ref{v2:fig:lissajous4})
\[
\vect{\Gamma}^{(2\vect{n})}_{\vect{\kappa}} = \left\{\,\vectgamma\in\mathbb{N}_0^{\mathsf{d}}\ \left|\begin {array}{ll} 
\forall\,\mathsf{i}\in\{1,\ldots,\mathsf{d}\}:&  \!\!\!\gamma_{\mathsf{i}}<2n_{\mathsf{i}}\!\!\\
\forall\,\mathsf{i},\mathsf{j}\ \text{with}\ \mathsf{i}\neq\mathsf{j}:&  \!\!\!\gamma_{\mathsf{i}}/n_{\mathsf{i}}+\gamma_{\mathsf{j}}/n_{\mathsf{j}}\leq  2\!\!\\
\forall\,\mathsf{i},\mathsf{j}\ \text{with}\ \kappa_{\mathsf{i}}\not \equiv \kappa_{\mathsf{j}}\tmod 2:& \!\!\!\gamma_{\mathsf{i}}/n_{\mathsf{i}}+\gamma_{\mathsf{j}}/n_{\mathsf{j}}< 2\!\!
\end{array}\right. 
\right\}\cup\{(0,\ldots,0,2n_{\mathsf{d}})\}.
\]
\begin{proposition}\label{v2:1506250740}
With the values from \eqref{v2:1505030732}, we have
\[
\#\vect{\Gamma}^{(2\vect{n})}_{\vect{\kappa}}= \# \I^{(2\vect{n})}_{\vect{\kappa}}= \#\I^{(2\vect{n})}_{\vect{\kappa},0}+ \# \I^{(2\vect{n})}_{\vect{\kappa},1}.
\] 
\end{proposition}

\begin{proof}
For $\mathfrak{r}\in\{0,1\}$, we consider the subsets
\[\vect{\Gamma}^{(2\vect{n})}_{\vect{\kappa},\mathfrak{r}} = 
\left\{\,\vectgamma\in\mathbb{N}_0^{\mathsf{d}}\left|\begin {array}{l} 
\forall\,\mathsf{i}\ \text{with}\ \kappa_{\mathsf{i}}\equiv \mathfrak{r}\tmod 2:\ \gamma_{\mathsf{i}}/n_{\mathsf{i}}\leq 1,\!\!\\
\forall\,\mathsf{i}\ \text{with}\  \kappa_{\mathsf{i}}\not\equiv \mathfrak{r}\tmod 2:\ \gamma_{\mathsf{i}}/n_{\mathsf{i}}< 1
\!\!\end{array}\right.
\right\}\]
of $\vect{\Gamma}^{(2\vect{n})}_{\vect{\kappa}}$. Further, we use the notation
 \[\vect{\Gamma}^{(2\vect{n}),\textrm{eq}}_{\vect{\kappa}}=\left\{\,\vectgamma\in \vect{\Gamma}^{(2\vect{n})}_{\vect{\kappa}}\,\left|\,\exists\,\mathsf{i}: \, \gamma_{\mathsf{i}}=n_{\mathsf{i}}\right.\right\},\] 
$\vect{\Gamma}^{(2\vect{n}),\textrm{ne}}_{\vect{\kappa}}=\vect{\Gamma}^{(2\vect{n})}_{\vect{\kappa}}\setminus \vect{\Gamma}^{(2\vect{n}),\textrm{eq}}_{\vect{\kappa}}$ and 
 $\vect{\Gamma}^{(2\vect{n}),\textrm{ne}}_{\vect{\kappa},\mathfrak{r}} = \vect{\Gamma}^{(2\vect{n})}_{\vect{\kappa},\mathfrak{r}}\setminus \vect{\Gamma}^{(2\vect{n}),\textrm{eq}}_{\vect{\kappa}}$. We use \eqref{v2:1509051005} with $\vect{m}=2\vect{n}$ in this proof. Since the numbers $n_{\mathsf{i}}$ are pairwise relatively prime, for all $\vectgamma\in \vect{\Gamma}^{(2\vect{n}),\textrm{ne}}_{\vect{\kappa},1}\setminus\{\vect{0}\}$ there exists a $\mathsf{k}$ such that \eqref{v2:1505012254} holds. 
As in the proof of Proposition \ref{v2:1507151430} we conclude that the function $\mathfrak{s}^{(2\vect{n})}$ is a bijection from
$\vect{\Gamma}^{(2\vect{n}),\textrm{ne}}_{\vect{\kappa},1}$ onto $\vect{\Gamma}^{(2\vect{n}),\textrm{ne}}_{\vect{\kappa}}\setminus \vect{\Gamma}^{(2\vect{n}),\textrm{ne}}_{\vect{\kappa},0}$. 
Every  element of $\vect{\Gamma}^{(2\vect{n}),\textrm{eq}}_{\vect{\kappa}}$ is either in 
$\vect{\Gamma}^{(2\vect{n})}_{\vect{\kappa},0}$ or in $\vect{\Gamma}^{(2\vect{n})}_{\vect{\kappa},1}$ and $\mathfrak{s}^{(2\vect{n})}(\vectgamma)=\vectgamma$ for  all $\vectgamma\in \vect{\Gamma}^{(2\vect{n}),\textrm{eq}}_{\vect{\kappa}}$. Therefore,  the function $\mathfrak{s}^{(2\vect{n})}$ is a bijection from
$\vect{\Gamma}^{(2\vect{n})}_{\vect{\kappa},1}$ onto $\vect{\Gamma}^{(2\vect{n})}_{\vect{\kappa}}\setminus \vect{\Gamma}^{(2\vect{n})}_{\vect{\kappa},0}$ and hence we have $\# \vect{\Gamma}^{(2\vect{n})}_{\vect{\kappa},1} = \# \vect{\Gamma}^{(2\vect{n})}_{\vect{\kappa}} - \# \vect{\Gamma}^{(2\vect{n})}_{\vect{\kappa},0}$. From the cross product structure of the set $\vect{\Gamma}^{(2\vect{n})}_{\vect{\kappa},\mathfrak{r}}$, we obtain the cardinality $\# \vect{\Gamma}^{(2\vect{n})}_{\vect{\kappa},\mathfrak{r}}$ as the cardinality $\# \I^{(2\vect{n})}_{\vect{\kappa},\mathfrak{r}}$ given in \eqref{v2:1505030732}. 
\end{proof}

\medskip

In the following, we use the numbers $\mathfrak{e}(\vectgamma)$ given in \eqref{v2:1508251214} and 
\begin{equation}\label{v2:1508251213}
\mathfrak{f}^{(2\vect{n})}(\vectgamma)=\left\{ \begin{array}{cl} 
\#\{\,\mathsf{i}\,|\,\gamma_{\mathsf{i}}=n_{\mathsf{i}}\,\}-1,\; & \text{if}\;\, \exists\,\mathsf{i}\in\{1,\ldots,\mathsf{d}\}:\,\gamma_{\mathsf{i}}=n_{\mathsf{i}}, \\
0, \; & \text{otherwise}.
\end{array} \right. 
\end{equation}

\begin{figure}[htb]
	\centering
	\subfigure[\hspace*{1em} $\vect{\Gamma}^{(10,6)}_{\vect{0}}$ and \;$\vect{\Gamma}^{(10,6)}_{(0,1)} = 
	\vect{\Gamma}^{(10,6)}_{\vect{0}} \setminus \{(5,3)\}$]{\includegraphics[scale=0.8]{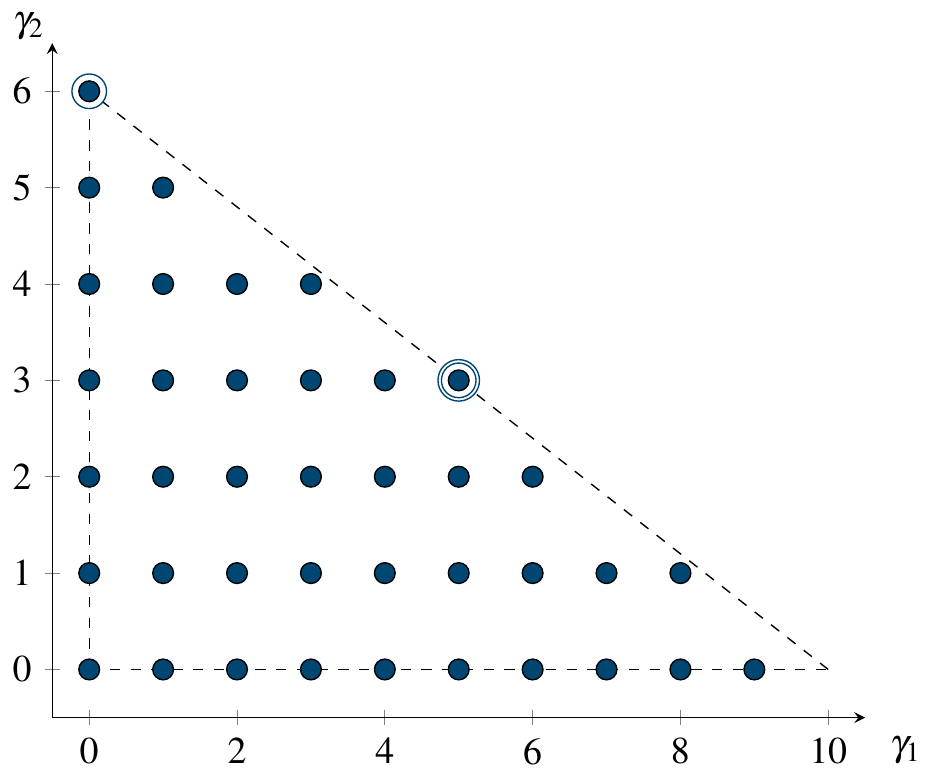}}
	\hfill	
	\subfigure[\hspace*{1em} $\vect{\Gamma}^{(6,2,4)}_{\vect{0}}$]{\includegraphics[scale=0.8]{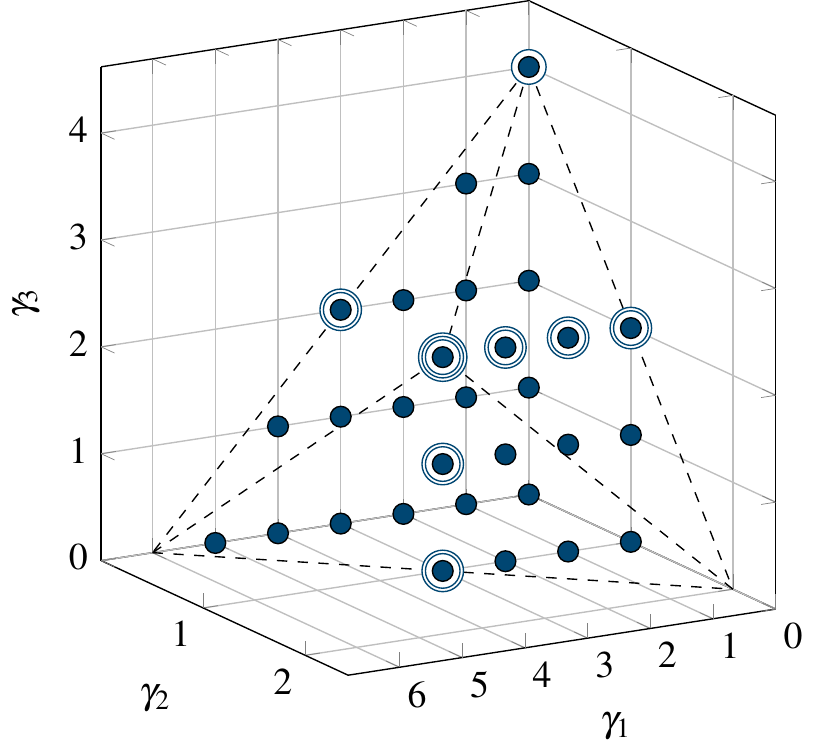}}
  	\caption{Examples of the sets $\vect{\Gamma}^{(2\vect{n})}_{\vect{\kappa}}$ for
  	$\mathsf{d} = 2$ and $\mathsf{d} = 3$ related to the examples $\LC^{(2\vect{n})}_{\vect{\kappa}}$ given in Figure \ref{v2:fig:lissajous2} and \ref{v2:fig:lissajous3}. 
  	The dotted lines mark the boundary of the polytopes enclosing the sets $\vect{\Gamma}^{(2\vect{n})}_{\vect{\kappa}}$. The element 
  	$\{(0,\ldots,0,2 n_{\mathsf{d}})\}$ is marked
with one ring. All dots with two rings indicate the elements $\vectgamma$ with $\mathfrak{f}^{(2\vect{n})}(\vectgamma) = 1$ in \eqref{v2:1508251213}.
The dot with three rings describes the element $\vectgamma$ satisfying $\mathfrak{f}^{(2\vect{n})}(\vectgamma) = 2$.
  	} \label{v2:fig:lissajous4}
\end{figure}

\begin{theorem} \label{v2:thm:210510081555}
The functions $\dchi^{(2\vect{n})}_{\indexvectgamma}$, $\vectgamma\in \vect{\Gamma}^{(2\vect{n})}_{\vect{\kappa}}$, form an  orthogonal basis of the inner product space  $(\mathcal{L}(\I^{(2\vect{n})}_{\vect{\kappa}}),\langle\,\cdot,\cdot\,\rangle_{\omega^{(2\vect{n})}_{\vect{\kappa}}})$.   Further,  the norms of the basis elements satisfy
\begin{equation}\label{v2:A1508221825}
\|\dchi^{(2\vect{n})}_{\indexvectgamma}\|_{\omega^{(2\vect{n})}_{\vect{\kappa}}}^2   =  
 \left\{ \begin{array}{cl}  2^{-\mathfrak{e}(\indexvectgamma)+\mathfrak{f}^{(2\vect{n})}(\indexvectgamma)},\; & \text{if}\quad \vectgamma \in\vect{\Gamma}^{(2\vect{n})}_{\vect{\kappa}}\setminus\{(0,\ldots,0,2n_{\mathsf{d}})\},\\
   1,\; & \text{if}\quad \vectgamma =(0,\ldots,0,2n_{\mathsf{d}}).
\end{array} \right. 
\end{equation}
\end{theorem}

\begin{proof} We fix $\vectgamma',\vectgamma''\in \vect{\Gamma}^{(2\vect{n})}_{\vect{\kappa}}$ and use the notation \eqref{v2:1508251455}.

Assume that $\vectgamma'\neq\vectgamma''$ and that there is $\vect{v}$ such that $\vectgamma=\vectgamma(\vect{v})$ satisfies \eqref{v2:1507201241}. Since  $\vectgamma'\neq \vectgamma''$, the set $\mathsf{P}=\{\,\mathsf{i}\,|\,h_{\mathsf{i}}>0\,\}$ is not empty.  For all $\mathsf{i}\notin \mathsf{P}$ we have $\gamma'_{\mathsf{i}}=\gamma''_{\mathsf{i}}$. As in the proof of Theorem \ref{v2:1507091911} we conclude that both,  
$\vectgamma'$ and $\vectgamma''$,  are different from $(0,\ldots,0,2n_{\mathsf{d}})$, that $h_{\mathsf{i}}=1$ for all $\mathsf{i}\in \mathsf{P}$, that $v_{\mathsf{i}}=1$ for all $\mathsf{i}\in \mathsf{P}$, and that $\#\mathsf{P}\geq 2$.
There is a $\mathsf{k}'\in \mathsf{P}$ such that 
$\gamma'_{\mathsf{k}'}>n_{\mathsf{k}'}$ or $\gamma''_{\mathsf{k}'}>n_{\mathsf{k}'}$, since otherwise 
$\gamma'_{\mathsf{i}}=\gamma''_{\mathsf{i}}=n_{\mathsf{i}}$, $\mathsf{i}\in \mathsf{P}$, and thus $\vectgamma'=\vectgamma''$.
Without loss of generality, we assume $\gamma'_{\mathsf{k}'}>n_{\mathsf{k}'}$. By the definition of $\vect{\Gamma}^{(2\vect{n})}_{\vect{\kappa}}$,
we have $\gamma'_{\mathsf{i}}<n_{\mathsf{i}}$ for all $\mathsf{i}\neq \mathsf{k}'$. Further, there is a $\mathsf{k}''\in \mathsf{P}$ such that $\gamma''_{\mathsf{k}''} > n_{\mathsf{k}''}$, for  otherwise   $\mathsf{P}=\{\mathsf{k}'\}$ and $\#\mathsf{P}=1$. By the definition of $\vect{\Gamma}^{(2\vect{n})}_{\vect{\kappa}}$, we obtain also here $\gamma''_{\mathsf{i}}<n_{\mathsf{k}''}$ for all $\mathsf{i}\neq \mathsf{k}''$. Therefore, $\#\mathsf{P}\leq 2$ and thus $\#\mathsf{P}=2$. Therefore, $\mathsf{P}=\{\mathsf{k}', \mathsf{k}''\}$ with $\mathsf{k}''\neq \mathsf{k}'$ and
we obtain
\[\gamma'_{\mathsf{k}'}/n_{\mathsf{k}'}+\gamma''_{\mathsf{k}'}/n_{\mathsf{k}'}=2,\quad \gamma'_{\mathsf{k}''}/n_{\mathsf{k}''}+\gamma''_{\mathsf{k}''}/n_{\mathsf{k}''}=2. \]
Further, by definition of $\vect{\Gamma}^{(2\vect{n})}_{\vect{\kappa}}$ we have
$\gamma'_{\mathsf{k}'}/n_{\mathsf{k}'}+\gamma'_{\mathsf{k}''}/n_{\mathsf{k}''}\leq 2$ and $\gamma''_{\mathsf{k}'}/n_{\mathsf{k}'}+\gamma''_{\mathsf{k}''}/n_{\mathsf{k}''}\leq 2$.
Therefore, $\gamma'_{\mathsf{k}'}n_{\mathsf{k}''}+\gamma'_{\mathsf{k}''}n_{\mathsf{k}'}= 2 n_{\mathsf{k}''}n_{\mathsf{k}'}$. Since $n_{\mathsf{k}''}$, $n_{\mathsf{k}'}$ are relatively prime,  $\gamma'_{\mathsf{k}'}$ is a multiple of $n_{\mathsf{k}'}$. This is a contradiction to $n_{\mathsf{k}'}<\gamma'_{\mathsf{k}'}< 2n_{\mathsf{k}'}$.
By Proposition \ref{v2:1508311218}, we have thus shown that $\tint\dchi^{(2\vect{n})}_{\indexvectgamma(\vect{v})}\mathrm{d}\rule{1pt}{0pt}\omega^{(2\vect{n})}_{\vect{\kappa}}=0$ for all $\vect{v} \in \{-1,1\}^{\mathsf{d}}$ if $\vectgamma'\neq\vectgamma''$. Similar to \eqref{v2:1507222159}, we have the identity
\begin{equation}\label{v2:1509011446}
\dchi^{(2\vect{n})}_{\indexvectgamma'}\dchi^{(2\vect{n})}_{\indexvectgamma''}=\dfrac1{2^{\mathsf{d}}}\tsum_{\vect{v}\in\{-1,1\}^{\mathsf{d}}}\dchi^{(2\vect{n})}_{\indexvectgamma(\vect{v})}
\end{equation}
and conclude that $\tint\dchi^{(2\vect{n})}_{\indexvectgamma'}\dchi^{(2\vect{n})}_{\indexvectgamma''}\mathrm{d}\rule{1pt}{0pt}\omega^{(2\vect{n})}_{\vect{\kappa}}=0$ if $\vectgamma'\neq\vectgamma''$.

\medskip

On the other hand,  let $\vectgamma'=\vectgamma''$.  If $\vectgamma'=(0,\ldots,0,2n_{\mathsf{d}})$, then $\|\dchi^{(2\vect{n})}_{\indexvectgamma'}\|_{\omega^{(2\vect{n})}_{\vect{\kappa}}}^2=1$. Now, we suppose that $\vectgamma'\neq (0,\ldots,0,2n_{\mathsf{d}})$
and consider the set
\begin{equation}\label{v2:1509011442}
\left\{\,\vect{v}\in\{-1,1\}^{\mathsf{d}}\,|\,\vectgamma = \vectgamma(\vect{v})\ \text{satisfies \eqref{v2:1507201241}}\right\}.
\end{equation}
Using the notation $\mathsf{N}=\{\,\mathsf{i}\,|\,\gamma'_{\mathsf{i}}>0\}$, $\mathsf{M}=\{\,\mathsf{i}\,|\,\gamma'_{\mathsf{i}} = n_{\mathsf{i}}\}$,
 $\mathsf{M}(\vect{v})=\{\,\mathsf{i}\in \mathsf{M}\,|\,v_{\mathsf{i}}=1\}$, the set \eqref{v2:1509011442} can be reformulated as
\begin{equation}
\label{v2:1508251259}
\{\,\vect{v}\in\{-1,1\}^{\mathsf{d}}\,|\,\text{$v_{\mathsf{i}}=-1$ for all $\mathsf{i}\in \mathsf{N}\setminus \mathsf{M}$ and $\#\mathsf{M}(\vect{v})$ is even}\,\}.
\end{equation}
The number $A$ of elements in \eqref{v2:1508251259} is given by
\[A=2^{\mathsf{d}-\#\mathsf{N}}\tsum_{\substack{m=0\\m\equiv 0\tmod 2}}^{\#\mathsf{M}}\displaystyle\binom{\#\mathsf{M}}{m}\]
and this number equals 
$A=2^{\mathsf{d}-\#\mathsf{N}}$ if $\mathsf{M}=\emptyset$ and $A=2^{\mathsf{d}-\#\mathsf{N}}2^{\#\mathsf{M}-1}$ otherwise.

\medskip

Now, let $\vect{v}$ be an element of \eqref{v2:1508251259}. We use the notation $\vartheta_{\vect{\kappa}}(\vectgamma)$ given in Proposition \ref{v2:1508311218} and obtain $\vartheta_{\vect{\kappa}}(\vectgamma(\vect{v}))=\tsum_{\mathsf{i}\in \mathsf{M}(\vect{v})}\kappa_{\mathsf{i}}$. By definition of $\vect{\Gamma}^{(2\vect{n})}_{\vect{\kappa}}$ we have $\kappa_{\mathsf{i}}\equiv\kappa_{\mathsf{j}}\tmod 2$ for all $\mathsf{i},\mathsf{j}\in \mathsf{M}$. Since  
 $\#\mathsf{M}(\vect{v})$ is even, the integer  $\vartheta_{\vect{\kappa}}(\vectgamma(\vect{v}))$ is even, and  by Proposition \ref{v2:1508311218} we have $ \tint \dchi^{(2\vect{n})}_{\indexvectgamma(\vect{v})} \mathrm{d}\rule{1pt}{0pt}\omega^{(2\vect{n})}_{\vect{\kappa}} =1$. 
Using \eqref{v2:1509011446}, we conclude $\|\dchi^{(2\vect{n})}_{\indexvectgamma'}\|_{\omega^{(2\vect{n})}_{\vect{\kappa}}}^2=A/2^{\mathsf{d}}$. 

The just shown formula \eqref{v2:A1508221825} implies that
$\dchi^{(2\vect{n})}_{\indexvectgamma}\neq 0$ for all $\vectgamma\in \vect{\Gamma}^{(2\vect{n})}_{\vect{\kappa}}$. 
Since the functions $\dchi^{(2\vect{n})}_{\indexvectgamma}$, $\vectgamma \in \Gamma^{(2\vect{n})}_{\vect{\kappa}}$, are pairwise orthogonal, 
they are linearly independent. Further, by Proposition \ref{v2:1506250740} we have
 $\#\vect{\Gamma}^{(2\vect{n})}_{\vect{\kappa}}= \#\I^{(2\vect{n})}_{\vect{\kappa}}$. This implies that the functions $\dchi^{(2\vect{n})}_{\indexvectgamma}$, $\vectgamma\in \vect{\Gamma}^{(2\vect{n})}_{\vect{\kappa}}$, form a  basis of the vector space $(\mathcal{L}(\I^{(2\vect{n})}_{\vect{\kappa}}),\langle\,\cdot,\cdot\,\rangle_{\omega^{(2\vect{n})}_{\vect{\kappa}}})$.
\end{proof}

\subsection{Polynomial interpolation}

As in \eqref{v2:1508201411}, we have for all $\vectgamma\in\mathbb{N}_0^{\mathsf{d}}$ and $\vect{i}\in \I^{(2\vect{n})}_{\vect{\kappa}}$ the relation
\[T_{\indexvectgamma}(\vect{z}^{(2\vect{n})}_{\vect{i}}) = \dchi^{(2\vect{n})}_{\indexvectgamma}(\vect{i})\]
between the $\mathsf{d}$-variate Chebyshev polynomials $T_{\indexvectgamma}$ and the functions $\dchi^{(2\vect{n})}_{\indexvectgamma} \in \mathcal{L}(\I^{(2\vect{n})}_{\vect{\kappa}})$. We want to find a polynomial $P^{(2\vect{n})}_{\vect{\kappa},h}$ that for given data values $h(\vect{i}) \in \mathbb{R}$, $\vect{i} \in \I^{(2\vect{n})}_{\vect{\kappa}}$,  satisfies the interpolation condition
\begin{equation} \label{v2:1509181420}
 P^{(2\vect{n})}_{\vect{\kappa},h} (\vect{z}^{(2\vect{n})}_{\vect{i}}) = h({\vect{i}}) \quad \text{for all}\quad \vect{i} \in \I^{(2\vect{n})}_{\vect{\kappa}}.
\end{equation}
We use 
\[\Pi^{(2\vect{n})}_{\vect{\kappa}} = \vspan \left\{\, T_{\indexvectgamma}\,\left|\, \vectgamma \in \vect{\Gamma}^{(2\vect{n})}_{\vect{\kappa}} \right. \right\}\]
as underlying space for this interpolation problem. 
The set $ \{\, T_{\indexvectgamma}\,|\, \vectgamma \in \vect{\Gamma}^{(2\vect{n})}_{\vect{\kappa}} \}$ is an orthogonal basis of the space $\Pi^{(2\vect{n})}_{\vect{\kappa}}$ with respect to the inner product given in \eqref{v2:1508220014}. 

\medskip

For $\vect{i} \in \I^{(2\vect{n})}_{\vect{\kappa}}$, we define on $[-1,1]^{\mathsf{d}}$ the polynomial functions $L^{(2\vect{n})}_{\vect{\kappa},\vect{i}}$ by 
\[  L^{(2\vect{n})}_{\vect{\kappa},\vect{i}}(\vect{x}) = \mathfrak{w}^{(2\vect{n})}_{\vect{\kappa},\vect{i}} \Bigl( \ 
 \tsum_{\substack{\indexvectgamma \in \vect{\Gamma}^{(2\vect{n})}_{\vect{\kappa}}}}\!\!\dfrac{2^{-\mathfrak{f}^{(2\vect{n})}(\indexvectgamma)}}{\|T_{\indexvectgamma}\|^2}  T_{\indexvectgamma}(\vect{z}_{\vect{i}}^{(2\vect{n})})\,T_{\indexvectgamma}(\vect{x}) \ - \ \;\! T_{2n_{\mathsf{d}}}(z^{(2n_{\mathsf{d}})}_{i_{\mathsf{d}}})\,T_{2n_{\mathsf{d}}}(x_{\mathsf{d}})\Bigr)
\]
and obtain the following result for polynomial interpolation on the point set $\LC^{(2\vect{n})}_{\vect{\kappa}}$.
The proof follows the lines of the proof of Theorem \ref{v2:1509011746}.

\begin{theorem} \label{v2:201510121554} 
For $h\in\mathcal{L}(\I^{(2\vect{n})}_{\vect{\kappa}})$, the interpolation problem \eqref{v2:1509181420} has the  uniquely determined solution  
\[ P^{(2\vect{n})}_{\vect{\kappa},h} = \sum_{\vect{i} \in \I^{(2\vect{n})}_{\vect{\kappa}}} h({\vect{i}}) L^{(2\vect{n})}_{\vect{\kappa},\vect{i}}\]
in the polynomial space $\Pi^{(2\vect{n})}_{\vect{\kappa}}$. Moreover, we have $\vspan \{\,P^{(2\vect{n})}_{\vect{\kappa},h}\,|\,h\in \mathcal{L}(\I^{(2\vect{n})}_{\vect{\kappa}})\,\}=\Pi^{(2\vect{n})}_{\vect{\kappa}}$ and the polynomials
$L^{(2\vect{n})}_{\vect{\kappa},\vect{i}}$, $\vect{i}\in\I^{(2\vect{n})}_{\vect{\kappa}}$, form a basis of  $\Pi^{(2\vect{n})}_{\vect{\kappa}}$. 
\end{theorem}

\medskip

\noindent Similar as in \eqref{v2:1509241909B}, we consider the  uniquely determined coefficients  in  the expansion
\begin{equation}\label{v2:1509241909}
P^{(2\vect{n})}_{\vect{\kappa},h} = \sum_{\indexvectgamma \in \vect{\Gamma}^{(2\vect{n})}_{\vect{\kappa}}} c_{\indexvectgamma}(h)\;\!T_{\indexvectgamma}
\end{equation}
of the interpolating polynomial $P^{(2\vect{n})}_{\vect{\kappa},h}$ in the orthogonal basis $T_{\indexvectgamma}$, $\vectgamma \in \vect{\Gamma}_{\vect{\kappa}}^{(2\vect{n})}$, of the
polynomial spaces $\Pi^{(2\vect{n})}_{\vect{\kappa}}$.   Using Theorem \ref{v2:thm:210510081555}, we can conclude the following result.

\begin{corollary}
For $h\in\mathcal{L}(\I^{(2\vect{n})}_{\vect{\kappa}})$, the  coefficients $c_{\indexvectgamma}(h)$ in  \eqref{v2:1509241909} are given by
\begin{equation} \label{v2:1510301651}
c_{\indexvectgamma}(h)=\dfrac1{\|\dchi^{(2\vect{n})}_{\indexvectgamma}\|_{\omega^{(2\vect{n})}_{\vect{\kappa}}}^2}\langle\;\! h,\dchi^{(2\vect{n})}_{\indexvectgamma}\rangle_{\omega^{(2\vect{n})}_{\vect{\kappa}}}.
\end{equation}
\end{corollary}
Similar as described at the end of Section \ref{v2:sec:polyinter1649}, also formula \eqref{v2:1510301651} can be performed
efficiently using discrete cosine transforms. We consider the index set 
\[ \J^{(2\vect{n})}= \bigtimes_{\substack{\vspace{-8pt}\\\mathsf{i}=1}}^{\substack{\mathsf{d}\\\vspace{-10pt}}} \{0,\ldots,2n_{\mathsf{i}}\},\]
as well as
\[ g^{(2 \vect{n})}_{\vect{\kappa}}(\vect{i}) = \left\{ \begin{array}{cl} \mathfrak{w}^{(2\vect{n})}_{\vect{\kappa},\vect{i}} h(\vect{i}), \quad  & \text{if}\; \vect{i}\in\I^{(2\vect{n})}_{\vect{\kappa}},\rule[-0.65em]{0pt}{1em}\\ 
       0, \quad 
  & \text{if}\; \vect{i}\in\J^{(2\vect{n})} \setminus \I^{(2\vect{n})}_{\vect{\kappa}}, \end{array} \right. \]
and, recursively for $\mathsf{j}=1,\ldots,\mathsf{d}$,
\[g^{(2 \vect{n})}_{\vect{\kappa},(\gamma_1,\ldots,\gamma_{\mathsf{j}})}(i_{\mathsf{j}+1},\ldots,i_{\mathsf{d}})=\sum_{i_{\mathsf{j}}=0}^{2n_{\mathsf{j}}} g^{(2 \vect{n})}_{\vect{\kappa},(\gamma_1,\ldots,\gamma_{\mathsf{j}-1})}(i_{\mathsf{j}},\ldots,i_{\mathsf{d}})\cos(\gamma_{\mathsf{j}}i_{\mathsf{j}}\pi/(2n_{\mathsf{j}})).\]
Then, since \eqref{v2:A1508221825}, we have
\[c_{\indexvectgamma}(h)=
 \left\{ \begin{array}{rl}  2^{\mathfrak{e}(\indexvectgamma)-\mathfrak{f}^{(2\vect{n})}(\indexvectgamma)}g^{(2 \vect{n})}_{\vect{\kappa},\indexvectgamma},\; & \text{if}\quad 
 \vectgamma \in\vect{\Gamma}^{(2\vect{n})}_{\vect{0}}\setminus\{(0,\ldots,0,2 n_{\mathsf{d}})\},\\[0.3em]
   g^{(2 \vect{n})}_{\vect{\kappa},\indexvectgamma},\; & \text{if}\quad \vectgamma =(0,\ldots,0,2 n_{\mathsf{d}}).
\end{array}\right.\]
As a composition of $\mathsf{d}$ fast cosine transforms, the computation of the set of coefficients $c_{\indexvectgamma}(h)$ can be executed with complexity $\mathcal{O}\!\left(\!\!\;\p[2\vect{n}] \ln \!\!\;\p[2\vect{n}] \right)$. 
Once the coefficients $c_{\indexvectgamma}(h)$ are computed, the interpolating polynomial $P^{(2\vect{n})}_{\vect{\kappa},h}(\vect{x})$ can be evaluated at $\vect{x} \in [-1,1]^{\mathsf{d}}$ using formula \eqref{v2:1509241909}. 

\medskip

Finally, we obtain in analogy to Theorem \ref{v2:1509011747} the following quadrature rule.

\begin{theorem} 
Let $P$ be  a  $\mathsf{d}$-variate polynomial function $[-1,1]^{\mathsf{d}}\to\mathbb{C}$. If
 \[ \text{$\langle P, T_{\indexvectgamma} \rangle = 0$ for all $\vectgamma \in\mathbb{N}_0^{\mathsf{d}} \setminus \{ \vect{0} \}$,  satisfying \eqref{v2:1507201241},}\] then 
\[\frac{1}{\pi^{\mathsf{d}}} \int_{[-1,1]^{\mathsf{d}}} P(\vect{x}) w(\vect{x}) \,\mathrm{d}\vect{x} = \sum_{\vect{i} \in \I^{(2\vect{n})}_{\vect{\kappa}}} \mathfrak{w}^{(2\vect{n})}_{\vect{\kappa},\vect{i}} P(\vect{z}^{(2 \vect{n})}_{\vect{i}}). \]
\end{theorem}


\end{document}